\documentclass[10pt, reqno]{amsart}
\usepackage{amsfonts,latexsym,enumerate}
\usepackage{amsmath}
\usepackage{amscd}
\usepackage{float,amsmath,amssymb,mathrsfs,bm,multirow,graphics}
\usepackage[dvips]{graphicx}
\usepackage[percent]{overpic}
\usepackage[pdftex]{color}
\usepackage{amsaddr}
\usepackage[numbers,sort&compress]{natbib}
\usepackage{mathtools}

\usepackage[pdfdisplaydoctitle=true,
            colorlinks=true,
            urlcolor=blue,
            citecolor=blue,
            linkcolor=blue,
            pdfstartview=FitH,
            pdfpagemode= UseNone,
            bookmarksnumbered=true]{hyperref}

\newtheorem{theorem}{Theorem}
\newtheorem{proposition}{Proposition}[section]
\newtheorem{corollary}[proposition]{Corollary}
\newtheorem{lemma}[proposition]{Lemma}
\newtheorem{definition}[proposition]{Definition}
\newtheorem{assumption}[proposition]{Assumption}
\newtheorem{remark}[proposition]{Remark}

\newtheorem{RHproblem}[proposition]{RH problem}

\numberwithin{equation}{section}

\newcommand{\CC}{\mathbb{C}}

\newcommand{\RR}{\mathbb{R}}
\newcommand{\ZZ}{\mathbb{Z}}

\newcommand{\norm}[1]{\left\Vert#1\right\Vert}
\newcommand{\abs}[1]{\left\vert#1\right\vert}

\DeclareMathOperator{\sech}{sech} %

\newcommand{\Res}{\text{\upshape Res} \,}

\newcommand{\im}{\mathrm{Im}\,}
\newcommand{\re}{\mathrm{Re}\,}
\newcommand{\diag}{\mathrm{diag}}

\newcommand{\reg}{\mathrm{reg}}

\title[The Robin problem for the nonlinear Schr\"odinger equation ]{The Robin problem for the nonlinear Schr\"odinger equation 
on the half-line}
\author{Jae Min Lee}
\address{School of Mathematics and Statistics, \\
Nanjing University of Information Science and Technology, 210044 Nanjing, China}

\author{Jonatan Lenells}

\address{Department of Mathematics, KTH Royal Institute of Technology, \\ 100 44 Stockholm, Sweden.} 

\email{100089@nuist.edu.cn}
\email{jlenells@kth.se}

\thanks{Lee acknowledges support from the European Research Council, Grant Agreement No. 682537 and the G\"{o}ran Gustafsson Foundation. Lenells acknowledges support from the European Research Council, Grant Agreement No. 682537, the Swedish Research Council, Grant No. 2015-05430, the G\"{o}ran Gustafsson Foundation, and the Ruth and Nils-Erik Stenb\"ack Foundation.}%

\begin{document}
\begin{abstract}
We consider the nonlinear Schr\"{o}dinger equation on the half-line $x \geq 0$ with a Robin boundary condition at $x = 0$ and with initial data in the weighted Sobolev space $H^{1,1}(\RR_+)$. We prove that there exists a global weak solution of this initial-boundary value problem and provide a representation for the solution in terms of the solution of a Riemann--Hilbert problem. Using this representation, we obtain asymptotic formulas for the long-time behavior of the solution. In particular, by restricting our asymptotic result to solutions whose initial data are close to the initial profile of the stationary one-soliton, we obtain results on the asymptotic stability of the stationary one-solitons under any small perturbation in $H^{1,1}(\RR_+)$. In the focusing case, such a result was already established by Deift and Park using different methods, and our work provides an alternative approach to obtain such results.
We treat both the focusing and the defocusing versions of the equation.
\end{abstract}

\maketitle




\section{Introduction}
The nonlinear Schr\"{o}dinger (NLS) equation
\begin{equation}\label{NLS}
i u_t+u_{xx}-2 \lambda |u|^2u=0, \qquad \lambda = \pm 1,
\end{equation}
where $\lambda=-1$ and $\lambda=1$ correspond to the focusing and defocusing versions of the equation, respectively, is one of the most important evolution equations in mathematical physics. It arises in a variety of situations such as the modeling of slowly varying wave packets in nonlinear media \cite{BN1967}, deep water waves \cite{ZP1983}, plasma physics \cite{Z1972}, nonlinear fiber optics \cite{HT1973, A2007}, and magneto-static spin waves \cite{ZP1983}.

In addition to its physical significance, equation \eqref{NLS} has rich mathematical properties stemming from its complete integrability. In the seminar paper \cite{ZS1972}, Zakharov and Shabat introduced a Lax pair for (\ref{NLS}) and implemented the inverse scattering transform (IST) for the solution of the initial value problem (IVP), see also \cite{ZS1974, AKNS1974}. The NLS equation admits a bi-Hamiltonian structure and infinitely many conservation laws. In the focusing case, it admits soliton solutions that decay exponentially as $x \to \pm \infty$. Questions related to the well-posedness of initial value problems and initial-boundary value problems (IBVPs) for (\ref{NLS}) have been studied extensively, see e.g. \cite{B1999, C2003, S1989, FHM2017, BSZ2018}.

In this paper, we are interested in the NLS equation (\ref{NLS}) posed on the half-line $x \geq 0$. More precisely, we study solutions of \eqref{NLS} in the domain $\{(x,t) \in \RR^2 : x\geq 0, t \geq 0\}$ which satisfy a Robin boundary condition at $x = 0$ of the form
\begin{align}\label{Robincondition}
u_x(0,t)+q u(0,t)=0, \qquad t \ge 0,
\end{align}
where $q \in \RR$ is a real parameter. In the case of $\lambda = -1$ (i.e., in the focusing case), this Robin problem was extensively analyzed in \cite{DP2011} by Deift and Park. In \cite{DP2011}, a B\"acklund transformation is first utilized to extend the half-line solution to a solution on the whole line, and then IST and Riemann--Hilbert (RH) techniques are used to obtain long-time asymptotic formulas for the solution $u(x,t)$ whenever the initial data $u_0(x) = u(x,0)$ is a small perturbation of the stationary one-soliton initial profile. 

The main goal of this work is to study the above Robin problem in the case of $\lambda = 1$, i.e., for the defocusing version of (\ref{NLS}). However, since our methods allow us to treat both the focusing and defocusing cases simultaneously with little additional effort, we will also consider the focusing case. Our work has partly been inspired by \cite{D2017} where P. Deift listed the study of the defocusing case as an interesting open problem. 

\subsection{Description of main results}
We assume that the initial data $u_0(x) \coloneqq u(x,0)$ is either in the Schwartz class $\mathcal{S}(\RR_+)$ of rapidly decaying functions or in the weighted Sobolev space $H^{1,1}(\RR_+)$ defined by
$$H^{1,1}(\RR_+)\coloneqq \{f \in L^2(\RR_+) : \partial_x f , x f \in L^2(\RR_+)\}.$$
In the latter case, we will need to formulate the IBVP in an appropriate weak sense (see Definition \ref{H11globalsoln}). As a first step, we prove that the Robin IBVP is globally well-posed  in $H^{1,1}(\RR_+)$ (see Proposition \ref{H11existenceprop}). The proof of this fact follows along the same lines as the proof of Theorem 3 in \cite{DP2011}. 
Our first main result is Theorem \ref{linearizableth}, which provides a RH representation for the solution $u(x,t)$ of the Robin IBVP for initial data in $H^{1,1}(\RR_+)$. In the focusing case, this theorem is stated under the assumption that the initial data is generic (see Assumption \ref{zerosassumption}), and in the defocusing case with $q > 0$ (where $q$ is the parameter in (\ref{Robincondition})), it is stated under the assumption that the associated RH problem has a solution. Our results are the most complete in the defocusing case with $q < 0$; in this case, no additional assumptions are required which means that the representation applies for any initial data $u_0 \in H^{1,1}(\RR_+)$. 

The representation formula of Theorem \ref{linearizableth} for the half-line solution is expressed in terms of a RH problem which has the same form as the RH problem associated to the NLS equation on the line. This means that we can obtain asymptotic theorems for the Robin IBVP by applying results developed for the pure IVP. In the defocusing case, this leads to Theorem  \ref{nosolitonsasymptoticsth} and Theorem \ref{onesolitonasymptoticsth} which provide the large $t$ asymptotics of the solution of the Robin problem when $q<0$ and $q > 0$, respectively. In the focusing case, we obtain Theorem \ref{focusingasymptoticsth} which provides the asymptotics for any generic initial data in $H^{1,1}(\RR_+)$. 

To describe our remaining two main theorems, Theorem \ref{asymptoticsthdefocusing} and Theorem \ref{asymptoticsthfocusing}, we first need to discuss the stationary one-soliton solutions of (\ref{NLS}). It is important to note that even though the defocusing NLS does not admit soliton solutions on the line with decay as $x \to \pm \infty$, it does admit a family of stationary one-solitons on the half-line. Indeed, there are two-parameter families of stationary one-soliton solutions of \eqref{NLS} on the half-line in both the focusing and defocusing cases which satisfy the Robin boundary condition (\ref{Robincondition}). These are given explicitly by (see \cite{L2015} for the defocusing case)
\begin{equation}\label{usdef}
u_s(x,t) \coloneqq e^{i\omega t}u_{s0}(x),
\end{equation}
where the initial data $u_{s0}(x)$ is given by
\begin{equation}\label{exsol}
u_{s0}(x)\coloneqq\begin{cases}
\sqrt{\omega} \sech\left(\sqrt{\omega} x+\phi\right), & \quad \lambda=-1, \\
\frac{2\alpha\sqrt{\omega}(\sqrt{\alpha^2+\omega}+\sqrt{\omega})e^{x\sqrt{\omega}}}{\alpha^2 (e^{2x\sqrt{\omega}}-1)+2\sqrt{\omega}(\sqrt{\alpha^2+\omega}+\sqrt{\omega})e^{2x\sqrt{\omega}}}, & \quad \lambda=1,
\end{cases}
\end{equation}
and the family of solutions is parametrized by $\omega > 0$ and $\phi \in \RR$ in the focusing case, and by $\omega > 0$ and $\alpha > 0$ in the defocusing case.
Even though $u_s(x,t)$ is singular at
\begin{align}\label{singularxvalue}
x = -\frac{1}{\sqrt{\omega }}\log \Big(\frac{\sqrt{\omega }  + \sqrt{\alpha^2+\omega }}{\alpha}\Big) < 0
\end{align}
when $\lambda = 1$, it is smooth in the quarter plane $\{x \geq 0, t \geq 0\}$.
We will use a subscript $s$ to denote quantities corresponding to the stationary one-soliton $u_s(x,t)$.
The parameter $q_s \coloneqq -u_{s0}'(0)/u_{s0}(0)$ relevant for the Robin boundary condition is given by $q_s = \sqrt{\omega} \tanh(\phi)$ in the focusing case, and by $q_s = \sqrt{\alpha^2+\omega} > 0$ in the defocusing case. In both cases, the Dirichlet and Neumann boundary values are time-periodic and satisfy the Robin boundary condition (\ref{Robincondition}):
$$u_s(0,t)=\alpha e^{i t \omega},\quad u_{sx}(0,t)=-\alpha q_s e^{i \omega t},\qquad t \ge 0,$$
where $\alpha = \sqrt{\omega}\sech\phi > 0$ if $\lambda = -1$. Theorem \ref{asymptoticsthdefocusing} and Theorem \ref{asymptoticsthfocusing} are concerned with small perturbations of the stationary one-soliton $u_s$. More precisely, they provide the long-time asymptotics for the solution of the Robin problem whenever the initial data $u_0$ is a small perturbation of $u_{s0}$ in $H^{1,1}(\RR_+)$. In the defocusing case, the asymptotics of $u(x,t)$ is given to leading order by a stationary one-soliton with slightly perturbed parameters. In the focusing case, the asymptotics of $u(x,t)$ is given to leading order by a stationary one-soliton or a stationary two-soliton depending on the distribution of the zeros of an associated spectral function. (The result of Theorem \ref{asymptoticsthfocusing} was already obtained in \cite{DP2011}, but we include it for completeness as mentioned above.)

\subsection{Methods}
It is not trivial to generalize the approach of \cite{DP2011} to the defocusing case, because when $\lambda = 1$ the B\"acklund transformation extending the solution from $x> 0$ to $x<0$ introduces certain singularities which must be controlled \cite{D2017}. This is one reason why we have decided to adopt a different approach. Our approach is based on a mix of ideas from the Unified Transform Method (UTM) of Fokas as well as ideas related to scattering problems in the context of weighted Sobolev spaces developed in \cite{Z1998, DZ2003, DP2011}. We also use continuity arguments to go from Schwartz class solutions to solutions in the weighted Sobolev space $H^{1,1}(\RR_+)$.

The UTM was introduced by Fokas in 1997 as a general approach to the solution of IBVPs for integrable PDEs \cite{F1997}. Just like the IST, it provides a way to express the solution in terms of the solution of a RH problem. Hence, the UTM can be thought of as an IBVP analog of the IST formalism. The UTM has proven successful for the study of many integrable nonlinear PDEs, see e.g., \cite{F2002, BFS2004, FI1996, HL2018, FIS2005, FL2009, IS2013, L2008, XF2013, FL2021}. In the case of the NLS equation on the half-line, the UTM represents the solution $u(x,t)$ in terms of the solution of a $2 \times 2$-matrix RH problem with a jump along the contour $\RR \cup i\RR$ \cite{F2002}. This RH problem is expressed in terms of four spectral functions $a(k)$, $b(k)$, $A(k)$, and $B(k)$, where $a(k)$ and $b(k)$ are defined in terms of the initial data, while $A(k)$ and $B(k)$ are defined in terms of the boundary values $u(0,t)$ and $u_x(0,t)$. Since both $u(0,t)$ and $u_x(0,t)$ cannot be independently specified for a well-posed problem, $A(k)$ and $B(k)$ remain unknown and the solution formula is therefore not fully effective in general. However, for certain boundary conditions, such as the Robin boundary condition \eqref{Robincondition}, the functions $A(k)$, $B(k)$ can be eliminated in terms of $a(k)$, $b(k)$ by algebraic manipulations involving the so-called global relation. Boundary conditions of this type are known as linearizable \cite{F2002}. 

It is well-known that the Robin boundary condition \eqref{Robincondition} is linearizable and the associated solution formula for $u(x,t)$ has been obtained in \cite{F2002} (see also \cite{FIS2005}). In \cite{F2002, FIS2005}, the formula for $u(x,t)$ was derived under the a priori assumption that the boundary values decay for large times. This is not always the case and the problem was therefore revisited by Its and Shepelsky in \cite{IS2013} who presented an independent proof that the Riemann-Hilbert problem for linearizable boundary conditions indeed yields the solution of the IBVP in question when $\lambda = -1$. 

The present work extends the results of \cite{F2002, FIS2005, IS2013} in two ways. First, we extend the construction of \cite{IS2013} to the case of $\lambda = 1$. The second extension is more significant: Whereas \cite{F2002, FIS2005, IS2013} consider solutions of sufficient smoothness and decay, by injecting ideas related to weighted Sobolev spaces developed in \cite{Z1998, DZ2003, DP2011} as well as a continuity argument, we are able to derive a representation formula for solutions in the weighted Sobolev space $H^{1,1}(\RR_+)$. To the best of our knowledge, this is the first time the half-line approach of \cite{F2002} is developed in such a weighted Sobolev space.

Once the RH representation for the solution has been obtained, we obtain formulas for the long-time asymptotics by means of the nonlinear steepest descent method.
This method was pioneered by Deift and Zhou \cite{DZ1993} in 1993 and has subsequently been successfully employed to derive asymptotics for the solution of IVPs for a large number of integrable PDEs, see e.g., \cite{BM2016, BV2007, DKKZ1996, DP2011, DVZ1994, GT2009, HXF2015, K2008, KMM2003, KT2009, XFC2013}. Asymptotic formulas for solutions of IBVPs have also been obtained by combining the UTM with nonlinear steepest descent techniques, see e.g., \cite{BFS2004, BIK2009, BKSZ2010, BS2009, L2016, HL2018, FIS2005, AL2017dNLS}. For a brief introduction to the Deift--Zhou method, we refer to \cite{DIZ1993}.

In recent years, the error terms in the asymptotic formulas for the solution of NLS on the line with initial data in $H^{1,1}$ have been sharpened by Borghese--Jenkins--McLaughlin \cite{BJM2018} and Dieng--McLaughlin--Miller \cite{DMM2019} using the $\overline{\partial}$ generalization of the nonlinear steepest descent method. 
By applying the results of \cite{BJM2018, DMM2019} to the RH representation of Theorem \ref{linearizableth}, we can immediately obtain asymptotic theorems for the Robin problem, thus we do not have to repeat the steepest descent analysis here. However, in the case considered in Theorem \ref{onesolitonasymptoticsth} where $\lambda = 1$ and $q > 0$, an additional argument is needed. Indeed, in this case the RH problem has poles even if we are considering the defocusing NLS. Such a situation is never encountered for the problem on the line because the defocusing NLS does not admit solitons with decay at spatial infinity. We deal with this issue by relating the singular RH solution to a regularized solution without poles.

We believe that the approach to the Robin problem presented here combined with the $\overline{\partial}$ nonlinear steepest descent method will prove useful also for the analysis of other IBVPs for integrable PDEs with data in weighted Sobolev spaces, yielding asymptotic formulas valid for initial data in spaces such as $H^{1,1}(\RR_+)$. 
It would be interesting in this regard to also consider boundary conditions which are not linearizable---this would require an analysis in weighted Sobolev spaces also of the spectral problem associated with the boundary data   and would present further challenges.

\subsection{Organization of the paper}
The main theorems of the paper are stated in Section \ref{mainresult}.
In Section \ref{spectralfunctionssec}, we perform the spectral analysis and establish several properties of the associated spectral functions for the half-line problem; in particular, we establish their continuous dependence on the initial data $u_0 \in H^{1,1}(\RR_+)$. In Section \ref{mfromusec}, we apply the UTM to obtain a representation for the solution of the Robin IBVP in terms of a RH problem under the a priori assumption that the solution exists and lies in the Schwartz class. Theorem \ref{linearizableth} is then first proved for initial data in the Schwartz class in Section \ref{Schwartzsec}. In Section \ref{linearizableproofsec}, the proof of Theorem \ref{linearizableth} is extended to initial data in $H^{1,1}(\RR_+)$ by means of density and continuity arguments. In Section \ref{asymproofsection}, we use the $\overline{\partial}$ nonlinear steepest descent method to prove the long-time asymptotics theorems.
Appendix \ref{globalwpn} contains a proof of the global well-posedness of the Robin problem in $H^{1,1}(\RR_+)$. In Appendix \ref{RHsolitonsec}, we recall various properties of the NLS stationary one-soliton solutions and the associated RH problems.

\subsection{Notation}

The following notation will be used throughout the article.

\begin{enumerate}[$-$]
\item $C>0$ will denote a generic constant that may change within a computation.

\item $[A]_1$ and $[A]_2$ denote the first and second columns of a $2 \times 2$ matrix $A$.

\item If $A$ is $n \times m$ matrix, we define $|A| \ge 0$ by $|A|^2=\Sigma_{i,j}|A_{ij}|^2$. Then $|A+B| \le |A|+|B|$ and $|AB| \le |A| |B|$.

\item For a (piecewise smooth) contour $\gamma \subset \CC$ and $1 \le p \le \infty$, we write $A \in L^p(\gamma)$ if $|A|$ belongs to $L^p(\gamma)$. 
We write $\norm{A}_{L^p(\gamma)}\coloneqq \norm{|A|}_{L^p(\gamma)}$.

\item $\CC_+=\{k \in \mathbb{C}\;|\;\im k > 0\}$ and $\CC_-=\{k \in \mathbb{C}\;|\;\im k < 0\}$ denote the open upper and lower halves of the complex plane.

\item $\RR_+ = [0,\infty)$ and $\RR_- = (-\infty, 0]$ denote the closed right and left half-lines.

\item $f^*(k)\coloneqq \overline{f(\bar{k})}$ denotes the Schwartz conjugate of a function $f(k)$.

\item We let $\{\sigma_j\}_{1}^3$ denote the three Pauli matrices defined by
$$\sigma_1=\begin{pmatrix}0 & 1 \\ 1 & 0 \end{pmatrix},\;\;\; \sigma_2=\begin{pmatrix}0 & -i \\ i & 0 \end{pmatrix},\;\;\; \sigma_3=\begin{pmatrix}1 & 0 \\ 0 & -1 \end{pmatrix}.$$

\item The weighted Sobolev space $H^{k,j}(\RR_+)$ is a Banach space defined by
$$H^{k,j}(\RR_+)\coloneqq \{f \in L^2(\RR_+) : \partial_x^k f , x^j f \in L^2(\RR_+)\},$$
equipped with the norm
$$\|f\|_{H^{k,j}(\RR_+)}\coloneqq \left(\|f\|_{L^2}^2+\|\partial_x^k f\|_{L^2}^2+\|x^j f\|_{L^2}^2\right)^{1/2}.$$
Note that $\|f\|_{L^\infty} \leq C \|f\|_{H^1} \leq C \|f\|_{H^{1,1}}$, so if $f \in H^{1,1}(\RR_+)$, then $f$ is bounded and $\lim_{x \to \infty} f(x) = 0$.

\item The space $\mathcal{S}(\RR_+)$ of Schwartz class functions on $\RR_+$ consists of all smooth functions $f$ on $[0, \infty)$ such that $f$ and all its derivatives have rapid decay as $x \to +\infty$. 

\item For $x \in \RR$, we write $\langle x \rangle \coloneqq \sqrt{1+ x^2}$.

\end{enumerate}

\subsection{Definitions of global solutions}
We next define precisely what we mean by a global solution of the Robin IBVP in the Schwartz class $\mathcal{S}(\RR_+)$ and in the weighted Sobolev space $H^{1,1}(\RR_+)$, respectively.

\begin{definition}\label{Schwartzsolutiondef}\upshape
Let $u_0 \in \mathcal{S}(\RR_+)$. We say that $u(x,t)$ is a {\it global Schwartz class solution of the Robin IBVP for NLS with initial data $u_0$} if 
\begin{enumerate}[$(i)$] 
\item $u(x,t)$ is a smooth complex-valued function of $x \geq 0$ and $t \geq 0$,
\item $u(x,t)$ solves (\ref{NLS}) for $x>0$ and $t>0$,
\item $u(x,0)=u_0(x)$ for $x \ge 0$,	
\item $u_x(0,t)+q u(0,t)=0$ for $t \ge 0$ where $q \coloneqq -u_{0x}(0)/u_0(0)$,
\item $u$ has rapid decay as $x \to +\infty$ in the sense that, for each $N \geq 1$ and each $T > 0$,
$$\sup_{x \geq 0, \, t \in [0, T]} \sum_{j =0}^N (1+|x|)^N |\partial_x^j u| < \infty.$$
\end{enumerate}
\end{definition}

\begin{definition}\label{H11globalsoln}\upshape
Let $q \in \RR$ and $u_0 \in H^{1,1}(\RR_+)$.
We say that $u(t)$ is a {\it global weak solution in $H^{1,1}(\RR_+)$ of the Robin IBVP for NLS with parameter $q$ and initial data $u_0$} if $t \mapsto u(t)$ is a continuous map from $[0,\infty)$ to $H^{1,1}(\RR_+)$ satisfying
\begin{equation}\label{NLSweak}
u(t)=e^{-i H_q^+ t/2}u_0-i\lambda \int_0^t e^{-i H_q^+(t-s)/2}|u(s)|^2 u(s) ds
\end{equation}
with $u(0)=u_0 \in H^{1,1}(\RR_+)$. Here, $H_q^+$ is the self-adjoint operator $-\frac{d^2}{dx^2}$ on $\RR_+$ with domain
\begin{align*}
D(H_q^+)=\{f \in L^2(\RR_+)\;|\;&\text{$f$ and $f'$ are absolutely continuous, $f'' \in L^2(\RR_+)$,}\\
&\text{$f'(0+)+q f(0)=0$}\}.
\end{align*}
\end{definition}

\section{Main results}\label{mainresult}
The statements of our main theorems involve four spectral functions $a(k)$, $b(k)$, $r(k)$, and $\Delta(k)$ which are defined as follows. Suppose that $u_0 \in H^{1,1}(\RR_+)$ and let $\mu(x,k)$ denote the unique solution of the linear Volterra integral equation
\begin{align}\label{muvolterra}
\mu(x,k) = I - \int_x^\infty e^{-ik(x-x')\hat{\sigma}_3}[\mathsf{U}_0(x') \mu(x',k)] dx', \qquad x \geq 0,
\end{align}
where 
\begin{align*}
\mathsf{U}_0(x)=\begin{pmatrix}
0 & u_0(x)\\ 
\lambda \overline{u_0(x)} & 0
\end{pmatrix}
\end{align*}
and $\hat{\sigma}_3$ acts on a $2 \times 2$ matrix $A$ by $\hat{\sigma}_3A = [\sigma_3, A]$, i.e., $e^{\hat{\sigma}_3}A = e^{\sigma_3} A e^{-\sigma_3}$. 
Define the functions $a(k)$ and $b(k)$ by
\begin{align}\label{abdef}
\begin{pmatrix}
b(k) \\ a(k)
\end{pmatrix}\coloneqq [\mu(0,k)]_2, \qquad \im k \geq 0,
\end{align}
and define the functions $r(k)$ and $\Delta(k)$ by
\begin{align}\label{rdef}
& r(k) := \frac{(2k-iq)\overline{b(k)a(-k)}+(2k+iq)\overline{a(k)b(-k)}}{\Delta(k)}, \qquad k \in \RR,
	\\ \label{Deltadef}
& \Delta(k) := (2k-iq)a(k)\overline{a(-\bar{k})}+\lambda(2k+iq)b(k)\overline{b(-\bar{k})}, \qquad \im k \geq 0.
\end{align}
where $q$ is the constant appearing in the Robin boundary condition (\ref{Robincondition}). 
The following proposition, whose proof is given in Section \ref{spectralfunctionssec}, establishes several properties of $r$ and $\Delta$. In particular, it shows that $r(k)$ lies in $H^{1,1}(\RR)$ and depends continuously on $u_0$ whenever $\Delta$ has no real zeros.

\begin{proposition}[Properties of the spectral functions]\label{rDeltaprop}
Given $u_0 \in H^{1,1}(\RR_+)$ and $q \in \RR$, define $a(k)$, $r(k)$, and $\Delta(k)$ by (\ref{abdef})--(\ref{Deltadef}). Then the following hold:

\begin{enumerate}[$(i)$]
\item\label{rDeltapropitem1}
 $\Delta(k)$ is continuous for $\im k \geq 0$ and analytic for $\im k > 0$. 

\item $\Delta(k)$ obeys the symmetry
\begin{equation}\label{Deltasymm}
\overline{\Delta(-\overline{k})}=-\Delta(k), \qquad \im k \geq 0.
\end{equation}

\item $\Delta(0)=-iq$.

\item\label{rDeltapropitem4}
 As $k \to \infty$, $\Delta(k)$ satisfies
\begin{align}\label{Deltalargek}
\Delta(k) = 2k + O(1), \qquad k \to \infty, \; \im k \geq 0.
\end{align}

\item \label{rDeltapropitem5}
If $\Delta(k) \neq 0$ for all $k \in \RR$, then $r \in H^{1,1}(\RR)$.
Moreover, the map $u_0 \mapsto r:H^{1,1}(\RR_+) \cap \{u_0 \, | \, \text{$\Delta(k) \neq 0$ for all $k \in \RR$}\} \to H^{1,1}(\RR)$ is continuous.
\end{enumerate}


In the defocusing case (i.e., in the case when $\lambda = 1$), the following also hold:
\begin{enumerate}[$(a)$]
\item \label{rDeltapropitema}
 $|r(k)| < 1$ for all $k \in \RR$.

\item \label{rDeltapropitemb}
If $q < 0$, then $\Delta$ has no zeros in $\bar{\CC}_+ = \{k \in \mathbb{C}\;|\;\im k \geq 0\}$. 

\item \label{rDeltapropitemc}
If $q > 0$, then $\Delta$ has no zeros on $\RR$ and exactly one simple zero in $\CC_+$. Moreover, this zero is pure imaginary.

\item \label{rDeltapropitemd}
If $q = 0$, then $\Delta(k) \neq 0$ for all $k \in \bar{\CC}_+ \setminus \{0\}$, but $\Delta(0) = 0$.

\item \label{rDeltapropiteme}
 $a(k) \neq 0$ for all $k \in \bar{\CC}_+$.

\end{enumerate}
\end{proposition}

Before stating our main results we also need the following well-posedness result for the Robin IBVP. The proof, which is given in Appendix \ref{globalwpn}, is based on a fixed-point argument and is an easy generalization of the argument presented in \cite[Theorem 3]{DP2011}, where global existence and uniqueness was established in the case of $\lambda = -1$.
Recall that we introduced the notion of a global weak solution in $H^{1,1}(\RR_+)$ in Definition \ref{H11globalsoln}.

\begin{proposition}[Global well-posedness of the Robin problem]\label{H11existenceprop}
Suppose $\lambda = 1$ or $\lambda = -1$. Let $q \in \RR$ and $u_0 \in H^{1,1}(\RR_+)$. Then there exists a unique global weak solution in $H^{1,1}(\RR_+)$ of the Robin IBVP for NLS with parameter $q$ and initial data $u_0$.
Moreover, for each $T > 0$, the data-to-solution mapping $u_0 \mapsto u$ is continuous from $H^{1,1}(\RR_+)$ to $C([0,T],H^{1,1}(\RR_+))$.
\end{proposition}

\begin{remark}
Himonas and Mantzavinos \cite{HM2021} have recently proved local well-posedness of  the NLS equation on the half-line for a more general class of Robin boundary conditions  where the parameter $q$ in (\ref{Robincondition}) is allowed to depend on time. Since we need global (and not just local) well-posedness, we give an independent proof of Proposition \ref{H11existenceprop}. It should be noted that the case when $q$ is constant is much easier to handle than the case considered in \cite{HM2021}.
\end{remark}

\subsection{Representation theorem}
Our first main result (Theorem \ref{linearizableth}) provides a representation for the solution $u(x,t)$ of the Robin IBVP in terms of the solution $m(x,t,k)$ of a RH problem. The formulation of this RH problem depends only on the initial data $u_0$, thus the solution representation is effective. 

In the focusing case, we make the following generic assumptions on the zeros of $a(k)$ and $\Delta(k)$.

\begin{assumption}\label{zerosassumption}
If $\lambda=-1$, we assume the following:
\begin{enumerate}[$(i)$]
\item $\Delta(k) \neq 0$ for every $k \in \RR$ and every zero of $\Delta(k)$ in $\CC_+$ is simple. 

\item If $\Delta(k) = 0$ for some $k \in \CC_+$, then $a(k) \neq 0$.

\end{enumerate}
\end{assumption}

\begin{remark}
In the defocusing case, the statements in Assumption \ref{zerosassumption} are automatically fulfilled for $q\in \RR \setminus \{0\}$ as a consequence of Proposition \ref{rDeltaprop}. 
\end{remark}

In view of (\ref{Deltalargek}), if $\lambda = -1$ and Assumption \ref{zerosassumption} holds, then $\Delta$ has at most finitely many zeros in $\CC_+$. On the other hand, by Proposition \ref{rDeltaprop}, if $\lambda =1$, then $a$ is zero-free and $\Delta$ is non-zero on $\RR$ and has either one or no zero in $\CC_+$ depending on the sign of $q$; more precisely, if $\lambda = 1$ and $q > 0$, then $\Delta$ has one zero in $\CC_+$, whereas if $\lambda = 1$ and $q < 0$, then  $\Delta$ has no zeros in $\CC_+$.
We denote the zeros of $\Delta$ in $\CC_+$ by $\{\xi_j\}_1^M$, $M \geq 0$.

\begin{theorem}[Representation of the solution of the Robin IBVP for NLS]\label{linearizableth}
Suppose $\lambda = 1$ or $\lambda = -1$. Let $q\in \RR \setminus \{0\}$ and $u_0 \in H^{1,1}(\RR_+)$ and define $a(k), b(k), r(k)$, and $\Delta(k)$ by (\ref{abdef})--(\ref{Deltadef}). If $\lambda = -1$, suppose that Assumption \ref{zerosassumption} holds. Consider the following RH problem for $m(x,t,k)$:
\begin{itemize}
\item $m(x,t,\cdot): \mathbb{C} \setminus (\RR \cup \{\xi_j, \bar{\xi}_j\}_1^M) \to \mathbb{C}^{2 \times 2}$ is analytic.
\item The boundary values of $m(x,t,k)$ as $k$ approaches $\RR$ from above $(+)$ and below $(-)$ exist, are continuous on $\RR$, and satisfy
\begin{align}\label{mjump}
  m_+(x,t,k)=&\; m_-(x, t, k) v(x,t,k), \qquad k \in \RR, 
\end{align}
where
\begin{align}\label{vdef}
v(x,t,k) \coloneqq \begin{pmatrix}
1-\lambda |r(k)|^2 & \overline{r(k)}e^{-2i\theta} \\ 
-\lambda r(k)e^{2i\theta}  & 1
\end{pmatrix}, \qquad \theta \equiv \theta(x,t,k) \coloneqq kx+2k^2t.
\end{align}

\item $m(x,t,k) \to I$ as $k \to \infty$.
\item The first column of $m$ has at most simple poles at the zeros $\xi_j \in \CC_+$ of $\Delta(k)$, the second column of $m$ has at most simple poles at the zeros $\bar{\xi}_j \in \CC_-$ of $\overline{\Delta(\bar{k})}$, and the following residue conditions hold for $j = 1, \dots, M$:
\begin{subequations}\label{mresidues}
\begin{align}\label{mresiduesa}
& \underset{k = \xi_j}{\Res} [m(x,t,k)]_1 =  c_j e^{2i\theta(x,t,\xi_j)} [m(x,t,\xi_j)]_2,
	\\\label{mresiduesb}
& \underset{k=\bar{\xi}_j}{\Res} [m(x,t,k)]_2 = \lambda \bar{c}_j e^{-2i\theta(x,t,\bar{\xi}_j)}  [m(x,t,\bar{\xi}_j)]_1,
\end{align}
\end{subequations}
where 
\begin{align}\label{cjdef}
c_j :=-\frac{\lambda \overline{b(-\bar{\xi}_j)}}{a(\xi_j)}\frac{2\xi_j+iq}{\dot{\Delta}(\xi_j)}, \qquad j = 1, \dots, M.
\end{align}

\end{itemize}
If $\lambda = 1$ and $q > 0$, then suppose additionally that the above RH problem has a solution for each $(x,t) \in [0,\infty) \times [0,\infty)$.

Then the above RH problem has a unique solution $m(x,t,k)$ for each $(x,t) \in [0,\infty) \times [0,\infty)$ and the function $u(x,t)$ defined by
\begin{align}\label{recoveru}
u(x,t) = 2i\lim_{k \to \infty} k(m(x,t,k))_{12}, \qquad x \geq 0, ~ t \geq 0,
\end{align}
where the limit is taken nontangentially with respect to $\RR$, is the unique global solution in $H^{1,1}(\RR_+)$ of the Robin IBVP for NLS with parameter $q$ and initial data $u_0$.
\end{theorem}
\begin{proof}
See Section \ref{linearizableproofsec}.
\end{proof}

\subsection{Asymptotic theorems for the defocusing NLS}
Using the RH representation of Theorem \ref{linearizableth}, we can obtain asymptotic formulas for the solution $u(x,t)$ by appealing to known asymptotic results for the NLS equation on the line. Indeed, the form of the RH problem of Theorem \ref{linearizableth} has the exact same form as the RH problem relevant for the NLS equation on the line. In the case when $\lambda = 1$ and $q<0$, this leads to the following result which provides the large $t$ asymptotics of the solution of the Robin problem for any initial data in $H^{1,1}$.

\begin{theorem}[Long-time asymptotics: defocusing NLS, $q<0$]\label{nosolitonsasymptoticsth}
Suppose $\lambda = 1$. Let $q < 0$ and $u_0 \in H^{1,1}(\RR_+)$. 
Then the global solution $u(x,t)$ in $H^{1,1}(\RR_+)$ of the Robin IBVP for NLS with parameter $q$ and initial data $u_0$ satisfies
\begin{equation}\label{udefocusingqnegative}
u(x,t)=\frac{\beta(r(k_0)) (8t)^{i\nu(r(k_0))} e^{2\chi(\zeta, k_0)}e^{4itk_0^2}}{\sqrt{2t}} + O\big(t^{-3/4}\big), \qquad t \to \infty,
\end{equation}
uniformly for $x \in [0, \infty)$, where 
\begin{align}\nonumber
& k_0 \coloneqq -\frac{\zeta}{4}=-\frac{x}{4t},  \quad \nu(y) \coloneqq \frac{1}{2\pi}\ln(1-\lambda|y|^2), \quad \beta(y) \coloneqq \sqrt{|\nu(y)|}e^{i\left(\frac{\pi}{4}-\arg y - \arg \Gamma(i\nu(y))\right)},
	\\ \label{k0nubetachidef}
& \chi(\zeta,k) \coloneqq -\frac{1}{2\pi i}\int_{-\infty}^{k_0}\ln(k-s)d \ln(1-\lambda|r(s)|^2),
\end{align}
and $\Gamma$ is the Gamma function.
\end{theorem}
\begin{proof}
In this case, $\Delta(k)$ has no zeros. The assumptions of Theorem \ref{linearizableth} are fulfilled and the asymptotic formula for $u(x,t)$ coincides with the formula for the long-time asymptotics for the solution of the initial-value problem for the defocusing NLS on the line with reflection coefficient $r(k)$ given by (\ref{rdef}). The derivation of such asymptotics has a long history; the form (\ref{udefocusingqnegative}) of the asymptotic formula with the error term $O(t^{-3/4})$ has recently been obtained in \cite{DMM2019} with the help of the $\overline{\partial}$ steepest descent method.
\end{proof}

We next consider the case when $\lambda = 1$ and $q > 0$. In this case, the statement of Theorem \ref{linearizableth} includes the additional assumption that the RH problem must have a solution for each $(x,t) \in [0,\infty) \times [0,\infty)$. Let us comment on this assumption. For $\lambda = -1$, the existence of a solution of the RH problem of Theorem \ref{linearizableth} for any $(x,t)$ can be established with the help of a vanishing lemma. A similar argument applies also when $\lambda = 1$ and $q < 0$, because in this case $m$ has no poles. However, if $\lambda = 1$ and $q > 0$, then $m$ has poles at $\xi_1$ and $\bar{\xi}_1$ and the construction of $m$ is more complicated. Indeed, if the residue conditions (\ref{mresidues}) are replaced with jumps on small circles in the standard manner, then these jump matrices have the appropriate symmetry properties for the derivation of a vanishing lemma only if $\lambda = -1$. If the residue conditions are alternatively handled with the help of a Darboux transformation and thereby replaced by an algebraic system (see e.g. \cite{FI1996} or Section \ref{asymproofsection}), then the algebraic system is only known to have a solution for all $(x,t)$ if $\lambda = -1$. (This is related to the fact that the stationary one-soliton of the defocusing NLS is singular at the value of $x$ specified in (\ref{singularxvalue}), so the existence of a solution may indeed break down at certain points.)
For these reasons, our next result is stated under the assumption of the existence of a solution of the RH problem.

\begin{theorem}[Long-time asymptotics: defocusing NLS, $q>0$]\label{onesolitonasymptoticsth}
Suppose $\lambda = 1$. Let $q > 0$ and $u_0 \in H^{1,1}(\RR_+)$. Suppose that the RH problem of Theorem \ref{linearizableth} has a solution for each $(x,t) \in [0,\infty) \times [0,\infty)$. Then the global solution $u(x,t)$ in $H^{1,1}(\RR_+)$ of the Robin IBVP for NLS with parameter $q$ and initial data $u_0$ satisfies
\begin{equation}\label{uqpositiveasymptotics}
u(x,t)=u_{\mathrm{sol}}(x,t)+\frac{u_{\mathrm{rad}}^{(1)}(x,t)+u_{\mathrm{rad}}^{(2)}(x,t)}{\sqrt{t}} + O\big(t^{-3/4}\big), \qquad t \to \infty,
\end{equation}
uniformly for $x \in [0, \infty)$, where
\begin{subequations}
\begin{align}\label{usoldefoc}
u_{\mathrm{sol}}(x,t)=&\frac{\lambda 4i\xi_1\overline{d_1}\delta(\zeta,\xi_1)^2}{\lambda |d_1|^2-|\delta(\zeta,\xi_1)|^4},
	\\ \label{urad1defoc}
u_{\mathrm{rad}}^{(1)}(x,t)=&\; \frac{\beta(r_\reg(k_0)) (8t)^{i\nu(r_\reg(k_0))} e^{2\chi(\zeta, k_0)}e^{4itk_0^2}}{\sqrt{2}},
	\\ \nonumber
u_{\mathrm{rad}}^{(2)}(x,t)=&-\frac{\lambda \sqrt{2} \xi_1 \beta(r_\reg(k_0)) e^{-4itk_0^2} \delta_0(\zeta,t)^2 \left(|d_1|^2 (k_0+\xi_1)+\lambda|\delta(\zeta,\xi_1)|^4(k_0-\xi_1)\right)}{|k_0-\xi_1|^2 \left(\lambda |d_1|^2-|\delta(\zeta,\xi_1)|^4\right)} 
	\\ \label{urad2defoc}
&+\frac{\lambda 4\sqrt{2} \xi_1^2 \overline{d_1} \delta(\zeta,\xi_1)^2 \mathrm{Re}\left[d_1 \beta(r_\reg(k_0)) e^{-4itk_0^2} \delta_0(\zeta,t)^2 \overline{\delta(\zeta,\xi_1)}^2\right]}{|k_0-\xi_1|^2 (\lambda |d_1|^2-|\delta(\zeta,\xi_1)|^4)^2}.
\end{align}
\end{subequations}
Here, $\xi_1$ is the simple pole of the reflection coefficient $r(k)$ corresponding to $u_0$. The pole-free reflection coefficient $r_\reg(k)$ is defined by the transformation \eqref{rregdef} and the regular RH solution $m_\reg$ is defined via the Darboux transformation \eqref{Darboux}. The functions $d_1=d_1(x,t)$, $\delta(\zeta, k)$, and $\delta_0(\zeta, t)$ are defined in \eqref{Darbouxconst} and (\ref{deltadelta0Ym1Xdef}).
\end{theorem}
\begin{proof}
See Section \ref{onesolitonasymptoticsthproof}.
\end{proof}

Our next theorem considers the asymptotic stability of the stationary one-soliton.
Recall that $u_{s0}(x) = u_s(x,0)$ denotes the initial data corresponding to the stationary one-soliton solution given in \eqref{usdef}. 
If $u_0$ is sufficiently close to $u_{s0}$, then it can be shown that the RH problem of Theorem \ref{linearizableth} has a solution for all $(x,t)$. As a consequence, we obtain the following theorem which provides the long-time asymptotics for the solution of the Robin problem whenever the initial data $u_0 \in H^{1,1}$ is close to $u_{s0}$. The main takeaway is that the asymptotics of $u(x,t)$ is given by \eqref{uqpositiveasymptotics} and that the ingredients of \eqref{uqpositiveasymptotics} depend continuously on $u_0$ in the $H^{1,1}$-norm.

\begin{theorem}[Long-time asymptotics: defocusing NLS, near stationary one-soliton]\label{asymptoticsthdefocusing}
Suppose $\lambda = 1$. Let $u_{s0}(x)$ be the initial profile (\ref{exsol}) of the stationary one-soliton for some choice of the parameters $\omega > 0$ and $\alpha > 0$. Let $q = -u_{s0}'(0)/u_{s0}(0) = \sqrt{\alpha^2+\omega}$ be the corresponding value of $q$. Let $\xi_{s1} = i\sqrt{\omega}/2$ be the simple zero corresponding to $u_{s0}$, see Appendix \ref{RHsolitonsec}. Then there exists a neighborhood $U$ of $u_{s0}$ in $H^{1,1}(\RR_+)$  such that if $u_0 \in U$, then the corresponding spectral function $\Delta(k)$ has only one simple zero $\xi_1 \in i\RR_+$, and this zero satisfies $\left\vert \xi_1-\xi_{s1} \right\vert \to 0$ as $\left\Vert u_0-u_{s0} \right\Vert_{H^{1,1}(\RR_+)} \to 0$. Moreover, the global weak solution $u(x,t)$ in $H^{1,1}(\RR_+)$ of the IBVP for the NLS equation with parameter $q$ and initial data $u_0$ satisfies \eqref{uqpositiveasymptotics} as $t \to \infty$ uniformly for $x \in [0, \infty)$.
%
\end{theorem}
\begin{proof}
See Section \ref{asymptoticsthdefocusingproof}.
\end{proof}

\begin{remark}
The leading term $u_{\mathrm{sol}}(x,t)$ in \eqref{uqpositiveasymptotics} is a stationary one-soliton solution of the defocusing NLS equation. Indeed, if we write $u_s(x,t;\alpha, \omega)$ for the stationary one-soliton in \eqref{usdef} corresponding to the parameters $\alpha >0$ and $\omega >0$, then
$$u_{\mathrm{sol}}(x,t) = u_s(x,t; \tilde{\alpha}, \tilde{\omega})$$
where $\tilde{\omega} =-4\xi_1^2$ and $\tilde{\alpha} =\frac{2\tilde{\omega}}{\sqrt{\tilde{\omega} -|c_s|^2}}$. Here, $c_s=\frac{c_1}{\delta(\zeta,\xi_1)^2}$ where  $c_1$ is the residue constant in (\ref{cjdef}) corresponding to the initial data $u_0$. As $u_0 \to u_{s0}(\,\cdot\,; \alpha, \omega)$ in $H^{1,1}(\RR_+)$, we have $\xi_1 \to \xi_{s1}$, $r(k) \to 0$, $\delta(\zeta,\xi_1) \to 1$, $\tilde{\omega} \to \omega$, and $\tilde{\alpha} \to \alpha$. Hence, in this limit, $u_{\mathrm{sol}}(x,t) \to u_s(x,t; \alpha, \omega)$, $u_{\mathrm{rad}}^{(1)} \to 0$, and $u_{\mathrm{rad}}^{(2)} \to 0$. It follows that as the initial data $u_0 \in H^{1,1}(\RR_+)$ approaches the stationary one-soliton initial data $u_{s0}(\,\cdot\,; \alpha, \omega)$, the asymptotics formula (\ref{uqpositiveasymptotics}) reduces to $u(x,t)=u_s(x,t; \alpha, \omega)+O(t^{-3/4})$ as expected.
\end{remark}

\subsection{Asymptotic theorems for the focusing NLS}
We now consider the focusing case. We let $\sigma_d \coloneqq \{(\xi_j,c_j)\}_{j=1}^M \subset \CC \times (\CC \setminus \{0\})$ denote the discrete scattering data associated to $u_0$, where $\{\xi_j\}_1^M$ is the set of simple zeros of $\Delta(k)$ and $\{c_j\}_1^M$ is the set of corresponding residue constants defined in (\ref{cjdef}).

The next theorem is a direct consequence of Theorem \ref{linearizableth} and the asymptotics for the focusing NLS on the line.

\begin{theorem}[Long-time asymptotics: focusing NLS, generic initial data]\label{focusingasymptoticsth}
Suppose $\lambda = -1$. Let $q \in \RR \setminus \{0\}$ and $u_0 \in H^{1,1}(\RR_+)$ be such that Assumption \ref{zerosassumption} holds. As $t \to \infty$ the global weak solution $u(x,t)$ in $H^{1,1}(\RR_+)$ of the Robin IBVP for NLS with parameter $q$ and initial data $u_0$ satisfies the asymptotics
\begin{equation}\label{ufocusingasymptotics}
u(x,t)=u_{\mathrm{sol}}(x,t;\hat{\sigma}_d)+\frac{u_{\mathrm{rad}}(x,t)}{\sqrt{2t}} + O\big(t^{-3/4}\big), \qquad t \to \infty,
\end{equation}
uniformly for $\zeta:=x/t$ in $[0,K] \subset [0,\infty)$ for any $K>0$.

\begin{enumerate}[$(1)$]
\item $u_{\mathrm{sol}}(x,t;\hat{\sigma}_d)$, the leading term, is the soliton solution defined by \eqref{recoverusol} corresponding to the modified discrete scattering data $\hat{\sigma}_d$ given by
\begin{align}\nonumber
& \hat{\sigma}_d \coloneqq \{(\xi_j,\hat{c}_j)\;|\;\xi_j \in Z(\mathcal{I})\},\quad \text{where} 
	\\ \label{modifieddataresidue}
& \hat{c}_j \coloneqq c_j\prod_{\xi_l \in Z^-(\mathcal{I})}\left(\frac{\xi_j-\xi_l}{\xi_j-\bar{\xi}_l}\right)^2 \exp\left(-\frac{1}{\pi i} \int_{-\infty}^{k_0}\frac{\log(1-\lambda |r(s)|^2)}{s-\xi_j} ds\right), 
	\\ \nonumber
& \mathcal{I}:=[-K/2,0], \qquad Z:=\{\xi_j\}_{j=1}^M, \qquad k_0 := -\zeta/4,
	\\ \nonumber
&Z(\mathcal{I}) \coloneqq \{\xi_j \in Z\;|\;\re \xi_j <k_0\}, \quad Z^-(\mathcal{I}) \coloneqq \{\xi_j \in Z\;|\;\re \xi_j<-K/2\}.
\end{align}

\item $u_{\mathrm{rad}}(x,t)$, the subleading term, is the asymptotic radiation contribution determined by the reflection coefficient $r(k)$ in \eqref{rdef} and the zeros $\xi_j \in \Box_{k_0}^-(\mathcal{I})$ where 
$$\Box_{k_0}^-(\mathcal{I}):=Z(\mathcal{I}) \setminus Z^-(\mathcal{I})=\{\xi_j \in Z\;|\;-K/2 \le \re \xi_j<k_0\},$$
as follows:
\begin{align}\nonumber
u_{\mathrm{rad}}(x,t)=&\; m_{11}^{\Box}(x,t,k_0)^2 \alpha(r(k_0))e^{ix^2/(4t)+i\nu(r(k_0))\log(8t)}	\\\label{uradformula}
&+ m_{12}^{\Box}(x,t,k_0)^2 \overline{\alpha(r(k_0))}e^{-ix^2/(4t)-i\nu(r(k_0))\log(8t)},
	\\ \nonumber
\alpha(r(k_0)) \coloneqq &\; \beta(r(k_0))\exp\bigg[i\bigg(2\pi \chi(\zeta,k_0)-4\sum_{\xi_j \in \Box_{k_0}^{-}}\arg(k_0-\xi_j)\bigg)\bigg],
\end{align}
where $\beta$ and $\chi$ are given by (\ref{k0nubetachidef}), and the coefficients $m_{11}^{\Box}(x,t,k_0)$ and $m_{12}^{\Box}(x,t,k_0)$ are the entries of the first row of the solution of RH problem \ref{RHsolitonrenorm} with discrete scattering data $\hat{\sigma}_d$ and $\Box=\{j \in \{1,\ldots,M\}\;|\;\xi_j \in \Box_{k_0}^- (\mathcal{I})\}$ evaluated at $k=k_0$.
\end{enumerate}
\end{theorem}
\begin{proof}
The assumptions of Theorem \ref{linearizableth} are fulfilled, so the asymptotics for $u(x,t)$ coincides with the asymptotics for the solution of the initial-value problem for the focusing NLS on the line with reflection coefficient $r(k)$ given by (\ref{rdef}). The strong form (\ref{ufocusingasymptotics}) of the asymptotic formula with the error term $O(t^{-3/4})$ has been obtained in \cite{BJM2018} by means of the $\overline{\partial}$ steepest descent method.
\end{proof}

\begin{remark}\label{referenceremark}
In the case when $Z \subset i\RR_+$, we have $\Box_{k_0}^-(\mathcal{I})=\emptyset$. In particular, this is the case for stationary soliton solutions.
\end{remark}

Our last theorem concerns the asymptotic stability of the stationary one-soliton of the focusing NLS. Let $u_{s0}(x)$ be the initial profile (\ref{exsol}) of the stationary one-soliton for some choice of the parameters $\omega > 0$ and $\phi \in \RR \setminus \{0\}$. Let $q = -u_{s0}'(0)/u_{s0}(0) = \sqrt{\omega} \tanh{\phi}$ be the corresponding value of $q$, and let $a_s(k)$, $b_s(k)$, $r_s(k)$, $\Delta_s(k)$ denote the spectral functions defined in (\ref{abdef})--(\ref{Deltadef}) corresponding to $u_{s0}$. For $\lambda = -1$, we have (see Appendix \ref{RHsolitonsec})
\begin{equation}
\Delta_s(k)=\frac{(2k-i\sqrt{\omega})(2k+iq)}{2k+i\sqrt{\omega}}, \qquad a_s(k)=\frac{2k+ i q}{2k+i\sqrt{\omega}}.
\end{equation}
Hence, if $q>0$, then $\Delta_s(k)$ has one simple zero $\xi_{s1} := i\sqrt{\omega}/2$ in $\CC_+$. On the other hand, if $q<0$, then $\Delta_s(k)$ has two simple zeros $\xi_{s1}:= i\sqrt{\omega}/2$ and $\xi_{s2} := -iq/2$ in $\CC_+$; note that $\xi_{s1}, \xi_{s2} \in i\RR_+$ and $0 < |\xi_{s2}| < |\xi_{s1}|$. Since $a_s(k)$ also has a simple zero at $\xi_{s2}$ in this case, it turns out that the zero $\xi_{s2}$ of $\Delta_s$ does not generate a soliton. 

Now let $u_0$ be a small perturbation of $u_{s0}$. Then one can show that the zeros of $\Delta$ are pure imaginary and close to the zeros of $\Delta_s$. If $q<0$, this means that $\Delta$ has two zeros $\xi_1,\xi_2 \in i\RR_+$ such that $\xi_1 \approx \xi_{s1}$ and $\xi_2 \approx \xi_{s2}$. It turns out that there are two cases to consider:
\begin{enumerate}
\item If $\xi_2=\xi_{s2}$, then $\xi_2$ does not generate a soliton.
\item If $\xi_2 \ne \xi_{s2}$, then both $\xi_1$ and $\xi_2$ generate solitons.
\end{enumerate}
It follows that the leading order term $u_{\mathrm{sol}}(x,t;\hat{\sigma}_d)$ in Theorem \ref{focusingasymptoticsth} is either a stationary one-soliton or a stationary two-soliton. More precisely, we have the following theorem. A result of this type was already obtained in \cite{DP2011}.

\begin{theorem}[Long-time asymptotics: focusing NLS, near stationary one-soliton]\label{asymptoticsthfocusing}
Suppose $\lambda = -1$.  Let $u_{s0}(x)$ be the initial profile (\ref{exsol}) of the stationary one-soliton for some choice of $\omega > 0$ and $\phi \in \RR \setminus \{0\}$. 
Then there exists a neighborhood $U$ of $u_{s0}$ in $H^{1,1}(\RR_+)$  such that the following statements hold whenever $u_0 \in U$:
\begin{enumerate}[$(1)$]
\item If $q>0$, then the spectral function $\Delta(k)$ corresponding to $u_0$ has exactly one zero in $\bar{\CC}_+$, denoted by $\xi_1$. The zero $\xi_1$ is pure imaginary and simple. Moreover, $\xi_1 \to \xi_{s1}$ as $u_0 \to u_{s0}$ in $H^{1,1}(\RR_+)$.

\item If $q<0$, then the spectral function $\Delta(k)$ corresponding to $u_0$ has exactly two zeros in $\bar{\CC}_+$, denoted by $\xi_1$ and $\xi_2$. The zeros $\xi_1$ and $\xi_2$ are pure imaginary and simple. Moreover $\xi_1 \to \xi_{s1}$ and $\xi_2 \to \xi_{s2}$ as $u_0 \to u_{s0}$ in $H^{1,1}(\RR_+)$.

\item As $t \to \infty$, the global weak solution $u(x,t)$ of the IBVP for the NLS equation with parameter $q = \sqrt{\omega} \tanh{\phi}$ and initial data $u_0$ satisfies the asymptotic formula (\ref{ufocusingasymptotics}) uniformly for $\zeta = x/t$ in compact subsets of $[0, \infty)$, where $u_{\mathrm{sol}}(x,t;\hat{\sigma}_d)$ is defined as follows:
\begin{enumerate}[$(a)$]
\item If $\Delta(k)$ has one simple zero $\xi_1$, then $u_{\mathrm{sol}}(x,t;\hat{\sigma}_d)$ is the stationary one-soliton solution with $\hat{\sigma}_d=\{(\xi_1,\hat{c}_1)\}$, where $\hat{c}_1$ is given by (\ref{modifieddataresidue}). 

\item If $\Delta(k)$ has two simple zeros $\xi_1$ and $\xi_2 = \xi_{s2}$, then $u_{\mathrm{sol}}(x,t;\hat{\sigma}_d)$ is the stationary one-soliton solution with $\hat{\sigma}_d=\{(\xi_1,\hat{c}_1)\}$.

\item If $\Delta(k)$ has two simple zeros $\xi_1$ and $\xi_2 \ne \xi_{s2}$, then $u_{\mathrm{sol}}(x,t;\hat{\sigma}_d)$ is the stationary two-soliton solution with $\hat{\sigma}_d=\{(\xi_1,\hat{c}_1),(\xi_2,\hat{c}_2)\}$ where $\hat{c}_1,\hat{c}_2$ are given by \eqref{modifieddataresidue}.
\end{enumerate}
The subleading term $u_{\mathrm{rad}}(x,t)$ in (\ref{ufocusingasymptotics}) is defined by \eqref{uradformula} with $r(k)$ given by \eqref{rdef} and with the modified scattering data $\hat{\sigma}_d$ specified for each case as above.

\end{enumerate}

%
%
%
\end{theorem}
\begin{proof}
See Section \ref{asymptoticsthfocusingproof}.
\end{proof}

\section{Properties of the spectral functions}\label{spectralfunctionssec}
In this section, we establish several properties of the spectral functions $a(k), b(k), r(k)$, and $\Delta(k)$ defined in (\ref{abdef})--(\ref{Deltadef}). In particular, we provide a proof of Proposition \ref{rDeltaprop}. 
Note that our spectral functions are associated to the half-line problem and therefore differ from the spectral functions encountered for the problem on the line.

\subsection{Volterra integral equation and estimates}
Our first lemma proves existence and uniqueness of the solution of the integral equation \eqref{muvolterra} for $\mu(x,k)$.

\begin{lemma}\label{mulemma}
If $u_0 \in H^{1,1}(\RR_+)$, then there is a unique $2 \times 2$-matrix-valued solution $\mu(x,k)$ of equation \eqref{muvolterra} with the following properties:
\begin{enumerate}[$(a)$]
\item The function $\mu(x,k)$ is defined for $x \ge 0$ and $k \in (\bar{\CC}_-,\bar{\CC}_+)$. For each $k \in (\bar{\CC}_-,\bar{\CC}_+)$, the function $\mu(\cdot,k)$ belongs to $C^1([0,\infty))$ and satisfies
\begin{align}\label{xparteqn0}
\mu_x+ik [\sigma_3, \mu] = \mathsf{U}_0 \mu, \qquad x \geq 0.
\end{align}

\item For each $x \ge 0$, the function $\mu(x,\cdot)$ is bounded and continuous for $k \in (\bar{\CC}_-,\bar{\CC}_+)$ and analytic for $k \in (\CC_-,\CC_+)$.

\item $\det \mu(x,k) = 1$ for $x \geq 0$ and $k \in \RR$.

\item For each $x \geq 0$ and $k \in (\bar{\CC}_-,\bar{\CC}_+)$, $\mu(x,k)$ obeys the symmetry
\begin{align}\label{musymm}
\mu(x,k)=\begin{cases}
\sigma_1 \overline{\mu(x,\bar{k})}\sigma_1 & \quad \text{if $\lambda=1$}, \\
\sigma_2 \overline{\mu(x,\bar{k})}\sigma_2 & \quad \text{if $\lambda=-1$}.
\end{cases}
\end{align}

\item As $k \to \infty$, the following asymptotic estimate holds uniformly for $x \geq 0$:
\begin{align}\label{mulargek}
\mu(x,k) = I + O(k^{-1}),  \qquad k \in (\bar{\CC}_-,\bar{\CC}_+).
\end{align}

\end{enumerate}
\end{lemma}
\begin{proof}
Let us consider the second column $\psi(x,k) \coloneqq [\mu(x,k)]_2$ of $\mu$. The second column of the Volterra equation \eqref{muvolterra} can be written as 
\begin{equation}\label{mu2volterra}
\psi(x,k)=\psi_0 - \int_x^\infty E(x,x',k)\mathsf{U}_0(x')\psi(x',k)dx',\qquad x \geq 0,
\end{equation}
where 
\begin{align}\label{Edef}
E(x,x',k) \coloneqq  \begin{pmatrix}e^{-2ik(x-x')} & 0 \\ 0 & 1\end{pmatrix} \quad \text{and} \quad \psi_0\coloneqq \begin{pmatrix}0 \\ 1\end{pmatrix}.
\end{align}
Define $\psi_l$ inductively for $l \ge 1$ by
\begin{equation*}
\psi_{l+1}(x,k) =-\int_x^\infty E(x, x', k)\mathsf{U}_0(x')\psi_{l}(x',k)dx',\qquad x \ge 0,\; \im k \geq 0.
\end{equation*}
Then
\begin{equation*}\label{Psil}
\psi_l(x,k)=(-1)^l \int_{x=x_{l+1} \le x_l \le \cdots \le x_1 \le \infty}\prod_{i=1}^l \big(E(x_{i+1},x_{i},k) \mathsf{U}_0(x_i) \big) \psi_0dx_1\cdots dx_l.
\end{equation*}
Note that we have the estimate
\begin{align}\label{Ebounded}
|E(x,x',k)|\le C, \qquad 0 \le x \le x' \le \infty, \; \im k \geq 0,
\end{align}
and that
$$\norm{\mathsf{U}_0}_{L^1(\RR_+)}=\int_{\RR_+}|\mathsf{U}_0|dx=\int_{\RR_+}\sqrt{2}|u_0|dx=\sqrt{2}\norm{u_0}_{L^1}.$$
Hence,
\begin{align} \label{psilestimate}
|\psi_l(x,k)| \le &\int_{x=x_{l+1} \le x_l \le \cdots \le x_1 \le \infty}\prod_{i=1}^l |E(x_{i+1},x_{i},k)| |\mathsf{U}_0(x_i)| |\psi_0| dx_1 \cdots dx_l 
\leq \frac{C^l \norm{u_0}_{L^1}^l}{l!},
\end{align}
Let us introduce the function
\begin{equation}\label{Psiseries}
\psi(x,k)=\sum_{l=0}^\infty \psi_l(x,k).
\end{equation}
For $x \ge 0$ and $\im k \geq 0$, we have
\begin{equation}\label{Psisum}
|\psi(x,k)| \le \sum_{l=0}^\infty |\psi_l(x,k)| \le e^{C\norm{u_0}_{L^1}}<\infty,
\end{equation}
and so the series $\psi(x,k)=\sum_{l=0}^\infty \psi_l(x,k)$ converges absolutely and uniformly for $x \geq 0$ and $\im k \geq 0$ to a continuous function $\psi(x,k)$. Using the uniform convergence of the series \eqref{Psiseries} and dominated convergence, we can integrate term by term and obtain
\begin{align*}
-\int_x^\infty E(x,x',k)\mathsf{U}_0(x')\psi(x',k)dx'=&-\sum_{l=0}^\infty \int_x^\infty E(x,x',k)\mathsf{U}_0(x')\psi_l(x',k)dx'\\
=&\sum_{l=0}^\infty \psi_{l+1}(x,k) = \psi(x,k)- \psi_0, \qquad x \geq 0, \; \im k \geq 0, 
\end{align*}
and so $\psi(x,k)$ satisfies \eqref{mu2volterra}. As a consequence of \eqref{mu2volterra}, $\psi(\cdot, k) \in C^1(\RR_+)$ for each $k \in \bar{\CC}_+$ and the second column of \eqref{xparteqn0} follows by differentiating \eqref{mu2volterra}. Moreover, $\psi$ is analytic in $\CC_+$ because a uniformly convergent series of analytic functions converges to an analytic function. Similar arguments apply to the first column of $\mu$ and the symmetries (\ref{musymm}) then follow from \eqref{xparteqn0}.
The unit determinant relation $\det \mu(x,k) = 1$ follows from \eqref{xparteqn0} and the fact that $\mathsf{U}_0$ is traceless.

Let us prove (\ref{mulargek}). Equation (\ref{mu2volterra}) can be written as
$$\hat{\psi}(x,k) = F(x,k) - \int_x^\infty E(x,x',k)\mathsf{U}_0(x')\hat{\psi}(x',k) dx',
$$
where
$$\hat{\psi}(x,k) \coloneqq  \psi(x,k) - \psi_0, \quad F(x,k) \coloneqq  - \int_x^\infty E(x,x',k)\mathsf{U}_0(x')\psi_0dx'.$$
Hence
\begin{align}\label{psihatFf}
|\hat{\psi}(x,k)| \leq |F(x,k)| + f(x,k), \qquad x \geq 0, \; \im k \geq 0,
\end{align}
where
$$f(x,k) \coloneqq  \int_x^\infty |\mathsf{U}_0(x')| |\hat{\psi}(x',k)| dx'.$$ 
Since $|\mathsf{U}_0(x')| \geq 0$, we find
$$-f_x(x,k) - |\mathsf{U}_0(x)| f(x,k)
 \leq |F(x,k) \mathsf{U}_0(x)|, \qquad x \geq 0, \; \im k \geq 0,$$
Hence
$$-\big(f(x,k) e^{-\int_x^\infty |\mathsf{U}_0(x')|dx'}\big)_x \leq |F(x,k) \mathsf{U}_0(x)|e^{-\int_x^\infty |\mathsf{U}_0(x')|dx'}$$
and integration from $x$ to $\infty$ gives
$$f(x,k) \leq \int_x^\infty |F(x',k) \mathsf{U}_0(x')|e^{\int_x^{x'} |\mathsf{U}_0(x'')|dx''} dx'.$$
Together with (\ref{psihatFf}), this yields
$$|\hat{\psi}(x,k)| \leq |F(x,k)| + e^{\sqrt{2}\|u_0\|_{L^1(\RR_+)}}\int_x^\infty \sqrt{2}|u_0(x')| |F(x',k)| dx'.
$$
Since $\|u_0\|_{L^1(\RR_+)} \leq \|u_0\|_{H^{1,1}(\RR_+)}$, we obtain
\begin{align}\label{psiFestimate}
|\psi(x,k) - \psi_0| \leq |F(x,k)| + C \|F(\cdot,k)\|_{L^{\infty}(x,\infty)}, \qquad x \geq 0, \; \im k \geq 0,
\end{align}
where $C > 0$ depends on $u_0$ but is independent of $x$ and $k$.
Since
$$F(x,k) = \begin{pmatrix}  - \int_x^\infty e^{-2ik(x-x')} u_0(x')  dx' \\ 0 \end{pmatrix},$$
an integration by parts gives, for $x \geq 0$ and $\im k \geq 0$,
$$|F(x,k)| = \bigg| \frac{1}{2ik}u_0(x) + \int_x^\infty \frac{e^{-2ik(x-x')}}{2ik} u_0'(x')  dx' \bigg|
\leq C\frac{\|u_0\|_{H^{1,1}(\RR_+)}}{|k|},$$
and hence the second column of (\ref{mulargek}) follows from (\ref{psiFestimate}). The first column of (\ref{mulargek}) then follows from the symmetries in (\ref{musymm}).
\end{proof}

Since $a(k) = \mu_{22}(0,k)$ and $b(k) = \mu_{12}(0,k)$, we immediately obtain the following properties of $a(k)$ and $b(k)$. 

\begin{corollary}\label{ablemma}
If $u_0 \in H^{1,1}(\RR_+)$, then $a(k)$ and $b(k)$ defined in (\ref{abdef}) are continuous for $\im k \geq 0$, analytic for $\im k > 0$, and satisfy 
\begin{align}\label{a2b21}
& |a(k)|^2 - \lambda |b(k)|^2= 1, \qquad k \in \RR.
\end{align}
As $k \to \infty$, the following asymptotic estimates hold uniformly for $\im k \geq 0$:
\begin{align}\label{ablargek}
a(k) = 1 + O(k^{-1}),  \quad b(k) = O(k^{-1}).
\end{align}
\end{corollary}


The next few lemmas are needed to show that $u_0 \mapsto r$ is a continuous map from $H^{1,1}(\RR_+) \cap \{u_0 \, | \, \text{$\Delta(k) \neq 0$ for all $k \in \RR$}\}$ into $H^{1,1}(\RR)$. 

\begin{lemma}\label{L2lemma}
The linear map $g \mapsto f(x)\coloneqq  \int_x^\infty g(x') dx'$  is continuous $H^{0,1}(\RR_+) \to L^2(\RR_+)$.
\end{lemma}

\begin{proof}
Assume that $g \in H^{0,1}(\RR_+)$. Then, by definition, $g, xg \in L^2(\RR_+)$. The estimate
$$|f| \le \int_0^\infty |g| dx = \int_0^\infty |(\langle x \rangle g) \langle x \rangle^{-1}| dx \leq \|\langle x\rangle g\|_{L^2(\RR_+)} \|\langle x \rangle^{-1}\|_{L^2(\RR_+)} < \infty,$$
shows that $f(x)$ is well-defined and absolutely continuous. Also, $f$ is bounded since $\lim_{x \to \infty}f(x)=0$.
Consider the Sobolev space $H^1(\RR) = \{h\in L^2(\RR) \, | \, h' \in L^2(\RR)\}$. A function $h$ belongs to $H^1(\RR)$ if and only if $k \mapsto \langle k \rangle\hat{h}(k)$ lies in $L^2(\RR)$, where $\hat{h}(k)=\int_{\RR}h(x)e^{-ikx}dx$ denotes the Fourier transform of $h$. 
Let us extend $g(x)$ and $f(x)$ to the negative real axis by setting $g(x) = 0$ and $f(x) = 0$ for $x < 0$. Then, since $\langle x \rangle g \in L^2(\RR)$, the Fourier transform $\hat{g}(k)$ of $g(x)$ belongs to $H^1(\RR)$. Moreover, the Fourier transform $\widehat{f'\,} = ik\hat{f}$ is well-defined as a tempered distribution and, for a test function $\phi$, we have
\begin{align*}
\langle \widehat{f'\,}, \phi \rangle = &\; \langle f', \hat{\phi} \rangle
= -\langle f, \hat{\phi}\,' \rangle
= - \int_0^\infty \int_x^\infty g(x')dx' \hat{\phi}'(x) dx
=- \int_x^\infty g(x') dx' \hat{\phi}(x) \bigg|_{x=0}^\infty
\\
&- \int_0^\infty g(x) \hat{\phi}(x) dx
= \int_0^\infty g(x') dx' \hat{\phi}(0) - \langle g, \hat{\phi} \rangle
= \langle \hat{g}(0)-\hat{g}, \phi \rangle,
\end{align*}
i.e., $ik \hat{f}(k)=\hat{g}(0)-\hat{g}(k)$ in the sense of distributions. Since $\hat{g} \in H^1(\RR)$, $\hat{g}$ is absolutely continuous and so 
\begin{equation*}\label{fhat}
ik \hat{f}(k)=\hat{g}(0)-\hat{g}(k) = -\int_0^k \hat{g}'(k') dk'.
\end{equation*}

Suppose we can show that the map 
\begin{align}\label{hatgtohatf}
\hat{g} \mapsto -\frac{1}{ik}\int_0^k \hat{g}'(k') dk'
\end{align}
is a bounded linear map $H^1(\RR) \to L^2(\RR)$. Since the Fourier transform is bounded $H^{0,1}(\RR) \to H^1(\RR)$ and the inverse Fourier transform is bounded $L^2(\RR) \to L^2(\RR)$, it would follow that the composition $g \mapsto \hat{g} \mapsto \hat{f} \mapsto f$ is a bounded linear map $H^{0,1}(\RR) \to L^2(\RR)$. 
Thus it only remains to show that the linear map \eqref{hatgtohatf} is bounded $H^1(\RR) \to L^2(\RR)$.
To prove this, note that
$$\bigg|\frac{1}{ik}\int_0^k \hat{g}'(k') dk'\bigg| 
\leq |(\mathcal{M}(\hat{g}'))(k/2)|,$$
where the Hardy-Littlewood maximal function $\mathcal{M}\varphi$ of $\varphi$ is defined by
$$(\mathcal{M}\varphi)(k) = \sup_{r>0} \frac{1}{2r} \int_{k-r}^{k+r} |\varphi(k')| dk'.$$
It is a standard theorem that $\varphi \mapsto \mathcal{M}\varphi$ is a bounded linear map $L^2(\RR) \to L^2(\RR)$. Hence
$$\|(\mathcal{M}(\hat{g}'))(\cdot/2)\|_{L^2(\RR)} \leq C\|(\mathcal{M}\hat{g}')(\cdot)\|_{L^2(\RR)} \leq C \|\hat{g}'(k)\|_{L^2(\RR)} \leq C \|\hat{g}(k)\|_{H^1(\RR)}.$$
This proves that the map \eqref{hatgtohatf}) is bounded $H^1(\RR) \to L^2(\RR)$.
\end{proof}

Let $E$ and $\psi_0$ be as in (\ref{Edef}). Define the linear operator $K_{u_0}$ acting on $2 \times 1$-vector-valued functions $f$ as follows:
\begin{equation}\label{Kdef}
K_{u_0}f (x,k)\coloneqq -\int_x^\infty E(x,x',k)\mathsf{U}_0(x')f(x',k)dx'.
\end{equation}
We will denote $K_{u_0}$ by $K$ for convenience and write $u_0$ explicitly when it is necessary. Equation \eqref{mu2volterra} can be written as $\psi = \psi_0 + K\psi$, i.e.,
\begin{equation}\label{psieqn}
\psi(x,k) = \psi_0+(1-K)^{-1}(K\psi_0),
\end{equation}
where $(1-K)^{-1}=1+K+\cdots$ is the Neumann series. Properties of $\psi(x,k)$ can be obtained by analyzing the operator $K$. In the following, $L^p_x(\RR_+,L^q_k(\RR))$ denotes the space of $L^q_k(\RR)$-valued $L^p_x(\RR_+)$-functions with the norm
$$\|f(x,k)\|_{L^p_x(\RR_+,L^q_k(\RR))}\coloneqq \Big\|\|f(x,k)\|_{L^q_k(\RR)}\Big\|_{L^p_x(\RR_+)}.$$

\begin{lemma}[Compare with Theorem 3.2 in  \cite{DZ2003}]\label{Klemma}
Define $K = K_{u_0}$ by (\ref{Kdef}). Then the following assertions hold for some $C > 0$ independent of $u_0$.
\begin{enumerate}[$(a)$]
\item \label{Klemmaitem1} 
The operator $K:L^\infty_x(\RR_+,L^2_k(\RR)) \to L^\infty_x(\RR_+,L^2_k(\RR))$ is bounded and
$$\|Kf\|_{L^\infty_x(\RR_+,L^2_k(\RR))} \leq C \|u_0\|_{L^1(\RR_+)}\|f\|_{L^\infty_x(\RR_+,L^2_k(\RR))}.$$

\item \label{Klemmaitem2}
The operator $K:L^2_x(\RR_+,L^2_k(\RR)) \to L^\infty_x(\RR_+,L^2_k(\RR))$ is bounded and
$$\|Kf\|_{L^\infty_x(\RR_+,L^2_k(\RR))} \leq C\|u_0\|_{L^2(\RR_+)}\|f\|_{L^2_x(\RR_+,L^2_k(\RR))}.$$

\item \label{Klemmaitem3}
The operator $(1-K)^{-1}:L^\infty_x(\RR_+,L^2_k(\RR)) \to L^\infty_x(\RR_+,L^2_k(\RR))$ is bounded and
$$\|(1-K)^{-1}\|_{L^\infty_x(\RR_+,L^2_k(\RR)) \to L^\infty_x(\RR_+,L^2_k(\RR))} \le e^{C\|u_0\|_{L^1(\RR_+)}}.$$

\item \label{Klemmaitem4}
The operator $K:L^\infty_x(\RR_+,L^2_k(\RR)) \to L^2_x(\RR_+,L^2_k(\RR))$ is bounded and
$$\|Kf\|_{L^2_x(\RR_+,L^2_k(\RR))} \le C \|u_0\|_{H^{0,1}(\RR_+)}\|f\|_{L^\infty_x(\RR_+,L^2_k(\RR)}.$$

\item \label{Klemmaitem5}
$\|K\psi_0\|_{L^\infty_x(\RR_+,L^2_k(\RR))} \leq C\|u_0\|_{L^2(\RR_+)}$.

\item \label{Klemmaitem6}
$\|K\psi_0\|_{L^2_x(\RR_+,L^2_k(\RR))} \le C \|u_0\|_{H^{0,1}(\RR_+)}$.

\end{enumerate}
\end{lemma}
\begin{proof}
For $x \in \RR_+$ and $k \in \RR$, we have
\begin{equation}\label{Kfestimate}
|Kf(x,k)| \le C \int_x^\infty |u_0(x')| |f(x',k)| dx'.
\end{equation}
Hence $\|K\|_{L^\infty_x(\RR_+) \to L^\infty_x(\RR_+)} \le C \|u_0\|_{L^1(\RR_+)}$ and
\begin{align*}
\|Kf&\|_{L^\infty_x(\RR_+,L^2_k(\RR))}^2
= \underset{x \in \RR_+}{\text{ess sup}} \int_\RR  |Kf(x,k)|^2 dk
	\\
& \leq 
C \;\underset{x \in \RR_+}{\text{ess sup}} \int_\RR  \bigg(\int_x^\infty |u_0(x')| |f(x',k)| dx'\bigg)^2 dk
	\\
& = C\! \int_0^\infty dx' |u_0(x')| \int_0^\infty dx'' |u_0(x'')| \int_\RR dk  |f(x',k)| |f(x'',k)|,
\end{align*}
where we have used Fubini's theorem to change the order of integration.
Since
\begin{align}\label{ffestimate}
\int_\RR dk  |f(x',k)| |f(x'',k)| \leq \|f(x',\cdot)\|_{L^2_k} \|f(x'',\cdot)\|_{L^2_k},
\end{align}
we obtain
\begin{align*}
\|Kf\|_{L^\infty_x(\RR_+,L^2_k(\RR))}^2
\leq C\bigg(\int_0^\infty dx' |u_0(x')|  \|f(x',\cdot)\|_{L^2_k} \bigg)^2,
\end{align*}
i.e.,
\begin{align}\label{Kfinfty2est}
\|Kf\|_{L^\infty_x(\RR_+,L^2_k(\RR))}
\leq C \int_0^\infty |u_0(x')|  \|f(x',\cdot)\|_{L^2_k} dx'.
\end{align}
Assertion (\ref{Klemmaitem1}) follows because \eqref{Kfinfty2est} implies
$$\|Kf\|_{L^\infty_x(\RR_+,L^2_k(\RR))}
\leq C \|u_0\|_{L^1}    \underset{x \in \RR_+}{\text{ess sup}} \|f(x,\cdot)\|_{L^2_k} 
= C \|u_0\|_{L^1}\|f\|_{L^\infty_x(\RR_+,L^2_k(\RR))},
$$
and assertion (\ref{Klemmaitem2}) follows because (\ref{Kfinfty2est}) implies
$$\|Kf\|_{L^\infty_x(\RR_+,L^2_k(\RR))}
\leq  C \|u_0\|_{L^2}   \|\|f(x',\cdot)\|_{L^2_k} \|_{L^2_x}
= C\|u_0\|_{L^2}\|f\|_{L^2_x(\RR_+,L^2_k(\RR))}.
$$
Assertion (\ref{Klemmaitem3}) follows from an estimate as in (\ref{psilestimate}) as follows:
\begin{align*}
 \|(1-K)^{-1}\|_{L^\infty_x(\RR_+,L^2_k(\RR)) \to L^\infty_x(\RR_+,L^2_k(\RR))} 
 & \le \sum_{l=0}^\infty \|K^l\|_{L^\infty_x(\RR_+,L^2_k(\RR)) \to L^\infty_x(\RR_+,L^2_k(\RR))}
	\\
&\leq \sum_{l=0}^\infty \frac{(C\|u_0\|_{L^1})^l}{l!}
\leq e^{C\|u_0\|_{L^1}}.
\end{align*}

To prove assertion (\ref{Klemmaitem4}), we note that, using \eqref{Kfestimate} and \eqref{ffestimate},
\begin{align*}
\|Kf\|_{L^2_x(\RR_+,L^2_k(\RR))}^2
& \leq \int_0^\infty dx \int_x^\infty dx' |u_0(x')|  \int_x^\infty dx'' |u_0(x'')| \|f(x',\cdot)\|_{L^2_k} \|f(x'',\cdot)\|_{L^2_k}.
\end{align*}
Consequently,
\begin{align*}
\|Kf\|_{L^2_x(\RR_+,L^2_k(\RR))}^2
& \leq \|f\|_{L^\infty_x(\RR_+,L^2_k(\RR))}^2 \int_0^\infty dx \bigg(\int_x^\infty dx' |u_0(x')|\bigg)^2,
\end{align*}
that is,
\begin{align*}
\|Kf\|_{L^2_x(\RR_+,L^2_k(\RR))}
 \leq  \bigg\|\int_x^\infty |u_0(x')| dx'\bigg\|_{L^2_x(\RR_+)} \|f\|_{L^\infty_x(\RR_+,L^2_k(\RR))}
\end{align*}
Using that, by Lemma \ref{L2lemma},
$$\bigg\|\int_x^\infty |u_0(x')| dx'\bigg\|_{L^2_x(\RR_+)} \leq C\|u_0\|_{H^{0,1}(\RR_+)},$$
assertion (\ref{Klemmaitem4}) follows.
Finally, by Plancherel's theorem,
\begin{align*}
\|K\psi_0\|_{L^\infty_x(\RR_+,L^2_k(\RR))} 
= \underset{x \in \RR_+}{\text{ess sup}}\;
\bigg\|\int_x^\infty e^{-2ik(x-x')}u_0(x') dx' \bigg\|_{L^2_k(\RR)}
\leq C\|u_0\|_{L^2(\RR_+)}
\end{align*}
and
$$\|K\psi_0\|_{L^2_x(\RR_+,L^2_k(\RR))}\le C\left\|\frac{1}{\langle x \rangle}\left(\int_x^\infty \left|\langle x' \rangle u_0(x')\right|^2 dx'\right)^{1/2}\right\|_{L^2_x(\RR_+)} \le C \|u_0\|_{H^{0,1}(\RR_+)},$$
which proves (\ref{Klemmaitem5}) and (\ref{Klemmaitem6}).
\end{proof}

\begin{lemma}\label{psipsi0lemma}
Let $\psi$ be the solution of (\ref{psieqn}) corresponding to $u_0$. Then
\begin{enumerate}[$(i)$]
\item \label{psipsi0lemmaitem1}
$\|\psi(x,k)-\psi_0\|_{L^\infty_x(\RR_+,L^2_k(\RR))} \le Ce^{C\|u_0\|_{L^1}} \|u_0\|_{L^2}$, and

\item \label{psipsi0lemmaitem2}
$\|\psi(x,k)-\psi_0\|_{L^2_x(\RR_+,L^2_k(\RR))} \le 
C(e^{C\|u_0\|_{L^1}} \|u_0\|_{L^2} + 1)\|u_0\|_{H^{0,1}}$,
\end{enumerate}
where the constant $C$ is independent of $u_0$.
\end{lemma}
\begin{proof}
We infer from equation \eqref{psieqn} and Lemma \ref{Klemma} (\ref{Klemmaitem3}) that
\begin{align*}
\|\psi(x,k)-\psi_0 \|_{L^\infty_x(\RR_+,L^2_k(\RR))}=&\; \|(1-K)^{-1}(K\psi_0)\|_{L^\infty_x(\RR_+,L^2_k(\RR))}
	\\ 
\le &\; e^{C\|u_0\|_{L^1(\RR_+)}} \| K\psi_0 \|_{L^\infty_x(\RR_+,L^2_k(\RR))}.
\end{align*}
Assertion (\ref{psipsi0lemmaitem1}) now follows from Lemma \ref{Klemma} (\ref{Klemmaitem5}). 
On the other hand, by (\ref{psieqn}),
\begin{align*}
\|\psi(x,k)&-\psi_0\|_{L^2_x(\RR_+,L^2_k(\RR))}=\left\|K(\psi(x,k)-\psi_0)+K\psi_0\right\|_{L^2_x(\RR_+,L^2_k(\RR))}\\
\le&\; \|K(\psi(x,k)-\psi_0)\|_{L^2_x(\RR_+,L^2_k(\RR))}+\left\|K\psi_0\right\|_{L^2_x(\RR_+,L^2_k(\RR))}\\
\le&\; \|K\|_{L^\infty_x(\RR_+,L^2_k(\RR)) \to L^2_x(\RR_+,L^2_k(\RR))} \|\psi(x,k)-\psi_0\|_{L^\infty_x(\RR_+,L^2_k(\RR))}+\|K\psi_0\|_{L^2_x(\RR_+,L^2_k(\RR))}.
\end{align*}
Using assertion (\ref{psipsi0lemmaitem1}) as well as Lemma \ref{Klemma} (\ref{Klemmaitem4}), this gives
\begin{align*}
\|\psi(x,k) -\psi_0\|_{L^2_x(\RR_+,L^2_k(\RR))} \le C\|u_0\|_{H^{0,1}} e^{C\|u_0\|_{L^1}} \|u_0\|_{L^2} +\|K\psi_0\|_{L^2_x(\RR_+,L^2_k(\RR))},
\end{align*}
and so assertion (\ref{psipsi0lemmaitem2}) follows from Lemma \ref{Klemma} (\ref{Klemmaitem6}).
\end{proof}

\subsection{Continuity of the spectral mappings}
Using the estimates established in the previous section, we will prove the continuity of the spectral mappings. The proof of the next lemma uses ideas from \cite[Theorem 3.2]{DZ2003} and \cite[Lemma 7.2]{DP2011}.

\begin{lemma} \label{continuitylemma1}
If $u_0 \in H^{1,1}(\RR_+)$, then $a-1, b \in H^1(\RR)$. Moreover, the maps
\begin{equation}\label{u0toab}
u_0 \mapsto a-1:H^{1,1}(\RR_+) \to H^1(\RR), \qquad
u_0 \mapsto b:H^{1,1}(\RR_+) \to H^1(\RR)	
\end{equation}
are continuous and map bounded sets to bounded sets.
\end{lemma}
\begin{proof}
Let $u_0 \in H^{1,1}(\RR_+)$. Since 
$$\begin{pmatrix}b(k) \\ a(k)-1\end{pmatrix}=\psi(0,k)-\psi_0,$$ 
it is immediate from part (\ref{psipsi0lemmaitem1}) of Lemma \ref{psipsi0lemma} that $a-1, b \in L^2(\RR)$. To show that $a', b' \in L^2(\RR)$, define $\tilde{\mu} \coloneqq  (\partial_k+ix\hat{\sigma}_3)\mu$. Suppose temporarily that $u_0$ has compact support (the set of such functions $u_0$ is dense in $H^{1,1}(\RR_+)$). Then $\mu(x,\cdot)$ is entire and we can differentiate under the integral sign in (\ref{muvolterra}) to obtain
\begin{align*}
& \tilde{\mu}(x,k)=- \int_x^\infty e^{-ik(x-x')\hat{\sigma}_3}\Big([i\sigma_3, x'\mathsf{U}_0(x') \mu(x',k)]  + \mathsf{U}_0(x')\partial_k\mu(x',k)\Big)dx'
	\\
& =- i\int_x^\infty e^{-ik(x-x')\hat{\sigma}_3}\big([\sigma_3, x'\mathsf{U}_0(x')] \mu(x',k)\big) dx'
- \int_x^\infty e^{-ik(x-x')\hat{\sigma}_3} \big(\mathsf{U}_0(x')\tilde{\mu}(x',k)\big)dx'.
\end{align*}
Hence, the second column $\phi(x,k) \coloneqq  [\tilde{\mu}(x,k)]_2$ of $\tilde{\mu}$ satisfies the Volterra equation
\begin{align}\label{phieqn}
\phi(x,k)= iK'_{xu_0}\psi+K_{u_0}\phi = h_1+h_2+K_{u_0}\phi,
\end{align}
where $K'$ is the operator defined in \eqref{Kdef} with $\mathsf{U}_0$ replaced by $[\sigma_3, \mathsf{U}_0]$, and
\begin{align*}
h_1 \coloneqq iK'_{xu_0}(\psi(x,k)-\psi_0),\quad h_2 \coloneqq iK'_{xu_0}\psi_0.
\end{align*}
The vectors $\psi$ and $\phi$ are related by
\begin{align}\label{phipsirelation}
\phi(x,k) = \partial_k\psi(x,k) + \begin{pmatrix} 2ix\psi_{1}(x,k) \\ 0 \end{pmatrix}.
\end{align}

Let us now again consider the case of a general potential $u_0 \in H^{1,1}$.
Since $xu_0 \in L^2(\RR_+)$, we have $h_1 \in L^\infty_x(\RR_+,L^2_k(\RR))$ by Lemma \ref{Klemma} (\ref{Klemmaitem2}) and Lemma \ref{psipsi0lemma} (\ref{psipsi0lemmaitem2}), and we have $h_2 \in L^\infty_x(\RR_+,L^2_k(\RR))$ by Lemma \ref{Klemma} (\ref{Klemmaitem5}).
By Lemma \ref{Klemma} (\ref{Klemmaitem3}), this means that the solution of \eqref{phieqn} satisfies
$$\phi=(1-K_{u_0})^{-1}(h_1+h_2) \in L^\infty_x(\RR_+,L^2_k(\RR)).$$
In particular, if $u_0$ has compact support, then $\begin{pmatrix}b'(k)\\a'(k)\end{pmatrix}=\phi(0,k) \in L^2(\RR)$ so that $a(k)-1$ and $b(k)$ lie in $H^1(\RR)$. To extend this result to the case of general $u_0 \in H^{1,1}$, we consider the continuity of the $u_0$-dependence. 

Let $u_0, \check{u}_0 \in H^{1,1}(\RR_+)$ and let $\psi,\check{\psi}, \phi, \check{\phi}$ be the associated solutions of the equations \eqref{mu2volterra} and \eqref{phieqn}, respectively. Denote $\Delta u_0\coloneqq u_0-\check{u}_0$, $\Delta \psi\coloneqq \psi-\check{\psi}$, and $\Delta \phi \coloneqq \phi-\check{\phi}$. We want to estimate $\|\Delta \psi\|_{L^\infty_x(\RR_+,L^2_k(\RR))}$ and $\|\Delta \phi\|_{L^\infty_x(\RR_+,L^2_k(\RR))}$ in terms of $\|\Delta u_0\|_{H^{1,1}}$. Note that $\Delta \psi$ satisfies the Volterra equation
\begin{equation}\label{DeltapsiVolterra}
\Delta \psi=K_{\Delta u_0}\check{\psi}+K_{u_0}\Delta \psi.
\end{equation}
Using parts (\ref{Klemmaitem3}), (\ref{Klemmaitem1}), and (\ref{Klemmaitem5}) of Lemma \ref{Klemma} as well as Lemma \ref{psipsi0lemma} (\ref{psipsi0lemmaitem1}), it follows that
\begin{align}\nonumber
\|\Delta \psi\|_{L^\infty_x(\RR_+,L^2_k(\RR))}=&\left\|(1-K_{u_0})^{-1}K_{\Delta u_0}\check{\psi}\right\|_{L^\infty_x(\RR_+,L^2_k(\RR))}
\le  C\|K_{\Delta u_0}\check{\psi}\|_{L^\infty_x(\RR_+,L^2_k(\RR))}
	\\\nonumber
=&\; C\|K_{\Delta u_0}(\check{\psi}-\psi_0)+K_{\Delta u_0}\psi_0\|_{L^\infty_x(\RR_+,L^2_k(\RR))} 
	\\\label{DeltapsiLinftyL2}
\le &\; C\|\Delta u_0\|_{L^1} + C \|\Delta u_0\|_{L^2}
\leq C\|\Delta u_0\|_{H^{0,1}}
\end{align}
uniformly for $u_0, \check{u}_0$ in bounded subsets of $H^{1,1}(\RR_+)$.
Moreover, using parts (\ref{Klemmaitem4}) and (\ref{Klemmaitem6}) of Lemma \ref{Klemma}, part (\ref{psipsi0lemmaitem1}) of Lemma \ref{psipsi0lemma}, as well as (\ref{DeltapsiLinftyL2}), we obtain 
\begin{align}\nonumber
\|\Delta \psi&\|_{L^2_x(\RR_+,L^2_k(\RR))}= \|K_{\Delta u_0}(\check{\psi}-\psi_0)+K_{\Delta u_0}\psi_0+K_{u_0}\Delta \psi\|_{L^2_x(\RR_+,L^2_k(\RR))}
	\\\nonumber
& \le C\|\Delta u_0\|_{H^{0,1}}\|\check{\psi}-\psi_0\|_{L^\infty_x(\RR_+,L^2_k(\RR))}
+C\|\Delta u_0\|_{H^{0,1}}
+C\|\Delta \psi\|_{L^\infty_x(\RR_+,L^2_k(\RR))}
	\\ \label{DeltapsiL2L2}
& \le C\|\Delta u_0\|_{H^{0,1}}
\end{align}
uniformly for $u_0, \check{u}_0$ in bounded subsets of $H^{1,1}(\RR_+)$.
Similarly, $\Delta \phi$ satisfies the Volterra equation
\begin{equation}
\Delta \phi= i\left(K'_{x \Delta u_0}\check{\psi}+K'_{xu_0}\Delta \psi\right)+K_{\Delta u_0}\check{\psi}+K_{u_0}\Delta\phi.
\end{equation}
Hence, using parts (\ref{Klemmaitem3}), (\ref{Klemmaitem2}), (\ref{Klemmaitem5}), and (\ref{Klemmaitem1}) of Lemma \ref{Klemma}, parts (\ref{psipsi0lemmaitem1}) and (\ref{psipsi0lemmaitem2}) of Lemma \ref{psipsi0lemma}, as well as (\ref{DeltapsiL2L2}),
\begin{align}\nonumber
\|\Delta \phi &\|_{L^\infty_x(\RR_+,L^2_k(\RR))}= \left\|(1-K_{u_0})^{-1}\left(i\left(K'_{x \Delta u_0}\check{\psi}+K'_{xu_0}\Delta \psi\right)+K_{\Delta u_0}\check{\psi}\right)\right\|_{L^\infty_x(\RR_+,L^2_k(\RR))}
	\\\nonumber
\le&\; C\left\|i\left(K'_{x \Delta u_0}\check{\psi}+K'_{xu_0}\Delta \psi\right)+K_{\Delta u_0}\check{\psi}\right\|_{L^\infty_x(\RR_+,L^2_k(\RR))}
	\\\nonumber
\le&\; C\big\|K'_{x \Delta u_0}(\check{\psi}-\psi_0)+K'_{x \Delta u_0}\psi_0\big\|_{L^\infty_x(\RR_+,L^2_k(\RR))}
	\\\nonumber
&+C \big\|iK'_{xu_0}\Delta \psi + K_{\Delta u_0}(\check{\psi} - \psi_0)
+K_{\Delta u_0}\psi_0 \big\|_{L^\infty_x(\RR_+,L^2_k(\RR))}
	\\\nonumber
\le&\; C\|x \Delta u_0\|_{L^2}\|\check{\psi}-\psi_0\|_{L^2_x(\RR_+,L^2_k(\RR))}
+ C \|x \Delta u_0\|_{L^2}
	\\\nonumber
&+C\|xu_0\|_{L^2}\|\Delta \psi\|_{L^2_x(\RR_+,L^2_k(\RR))}
+C\|\Delta u_0\|_{L^1}\|\check{\psi}-\psi_0\|_{L^\infty_x(\RR_+,L^2_k(\RR))}
+C\|\Delta u_0\|_{L^2}
	\\ \label{DeltaphiLinftyL2}
\le&\; C\|\Delta u_0\|_{H^{0,1}}
\end{align}
uniformly for $u_0, \check{u}_0$ in bounded subsets of $H^{1,1}(\RR_+)$.
The inequalities (\ref{DeltapsiLinftyL2}) and (\ref{DeltaphiLinftyL2}) show that $\|\Delta \psi(0,k)\|_{L^2}$ and $\|\Delta \phi(0,k)\|_{L^2}$ tend to zero whenever $\|\Delta u_0\|_{H^{1,1}}$ tends to zero, provided that $u_0, \check{u}_0$ remain in a bounded subset of $H^{1,1}(\RR_+)$. 
Since (\ref{phipsirelation}) holds on a dense subset of $H^{1,1}(\RR_+)$, a continuity argument implies that $\mu(x,\cdot) \in H^1(\RR)$ for any $u_0 \in H^{1,1}(\RR_+)$ and that the relation (\ref{phipsirelation}) holds for any $u_0 \in H^{1,1}(\RR_+)$. 
We conclude that $a-1, b \in H^1(\RR)$ and that the maps in (\ref{u0toab}) are continuous and map bounded sets to bounded sets.
\end{proof}

We want to show that the reflection coefficient $r(k)$ lies in the weighted Sobolev space $H^{1,1}(\RR)$, but the spectral functions $a(k)-1$ and $b(k)$ are {\it not} in $H^{1,1}(\RR)$. To proceed, we therefore need to subtract suitable constants.

\begin{lemma}\label{specfcnlemma}
Let $u_0 \in H^{1,1}(\RR_+)$ and define $b_1 \in \CC$ by
\begin{align}\label{a1b1}
b_1 = \frac{u_0(0)}{2i}.
\end{align}
Then $kb(k) - b_1 \in L^2(\RR)$.
\end{lemma}
\begin{proof}
If we expand the equation \eqref{mu2volterra} in components $\psi=\begin{pmatrix} \psi_1 & \psi_2 \end{pmatrix}^T$, we have
\begin{equation}\label{psicomponent}
\begin{cases}
\psi_1(x,k)=-\int_x^\infty e^{-2ik(x-x')}u_0(x')\psi_2(x',k)dx', \\
\psi_2(x,k)=1-\int_x^\infty \lambda \overline{u_0(x')}\psi_1(x',k) dx'.\\
\end{cases}
\end{equation}
Thus an integration by parts gives
\begin{align}\nonumber
2ik\psi_1(x,k) 
=& - \int_x^\infty [\partial_{x'}e^{-2ik(x-x')}] u_0(x') \psi_2(x',k) dx'	
	\\ \nonumber
= &\; u_0(x) \psi_2(x,k)	
+ \int_x^\infty e^{-2ik(x-x')} u_0'(x') \psi_2(x',k) dx'	
	\\\label{kpsi1}
& + \lambda \int_x^\infty e^{-2ik(x-x')} |u_0(x')|^2 \psi_1(x',k) dx'	, \qquad x\geq 0, ~ k \in \RR.
\end{align}	
We also have
\begin{align}\label{kpsi2}
k(\psi_2(x,k) - 1) = & - \int_x^\infty \lambda \overline{u_0(x')} k\psi_1(x',k) dx', \qquad x\geq 0, ~ k \in \RR.
\end{align}
Since $b(k) = \psi_1(0,k)$, we deduce from \eqref{a1b1} and \eqref{kpsi1} that
\begin{align*}
\|2i(kb(k) - b_1)\|_{L^2_k}=&\|2ik\psi_1(0,k) - u_0(0)\|_{L^2_k} 
	\\
\leq
&\; \|u_0(0)( \psi_2(0,k)-1)\|_{L^2_k} + \left\|\int_0^\infty e^{2ikx} u_0'(x) (\psi_2(x,k)-1)dx\right\|_{L^2_k}\\
&+ \left\|\int_0^\infty e^{2ikx} u_0'(x)dx\right\|_{L^2_k} + \left\|\int_0^\infty e^{2ikx} |u_0(x)|^2 \psi_1(x,k) dx\right\|_{L^2_k}.
\end{align*}
Employing Lemma \ref{psipsi0lemma} and Plancherel's theorem, we can estimate each of the terms on the right-hand side. We obtain
\begin{align*}
\|2i(kb(k) - b_1)\|_{L^2_k}
\le&\; C|u_0(0)| \|\psi_2(x,k)-1\|_{L^\infty_x(\RR_+,L^2_k(\RR))}+\|u_0'\|_{L^2}\|\psi_2(x,k)-1\|_{L^2_x(\RR_+,L^2_k(\RR))}\\
&+C\|u_0'\|_{L^2}+C\|u_0\|_{L^\infty} \|u_0\|_{L^2}\|\psi_1(x,k)\|_{L^2_x(\RR_+,L^2_k(\RR))} < \infty,
\end{align*}
showing that $kb(k) - b_1 \in L^2(\RR)$. Here we have used that by the Fubini--Tonelli theorem, if $f \in L^2_k(\RR, L^2_x(\RR^+))$ or $f \in L^2_x(\RR^+, L^2_k(\RR))$, then
$\|\|f(x,k)\|_{L^2_x}\|_{L^2_k} = \|f(x,k)\|_{L^2_k}\|_{L^2_x}$.

\end{proof}

\begin{lemma}\label{continuitylemma}
The following maps are continuous:
\begin{subequations}
\begin{align}
&u_0 \mapsto b_1:H^{1,1}(\RR_+) \to \CC,\label{u0toa1b1}
	\\
& u_0 \mapsto kb(k) - b_1:H^{1,1}(\RR_+) \to L^2(\RR).\label{u0tokbk}
\end{align}
\end{subequations}

\end{lemma}

\begin{proof}
Let $u_0, \check{u}_0 \in H^{1,1}(\RR_+)$ and let $\psi,\check{\psi}$ be the associated solutions of \eqref{mu2volterra}. Let $a, b, b_1$ and $\check{a}, \check{b}, \check{b}_1$ be the quantities corresponding to $u_0$ and $\check{u}_0$, respectively, and let $\Delta u_0 \coloneqq u_0-\check{u}_0$, $\Delta a \coloneqq a-\check{a}$, $\Delta b \coloneqq b-\check{b}$, 
and $\Delta b_1 \coloneqq b_1-\check{b}_1$. We want to estimate
$$|\Delta b_1|,\quad \|k\Delta b-\Delta b_1\|_{L^2}$$
in terms of $\|\Delta u_0\|_{H^{1,1}}$. 

From the definition \eqref{a1b1} of $b_1$, it follows that
\begin{align*}
|\Delta b_1| \le C\|\Delta u_0\|_{L^\infty},
\end{align*}
and hence $|\Delta b_1| \le C \|\Delta u_0\|_{H^{1,1}}$ uniformly for $u_0, \check{u}_0$ in bounded subsets of $H^{1,1}$. Next, if we set $x=0$ in equation \eqref{kpsi1} and subtract $2ib_1$ from both sides,
$$2ikb(k)-2ib_1=u_0(0) (a(k)-1)+\int_0^\infty e^{2ikx} u_0'(x) \psi_2(x,k) dx+\lambda \int_0^\infty e^{2ikx} |u_0(x)|^2 \psi_1(x,k) dx.$$
We infer that the differences $\Delta u_0, \Delta a, \Delta b, \Delta \psi$ satisfy the equation
\begin{align*}
2ik \Delta b-2i \Delta b_1=&\; u_0(0)\Delta a+\Delta u_0(0) (\check{a} - 1)+\int_0^\infty e^{2ikx} \left(u_0' \Delta \psi_2+\Delta u_0' \check{\psi}_2\right)dx\\
&+\lambda \int_0^\infty e^{2ikx} \left(|u_0|^2 \Delta \psi_1+ (u_0 \overline{\Delta u_0} + \Delta u_0 \overline{\check{u}_0}) \check{\psi}_1\right) dx.
\end{align*}
Hence, proceeding as in the proof of Lemma \ref{continuitylemma1} (see the estimates \eqref{DeltapsiLinftyL2} and \eqref{DeltapsiL2L2}) and using Lemma \ref{psipsi0lemma} (\ref{psipsi0lemmaitem2}),
\begin{align*}
\|k\Delta b-\Delta b_1\|_{L^2_k} \le &\; C \|\Delta a\|_{L^2_k}+C \|\Delta u_0\|_{L^\infty_x}+ \|u_0'\|_{L^2_x}\|\Delta \psi_2\|_{L^2_k(\RR,L^2_x(\RR_+))}
	\\
&+\|\Delta u_0'\|_{L^2_x}\big(\|\check{\psi}_2-1\|_{L^2_k(\RR,L^2_x(\RR_+))}+1\big)
	\\
&+\|u_0\|_{L^\infty_x}\|u_0\|_{L^2_x}\|\Delta \psi_1\|_{L^2_k(\RR,L^2_x(\RR_+))}
	\\
& + \| \Delta u_0 \|_{L^\infty_x} (\|u_0\|_{L^2_x} + \|\check{u}_0\|_{L^2_x})\| \check{\psi}_1\|_{L^2_k(\RR,L^2_x(\RR_+))}\\
\le&\; C\|\Delta u_0\|_{H^{1,1}}
\end{align*}
uniformly for $u_0, \check{u}_0$ in bounded subsets of $H^{1,1}$.
\end{proof}

\begin{remark}
Conclusions similar to those established in Lemma \ref{specfcnlemma} and \ref{continuitylemma} for $b(k)$ hold also for the spectral function $a(k)$. Namely, if for $u_0 \in H^{1,1}(\RR_+)$ we define $a_1 :=-\frac{\lambda}{2i} \int_0^\infty  |u_0(x)|^2 dx$, then the mappings
\begin{subequations}
\begin{align}
&u_0 \mapsto a_1:H^{1,1}(\RR_+) \to \CC	\\
&u_0 \mapsto k(a(k)-1) - a_1:H^{1,1}(\RR_+) \to L^2(\RR),\label{u0tokak}
\end{align}
\end{subequations}
are continuous. We will not need these properties, but for completeness we give the proof of the fact that $k(a(k)-1) - a_1 \in L^2(\RR)$; given this fact, the continuity follows by arguments similar to those used in the proof of Lemma \ref{continuitylemma}. To see that $k(a(k)-1) - a_1 \in L^2(\RR)$, note that equation \eqref{kpsi2} implies
\begin{align}\nonumber
 \|k(a(k)-1) - a_1\|_{L^2_k} 
& = \|k(\psi_2(0,k)-1) - a_1\|_{L^2_k} 
	\\\nonumber
& = \bigg\| -\int_0^\infty \lambda \overline{u_0(x)} \bigg(k\psi_1(x,k) - \frac{u_0(x)}{2i}\bigg)dx\bigg\|_{L^2_k} 
	\\\nonumber
& =\frac{1}{2} \bigg\| \int_0^\infty \langle x \rangle \overline{u_0(x)} \frac{2ik\psi_1(x,k) - u_0(x)}{\langle x \rangle} dx \bigg\|_{L^2_k} 
	\\ \label{kaminus1a1}
& \leq C \|u_0\|_{H^{0,1}} \left\|\left\|\frac{2ik\psi_1(x,k) - u_0(x)}{\langle x \rangle}\right\|_{L^2_x}\right\|_{L^2_k}.
\end{align}
Furthermore, from equation \eqref{kpsi1}, we have
\begin{align}\nonumber
\frac{2ik\psi_1(x,k) - u_0(x)}{\langle x \rangle}=&\; u_0(x)\frac{\psi_2(x,k)-1}{\langle x \rangle}+\frac{1}{\langle x \rangle}\int_x^\infty e^{-2ik(x-x')} u_0'(x') \psi_2(x',k) dx'
	\\ \label{psi1quotient}
&+\frac{\lambda}{\langle x \rangle} \int_x^\infty e^{-2ik(x-x')} |u_0(x')|^2 \psi_1(x',k) dx'.
\end{align}
Using Lemma \ref{psipsi0lemma} and the fact that the function $\frac{1}{\langle x \rangle}$ is in $L^\infty_x(\RR_+) \cap L^2_x(\RR_+)$, we can estimate the three terms on the right-hand side of (\ref{psi1quotient}) for $u_0 \in H^{1,1}(\RR_+)$ as follows:
\begin{align*}
\left\Vert u_0(x)\frac{\psi_2(x,k)-1}{\langle x \rangle} \right\Vert_{L^2_x(\RR_+,L^2_k(\RR))}=&\left\Vert\frac{u_0}{\langle x \rangle}\left\Vert \psi_2(x,k)-1\right\Vert_{L^2_k}\right\Vert_{L^2_x}
	\\
\le &\; C \left\Vert u_0\right\Vert_{L^\infty} \left\Vert \psi_2(x,k)-1 \right\Vert_{L^2_x(\RR_+,L^2_k(\RR))}
< \infty,
\end{align*}
\begin{align*}
\bigg\Vert \frac{1}{\langle x \rangle}\int_x^\infty  e^{-2ik(x-x')} &u_0'(x') \psi_2(x',k) dx' \bigg\Vert_{L^2_x(\RR_+,L^2_k(\RR))}
	\\
\le &\left\Vert \frac{\left\Vert u_0' \right\Vert_{L^2_x}}{\langle x \rangle}\left\Vert \psi_2(x,k)-1 \right\Vert_{L^2_x}\right\Vert_{L^2_x(\RR_+,L^2_k(\RR))}
+ \left\Vert \frac{\left\Vert u_0' \right\Vert_{L^2_x}}{\langle x \rangle}\right\Vert_{L^2_x}
	\\
\le &\; C\left\Vert u_0' \right\Vert_{L^2} \left\Vert \psi_2(x,k)-1 \right\Vert_{L^2_x(\RR_+,L^2_k(\RR))}
+ C \left\Vert u_0' \right\Vert_{L^2}< \infty,
\end{align*}
and
\begin{align*}
\bigg\Vert \frac{\lambda}{\langle x \rangle} \int_x^\infty e^{-2ik(x-x')} &|u_0(x')|^2 \psi_1(x',k) dx' \bigg\Vert_{L^2_x(\RR_+,L^2_k(\RR))}
	\\
\le &\; \bigg\Vert \frac{\left\Vert u_0 \right\Vert_{L^\infty_x}}{\langle x \rangle} \int_0^\infty \left\vert u_0(x) \psi_1(x,k)\right\vert dx \bigg\Vert_{L^2_x(\RR_+,L^2_k(\RR))}
	\\
\le &\; \bigg\Vert \frac{\left\Vert u_0 \right\Vert_{L^\infty_x} \left\Vert u_0 \right\Vert_{L^2_x}}{\langle x \rangle}\left\Vert \psi_1(x,k) \right\Vert_{L^2_x} \bigg\Vert_{L^2_x(\RR_+,L^2_k(\RR))}
	\\
\le &\; C\left\Vert u_0 \right\Vert_{L^\infty} \left\Vert u_0 \right\Vert_{L^2} \left\Vert \psi_1(x,k)\right\Vert_{L^2_k(\RR,L^2_x(\RR_+))} 
< \infty.
\end{align*}
Combining the above three estimates with (\ref{kaminus1a1}) and (\ref{psi1quotient}), we conclude that $\|k(a(k)-1)-a_1\|_{L^2_k} < \infty$.
\end{remark}

\begin{lemma}\label{multiplicationlemma}
The following maps are continuous:
\begin{align*}
(f,g) \mapsto fg:&H^1(\RR) \times H^1(\RR) \to H^1(\RR), 
	\\
(f,g) \mapsto fg:&L^2(\RR) \times H^1(\RR) \to L^2(\RR), 
	\\
(f,g) \mapsto f/g:&H^1(\RR) \times \{g \in 1+H^1(\RR) : \text{\upshape $g(x) \neq 0$ for all $x \in \RR$}\} \to H^1(\RR),
	\\
(f,g) \mapsto f/g:&L^2(\RR) \times \{g \in 1+H^1(\RR) : \text{\upshape $g(x) \neq 0$ for all $x \in \RR$}\} \to L^2(\RR).
\end{align*}
\end{lemma}
\begin{proof}
The assertions follow from standard estimates.
\end{proof}

We are now ready to prove that $r \in H^{1,1}(\RR)$ and to establish the continuity of the map $u_0 \mapsto r$. 

\begin{lemma}\label{rcontinuitylemma}
Given $u_0 \in H^{1,1}(\RR_+)$ and $q \in \RR$, define the reflection coefficient $r(k)$ by the formula \eqref{rdef}.
If $\Delta(k) \neq 0$ for all $k \in \RR$, then $r \in H^{1,1}(\RR)$. Moreover, the map $u_0 \mapsto r:H^{1,1}(\RR_+) \cap \{u_0 \, | \, \text{$\Delta(k) \neq 0$ for all $k \in \RR$}\} \to H^{1,1}(\RR)$ is continuous.
\end{lemma}
\begin{proof}
We have
\begin{align}\label{rdef2}
r(k) = \frac{F(k)}{\Delta_a(k)}, \qquad k \in \RR,
\end{align}
where
$$F(k) \coloneqq \overline{b(k)a(-k)}+\frac{2k+iq}{2k-iq}\overline{a(k)b(-k)}$$
and 
\begin{align}\label{Deltaadef}
\Delta_a(k) \coloneqq \frac{\Delta(k)}{2k-iq} = a(k)\overline{a(-\bar{k})}+\lambda \frac{2k+iq}{2k-iq} b(k)\overline{b(-\bar{k})}.
\end{align}

By Lemma \ref{multiplicationlemma}, multiplication $(f,g) \mapsto fg:H^{1} \times H^{1} \to H^{1}$ is continuous. Since $\frac{2k+iq}{2k-iq} \in 1+ H^1(\RR)$, we see that 
$$f \mapsto \frac{2k+iq}{2k-iq}f: H^1(\RR) \to H^1(\RR)$$
is continuous. Using also Lemma \ref{continuitylemma1}, it follows that the maps
\begin{align}\label{u0toDeltaacontinuous}
& u_0 \mapsto F:H^{1,1}(\RR_+) \to H^{1}(\RR)
\quad \text{and}\quad u_0 \mapsto \Delta_a - 1:H^{1,1}(\RR_+) \to H^{1}(\RR)
\end{align}
are continuous. Hence, using Lemma \ref{multiplicationlemma} again, we conclude that
$$u_0 \mapsto r = \frac{F}{\Delta_a}:H^{1,1}(\RR_+) \cap \{u_0 \, | \,\text{$\Delta \neq 0$ on $\RR$}\} \to H^{1}(\RR)$$ 
is continuous.
It remains to show that
$$u_0 \mapsto kr(k):H^{1,1}(\RR_+) \cap \{u_0 \, | \,\text{$\Delta \neq 0$ on $\RR$}\} \to L^2(\RR)$$ 
is continuous. Since, by Lemma \ref{multiplicationlemma} and (\ref{u0toDeltaacontinuous}), the map
$$(f, u_0) \mapsto \frac{f}{\Delta_a}:L^2(\RR)  \times \Big(H^{1,1}(\RR_+) \cap \{u_0 \, | \,\text{$\Delta \neq 0$ on $\RR$}\}\Big) \to L^2(\RR)$$
is continuous, it is enough to show that
$$u_0 \mapsto kF(k):H^{1,1}(\RR_+) \cap \{u_0 \, | \,\text{$\Delta \neq 0$ on $\RR$}\} \to L^2(\RR)$$ 
is continuous. This will follow if we can show that
\begin{align}\label{u0toF}
u_0 \mapsto (2k-iq)F(k):H^{1,1}(\RR_+) \cap \{u_0 \, | \,\text{$\Delta \neq 0$ on $\RR$}\} \to L^2(\RR)
\end{align}
is continuous. Recall that $a-1 \in H^1(\RR)$ by Lemma \ref{continuitylemma1} and $kb(k) - b_1 \in L^2(\RR)$ by Lemma \ref{specfcnlemma}. Thus there exist functions $f_a \in H^1(\RR)$ and $f_b \in L^2(\RR)$ such that
$$a(k) = 1 + f_a(k), \qquad b(k) = \frac{b_1 + f_b(k)}{k}.$$
Using these expressions for $a$ and $b$, we can write
\begin{align*}
(2k-iq)F(k) = & -2\Big(\bar{b}_1(\overline{f_a(k)} - \overline{f_a(-k)}) - \overline{a(-k)}\overline{f_b(k)} + \overline{a(k)}\overline{f_b(-k)}\Big)
	\\
&- i q\Big(\overline{a(-k)}\overline{b(k)} - \overline{a(k)}\overline{b(-k)} \Big).
\end{align*}
According to Lemma \ref{continuitylemma1} and Lemma \ref{continuitylemma}, $b_1\in \CC$, $f_a \in H^1(\RR)$, $f_b \in L^2(\RR)$, $a \in 1 + H^1(\RR)$, $b \in H^1(\RR)$ depend continuously on $u_0 \in H^{1,1}(\RR_+)$. Thus, using Lemma \ref{multiplicationlemma}, the continuity of the map in (\ref{u0toF}) follows. 
\end{proof}

\subsection{Proof of Proposition \ref{rDeltaprop}}

It follows from Lemma \ref{ablemma} and the definition (\ref{Deltadef}) that $\Delta(k)$ is continuous for $\im k \geq 0$ and analytic for $\im k > 0$.
The symmetry (\ref{Deltasymm}) is an immediate consequence of (\ref{Deltadef}). Evaluating (\ref{Deltadef}) at $k = 0$ and using the unit determinant relation (\ref{a2b21}), we find $\Delta(0)=-iq$.
The asymptotic formula (\ref{Deltalargek}) follows from (\ref{ablargek}).
This proves (\ref{rDeltapropitem1})--(\ref{rDeltapropitem4}) and assertion (\ref{rDeltapropitem5}) was proved in Lemma \ref{rcontinuitylemma}.

Let $\Delta_a(k) = \Delta(k)/(2k-iq)$ be the function defined in (\ref{Deltaadef}). A straightforward calculation using (\ref{a2b21}) shows that
\begin{equation}\label{rDeltaarelation}
1-\lambda |r(k)|^2= |\Delta_a(k)|^{-2}, \qquad k \in \RR.
\end{equation} 
Note that $\Delta_a$ is continuous on $\RR$.
Hence, equation (\ref{rDeltaarelation}) implies that $|r(k)|<1$ for all $k \in \RR$ if $\lambda = 1$. This proves (\ref{rDeltapropitema}). To prove (\ref{rDeltapropitemb})--(\ref{rDeltapropitemd}), we employ (\ref{rdef}) to eliminate $\overline{b(-\bar{k})}$ from the right-hand side of (\ref{Deltaadef}) and use the unit determinant relation (\ref{a2b21}) to see that
\begin{equation}\label{Deltaasimplified}
\Delta_a(k) = \frac{\overline{a(-k)}}{\overline{a(k)}} \frac{1}{1 - \lambda \frac{b(k)}{\overline{a(k)}} r(k)} , \qquad k \in \RR.
\end{equation}


Suppose now that $\lambda = 1$. Then equation (\ref{a2b21}) implies that $a(k) \neq 0$ for all $k \in \RR$, and thus (\ref{Deltaasimplified}) shows that $\Delta_a(k) \neq 0$ for all $k \in \RR$. Let $Z_{\Delta_a}$ and $P_{\Delta_a}$ denote the number of zeros and poles of $\Delta_a$ in $\CC_+$ counted with multiplicity, and let $Z_a$ denote the number of zeros of $a(k)$ in $\CC_+$ counted with multiplicity.
Recalling the asymptotic formulas (\ref{Deltalargek}) and (\ref{ablargek}) for $\Delta$ and $a$, the argument principle applied to a large semicircle enclosing the upper half-plane yields
\begin{align}\label{ZPDeltaa}
Z_{\Delta_a} - P_{\Delta_a} = \frac{\log \Delta_a(k)}{2\pi i} \bigg|_{k=-\infty}^{+\infty}, \qquad Z_a = \frac{\log a(k)}{2\pi i} \bigg|_{k=-\infty}^{+\infty},
\end{align}
where $\frac{\log f(k)}{2\pi i}|_{k=-\infty}^{+\infty}$ denotes the winding number of $f(k)$ around the origin as $k$ traverses the real axis from $-\infty$ to $+\infty$.
On the other hand, by (\ref{a2b21}), 
$$\bigg|\frac{b(k)}{a(k)}\bigg|^2 = 1 - \frac{1}{|a(k)|^2} < 1, \qquad k \in \RR.$$
Hence, using also that $|r|<1$ on $\RR$, 
$$\bigg|\frac{b(k)}{\overline{a(k)}} r(k)\bigg| < 1, \qquad k \in \RR.$$
In particular, the second factor on the right-hand side of (\ref{Deltaasimplified}) winds zero times around the origin as $k$ traverses $\RR$.
Consequently, it transpires from (\ref{ZPDeltaa}) that
\begin{align*}
Z_{\Delta_a} - P_{\Delta_a}
& = \frac{\log \overline{a(-k)} - \log \overline{a(k)}}{2\pi i} \bigg|_{k=-\infty}^{+\infty}
= 2\frac{\log a(k)}{2\pi i} \bigg|_{k=-\infty}^{+\infty}
 = 2 Z_{a}.
\end{align*}
However, an argument which involves viewing the zeros of $a$ as eigenvalues of a self-adjoint operator obtained by extending (\ref{xparteqn0}) from the half-line $x \geq 0$ to the real line, implies that $Z_a = 0$ (see \cite[Section 3]{L2016} for a detailed proof).
This establishes assertion (\ref{rDeltapropiteme}), and since $\Delta$ has no poles in $\CC_+$, we arrive at
$$Z_{\Delta} = \begin{cases} Z_{\Delta_a} - P_{\Delta_a} + 1 = 1, & q > 0, \\
Z_{\Delta_a} - P_{\Delta_a} = 0, & q < 0,
\end{cases}
$$
where $Z_\Delta$ denotes the number of zeros of $\Delta$ in $\CC_+$ counted with multiplicity. The symmetry (\ref{Deltasymm}) implies that if $\Delta$ has only one zero in $\CC_+$, then it must be pure imaginary. 
Since we already saw that $\Delta_a(k) \neq 0$ for all $k \in \RR$, the assertions (\ref{rDeltapropitemb})--(\ref{rDeltapropitemd}) about the zeros of $\Delta$ follow.

\section{Construction of $m$ from $u$}\label{mfromusec}
The purpose of this section is to prove Proposition \ref{mfromuprop}, which shows that if there is a global Schwartz class solution $u(x,t)$ of the Robin IBVP, then the RH problem of Theorem \ref{linearizableth} has a unique solution $m(x,t,k)$, and $u(x,t)$ can be recovered from $m(x,t,k)$ via (\ref{recoveru}). In addition to motivating the structure of the RH problem of Theorem \ref{linearizableth}, this proposition will be used to establish uniqueness in Section \ref{Schwartzsec}.

\begin{proposition}\label{mfromuprop}
Suppose $\lambda = 1$ or $\lambda = -1$. Let $u_0 \in \mathcal{S}(\RR_+)$ be such that $q \coloneqq -u_{0x}(0)/u_0(0) \in \RR$. If $\lambda = -1$, suppose that Assumption \ref{zerosassumption} holds.
Suppose that there exists a global Schwartz class solution $u(x,t)$ of the Robin IBVP for NLS with initial data $u_0$.
Then the RH problem of Theorem \ref{linearizableth} has a unique solution $m(x,t,k)$ and, for each $(x,t) \in [0, \infty) \times [0, \infty)$, $u(x,t)$ is given in terms of $m(x,t,k)$ by (\ref{recoveru}).
\end{proposition}

The remainder of this section is devoted to the proof of Proposition \ref{mfromuprop}. The proof relies on the theory of linearizable boundary conditions within the framework of the Unified Transform Method of Fokas \cite{F1997} (see also \cite{F2002, FIS2005, IS2013}).

\subsection{The solution $M$} 
The NLS equation (\ref{NLS}) is the compatibility condition of the Lax pair equations
\begin{subequations}\label{mulax}
\begin{align}
\mu_x+ik [\sigma_3, \mu] = \mathsf{U} \mu, \label{xparteqn}\\
\mu_t+2ik^2 [\sigma_3, \mu] = \mathsf{V} \mu, \label{tparteqn}
\end{align}
\end{subequations}
where $k \in \CC$ is the spectral parameter, $\mu(x,t,k)$ is a $2 \times 2$-matrix valued eigenfunction, and $\mathsf{U}, \mathsf{V}$ are defined by
\begin{align*}
\mathsf{U}(x,t) = \begin{pmatrix}
0 & u \\ 
\lambda \bar{u} & 0
\end{pmatrix}, \qquad
\mathsf{V}(x,t,k)= \begin{pmatrix}
-i\lambda|u|^2 & 2ku+iu_x\\ 
2\lambda k\bar{u}-i\lambda \bar{u}_x & i\lambda|u|^2
\end{pmatrix}.
\end{align*}

Suppose $u(x,t)$ is a global Schwartz class solution of the Robin IBVP for NLS with initial data $u_0 \in \mathcal{S}(\RR_+)$.
Fix a final time $T \in (0,\infty)$ and define three solutions $\mu_j$, $j = 1,2,3$, of (\ref{mulax}) as the unique solutions of the integral equations
\begin{equation}\label{inteqn}
\mu_j(x,t,k)=I+\int_{(x_j,t_j)}^{(x,t)}e^{-i(kx+2k^2t)\hat{\sigma}_3}W_j(x',t',k),
\end{equation}
where $(x_1, t_1) = (0,T)$, $(x_2, t_2) = (0,0)$, $(x_3, t_3) = (\infty,t)$, and the exact 1-form $W_j$ is defined by
$$W_j = e^{i(kx+2k^2t)\hat{\sigma}_3}(\mathsf{U}\mu_j dx+\mathsf{V}\mu_j dt).$$
It follows from (\ref{mulax}) that the functions $\mu_j$ are related as follows:
\begin{equation}\label{preRH}
\begin{cases}
\mu_3(x,t,k)=\mu_2(x,t,k)e^{-i\theta \hat{\sigma}_3} \mu_3(0,0,k),\\
\mu_1(x,t,k)=\mu_2(x,t,k)e^{-i\theta \hat{\sigma}_3} \left[e^{2ik^2 T \hat{\sigma}_3}\mu_2(0,T,k)\right]^{-1}.
\end{cases}
\end{equation}
Define the spectral functions $s(k)$ and $S(k;T)$ by
\begin{equation}
s(k):=\mu_3(0,0,k), \qquad S(k; T):=[e^{2ik^2T\hat{\sigma}_3}\mu_2(0,T,k)]^{-1}.
\end{equation}

Let $D_j$, $j =1, \dots, 4$, denote the four open quadrants of the complex plane, see the left half of Figure \ref{Sigma.pdf}.
The properties of the $\mu_j$ imply that $s$ and $S$ can be expressed as
\begin{equation}\label{abABdef}
s(k)=\begin{pmatrix}
\overline{a(\bar{k})} & b(k)\\ 
\lambda\overline{b(\bar{k})} & a(k)
\end{pmatrix}, \qquad
S(k;T)=\begin{pmatrix}
\overline{A(\bar{k};T)} & B(k;T)\\ 
\lambda\overline{B(\bar{k};T)} & A(k;T)
\end{pmatrix},
\end{equation}
where the functions $a(k)$ and $b(k)$ are defined for $\im k \geq 0$ and analytic for $\im k > 0$, whereas $A(k;T)$ and $B(k;T)$ are entire functions of $k$ which are bounded in $\bar{D}_1 \cup \bar{D}_3$. 
Define $d(k;T)$ by
\begin{equation}\label{ddef}
d(k; T) = a(k)\overline{A(\bar{k};T)}-\lambda b(k)\overline{B(\bar{k};T)}, \qquad k \in \bar{D}_2.
\end{equation}
As $k\to \infty$,
\begin{subequations}\label{abABlargek}
\begin{align}\label{ablargek2}
& a(k) = 1 + O(k^{-1}), \quad b(k) = O(k^{-1}), && k \to \infty, ~ \im k \geq 0,
	\\\label{ABlargek}
& A(k; T) = 1 + O(k^{-1}), \quad B(k; T) = O(k^{-1}), && k \to \infty, ~ k \in \bar{D}_1 \cup \bar{D}_3.
\end{align}
\end{subequations}
It follows that any possible zeros of $a(k)$ in the upper half-plane and of $d(k)$ in $\bar{D}_2$ are contained in a disk of finite radius $\{k \in \CC : |k| < R\}$ for some $R > 0$. 
Thus we may define deformed quadrants $\mathcal{D}_j$, $j = 1,\dots, 4$, such that $a$ and $d$ have no zeros in $\bar{\mathcal{D}}_2$ (see the right half of Figure \ref{Sigma.pdf}).
Define the sectionally meromorphic function $M(x,t,k) \equiv M(x,t,k; T)$ by
\begin{align}\label{Mdef}
M = \begin{cases} 
(\frac{[\mu_2]_1}{a}, [\mu_3]_2), \quad & k \in \mathcal{D}_1, \\
(\frac{[\mu_1]_1}{d}, [\mu_3]_2), & k \in \mathcal{D}_2, \\
([\mu_3]_1, \frac{[\mu_1]_2}{d^*}), & k \in \mathcal{D}_3, \\
([\mu_3]_1, \frac{[\mu_2]_2}{a^*}), & k \in \mathcal{D}_4.
\end{cases}
\end{align}
\begin{figure}
\begin{center}
\begin{overpic}[width=.4\textwidth]{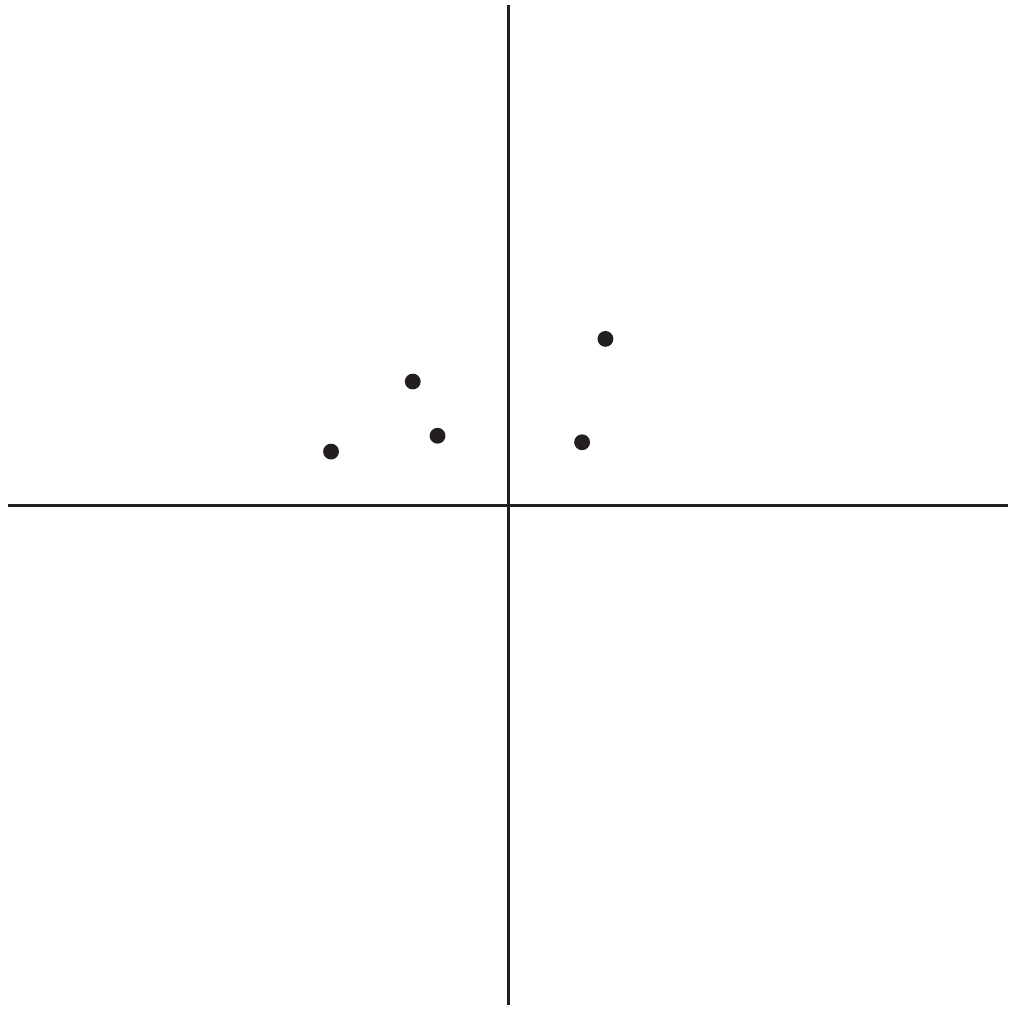}
      \put(78,80){\small $D_1$} 
      \put(16,80){\small $D_2$} 
      \put(16,16){\small $D_3$} 
      \put(78,16){\small $D_4$} 
     \end{overpic}\qquad
\begin{overpic}[width=.4\textwidth]{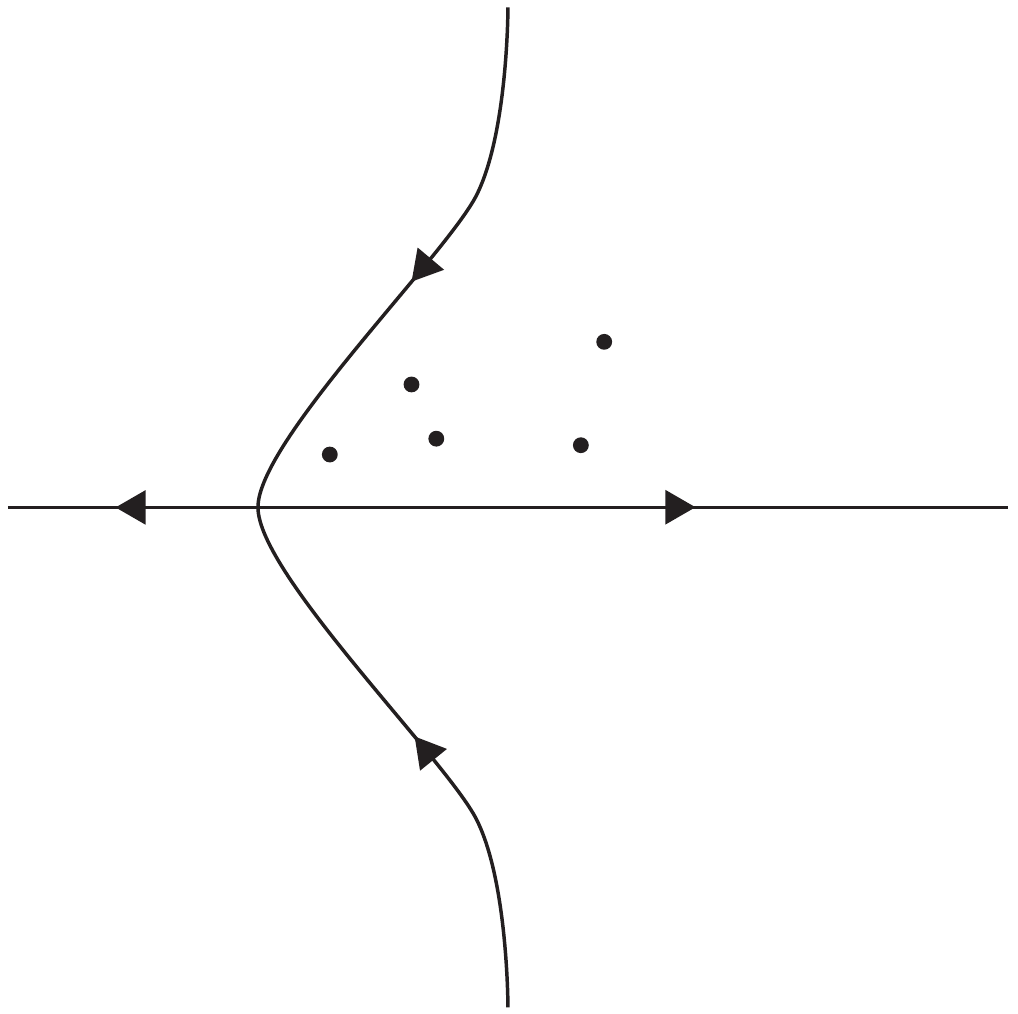}
      \put(78,80){\small $\mathcal{D}_1$} 
      \put(16,80){\small $\mathcal{D}_2$} 
      \put(16,16){\small $\mathcal{D}_3$} 
      \put(78,16){\small $\mathcal{D}_4$} 
       \put(65, 43){\small $\Sigma_1$} 
      \put(45, 71.4){\small $\Sigma_2$} 
      \put(10, 43){\small $\Sigma_3$} 
      \put(45, 25){\small $\Sigma_4$} 
     \end{overpic}
    \caption{\label{Sigma.pdf} The four quadrants $D_1, \dots, D_4$ (left) and the four deformed quadrants $\mathcal{D}_1, \dots, \mathcal{D}_4$ separated by the contour $\Sigma = \cup_{j=1}^4 \Sigma_j$ (right). The dots represent zeros of $a$ and $d$.}
     \end{center}
\end{figure}
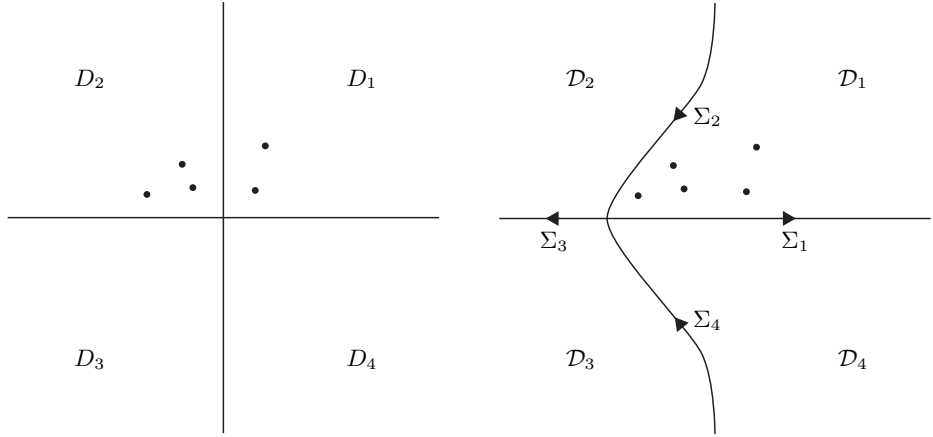
Let $\Sigma = \cup_{j=1}^4 \Sigma_j$, where $\Sigma_j = \bar{\mathcal{D}}_j \cap \bar{\mathcal{D}}_{j-1}$ and $\mathcal{D}_0 \equiv \mathcal{D}_4$, denote the contour separating the domains $\mathcal{D}_j$ (see Figure \ref{Sigma.pdf}) and define the jump matrix $J(x,t,k)$ for $k \in \Sigma$ by
\begin{equation}\label{Jdef}
J(x,t,k)= \begin{cases}
\begin{pmatrix}
1-\lambda \abs{\gamma}^2 & \gamma e^{-2i\theta} \\ 
-\lambda \gamma^* e^{2i\theta} & 1
\end{pmatrix}, & \quad k \in \Sigma_1, 
	\\
\begin{pmatrix}
1 & 0\\ 
-\Gamma e^{2i\theta} & 1
\end{pmatrix}, & \quad k \in \Sigma_2,
	\\
\begin{pmatrix}
1 & -(\gamma - \lambda \Gamma^*)e^{-2i\theta}\\ 
\lambda(\gamma^* -\lambda \Gamma)e^{2i\theta} & 1-\lambda|\gamma^* - \lambda\Gamma|^2
\end{pmatrix}, & \quad k \in \Sigma_3,
	\\
\begin{pmatrix}
1 & \lambda \Gamma^* e^{-2i\theta} \\ 
0 & 1
\end{pmatrix}, & \quad k \in \Sigma_4,
\end{cases}
\end{equation}
where
\begin{align}\label{gammaGammadef}
 \gamma(k) := \frac{b(k)}{\overline{a(\bar{k})}}, \qquad \Gamma(k;T) := \frac{\lambda \overline{B(\bar{k};T)}}{a(k)d(k;T)}.
\end{align}
Letting $p = \cap_{j=1}^4 \bar{\mathcal{D}}_j$ denote the point at which the horizontal and vertical branches of $\Sigma$ intersect, we have the following lemma.

\begin{lemma}[RH problem for $M$]\label{Mlemma}
Suppose that $a(k)$ has no zeros on $\RR$ and only simple zeros in $\CC_+$. For each $(x,t) \in [0, \infty) \times [0, \infty)$, $M(x,t,k)$ is the unique solution of the following RH problem:
\begin{enumerate}[$(a)$]
\item $M(x,t,\cdot): \mathbb{C} \setminus (\Sigma \cup \{k_j, \bar{k}_j\}_1^N) \to \mathbb{C}^{2 \times 2}$ is analytic.

\item The boundary values of $M(x,t,k)$ as $k$ approaches $\Sigma \setminus \{p\}$ from the left $(+)$ and right $(-)$ exist, are continuous on $\Sigma \setminus \{p\}$, and satisfy
\begin{align}\label{Mjump}
  M_+(x,t,k) = M_-(x, t, k) J(x, t, k), \qquad k \in \Sigma \setminus \{p\}.
\end{align}

\item $M(x,t,k) = I + O(k^{-1})$ as $k \to \infty$.

\item $M(x,t,k) = O(1)$ as $k \to p$. 

\item Let $\{k_j\}_1^N$ denote the (necessarily finitely many) simple zeros of $a(k)$ in $\CC_+$. The first column of $M$ has at most a simple pole at each $k_j \in \CC_+$, the second column of $M$ has at most a simple pole at each $\bar{k}_j \in \CC_-$, and the following residue conditions hold for $j = 1, \dots, N$:
\begin{subequations}\label{Mresidues}
\begin{align}\label{Mresiduesa}
& \underset{k = k_j}{\Res} [M(x,t,k)]_1 =  \frac{1}{\dot{a}(k_j)b(k_j)} e^{2i\theta(x,t,k_j)} [M(x,t,k_j)]_2,
	\\\label{Mresiduesb}
& \underset{k=\bar{k}_j}{\Res} [M(x,t,k)]_2 =  \frac{\lambda}{\overline{\dot{a}(k_j)b(k_j)}} e^{-2i\theta(x,t,\bar{k}_j)}  [M(x,t,\bar{k}_j)]_1.	
\end{align}
\end{subequations}

\end{enumerate}
Moreover, 
\begin{align}\label{recoverufromM}
u(x,t) = 2i\lim_{k \to \infty} k(M(x,t,k))_{12}, \qquad x \geq 0, ~ t \geq 0.
\end{align}
\end{lemma}
\begin{proof}
The lemma is standard in the context of the unified transform method, see \cite{F2002, FIS2005}. In \cite{F2002, FIS2005}, the function $M$ is defined using the standard quadrants $D_j$ instead of the deformed quadrants $\mathcal{D}_j$, and residue conditions are included at the possible zeros of $a$ and $d$ in $D_2$. For the present purposes, it is more convenient to instead handle the possible zeros of $d$ by replacing the quadrants $D_j$ by the deformed quadrants $\mathcal{D}_j$. 
\end{proof}

\subsection{The solution $\tilde{M}$} 
The formulation of the RH problem for $M$ involves the spectral function $\Gamma(k;T)$ which is defined in terms of the unknown boundary values $u(0,t)$ and $u_x(0,t)$. In order to arrive at an effective solution of the problem, we seek to replace $\Gamma(k; T)$ by another function $\tilde{\Gamma}(k)$ which is defined in terms of the initial data alone. 

The definition of $\tilde{\Gamma}$ can be motivated as follows. The spectral functions $a,b,A,B$ are not independent but satisfy the global relation (see \cite{FIS2005})
\begin{equation}\label{globalreln}
A(k;T)b(k)-a(k)B(k;T)=c(k;T)e^{4ik^2 T}, \qquad \im k \geq 0,
\end{equation}
where the function $c(k;T)$ is analytic for $\im k > 0$ and satisfies $c(k,T)=O(1/k)$ as $k \to \infty$, $\im k \geq 0$. The global relation implies that
\begin{align}\label{BAfromGR}
\frac{B(k;T)}{A(k;T)} = \frac{b(k)}{a(k)} - \frac{c(k;T)}{a(k)A(k;T)}e^{4ik^2 T}.
\end{align}
On the other hand, we can rewrite the definition (\ref{gammaGammadef}) of $\Gamma(k;T)$ as
\begin{align}\label{Gammalambdaaa}
\Gamma(k;T) = \frac{\lambda}{a(k)(a(k)\frac{A^*(k;T)}{B^*(k;T)} - \lambda b(k))}, \qquad k \in \bar{D}_2.
\end{align}
The Robin boundary condition $u + qu_x = 0$ implies that the matrix $\mathsf{V}$ in \eqref{tparteqn} obeys the symmetry $\mathsf{V}(0,t,-k) = \mathcal{N}(k) \mathsf{V}(0,t,k) \mathcal{N}(k)^{-1}$, where $\mathcal{N}(k) = \diag(2k+iq, -2k+iq)$ (see \cite{F2002}). Hence, $S(-k) = \mathcal{N}(k) S(k) \mathcal{N}(k)^{-1}$ which yields the following symmetries for $A$ and $B$:
\begin{align}\label{ABsymm}
A(k;T) = A(-k;T), \qquad B(k;T) = -\frac{2k-iq}{2k+iq}B(-k;T).
\end{align}
Using these symmetries, we can write (\ref{Gammalambdaaa}) as
$$\Gamma(k;T) = -\frac{\lambda}{a(k)\big[a(k)\frac{(2k-iq)A^*(-k;T)}{(2k+iq)B^*(-k;T)} + \lambda b(k)\big]}, \qquad k \in \bar{D}_2.$$
Employing (\ref{BAfromGR}), this becomes
\begin{align}\label{GammaaAc}
\Gamma(k;T) = -\frac{\lambda}{a(k)\big[a(k)\frac{2k-iq}{2k+iq}\big(\frac{b^*(-k)}{a^*(-k)} - \frac{c^*(-k;T)}{a^*(-k)A^*(-k;T)}e^{-4ik^2 T}\big)^{-1} + \lambda b(k)\big]}, \qquad k \in \bar{D}_2.
\end{align}

\begin{remark}
The solution $u(x,t)$ evaluated at some time $t < T$ should not depend on $T$.
This suggests that we define the new RH problem for $\tilde{M}$ by replacing $\Gamma$ with the $T$-independent function $\lim_{T \to \infty} \Gamma(k;T)$, if the limit exists. Since $e^{-4ik^2 T}$ has exponential decay as $T \to \infty$ for $k \in D_2$, equation (\ref{GammaaAc}) suggests that
\begin{align}\label{tildeGammamotivation}
\lim_{T \to \infty} \Gamma(k;T) = -\frac{\lambda}{a(k)(a(k)\frac{2k-iq}{2k+iq}\frac{a^*(-k)}{b^*(-k)} + \lambda b(k))}, \qquad k \in D_2.
\end{align}
Regardless of whether the large $T$ limit of $\Gamma(k;T)$ exists or not, we will take the right-hand side of (\ref{tildeGammamotivation}) as our definition of $\tilde{\Gamma}(k)$.
\end{remark}
Inspired by (\ref{tildeGammamotivation}), we define $\tilde{\Gamma}(k)$ by
\begin{align}\label{tildeGammadef}
\tilde{\Gamma}(k)=-\frac{\lambda \overline{b(-\bar{k})}}{a(k)}\frac{2k+iq}{\Delta(k)},\qquad \im k \geq 0,
\end{align}
where $\Delta$ is the function defined in (\ref{Deltadef}). 
Deforming the vertical branch of $\Sigma$ further into the left half-plane if necessary, we may assume that $\Delta$ has no zeros in $\mathcal{D}_2$.
We define the function $G(k;T)$ by
\begin{align}\label{Gdef}
  G_1 = I, \quad G_2 = \begin{pmatrix} 1 & 0 \\ (\tilde{\Gamma} - \Gamma)e^{2i\theta} & 1 \end{pmatrix}, \quad G_3 = \begin{pmatrix} 1 & \lambda (\tilde{\Gamma}^* - \Gamma^*)e^{-2i\theta} \\ 0 & 1 \end{pmatrix}, \quad G_4 = I,
\end{align}
where $G_j$ denotes the restriction of $G$ to  $\mathcal{D}_j$ for $j = 1, \dots, 4$.
Introducing $\tilde{M}(x,t,k)$ by
\begin{align}\label{tildeMdef}
\tilde{M} = MG,
\end{align}
it is easy to verify that $\tilde{M}$ satisfies the same jump relations on $\Sigma$ as $M$ except that $\Gamma(k;T)$ is replaced by $\tilde{\Gamma}(k)$. In fact, we have the following lemma.

\begin{lemma}[RH problem for $\tilde{M}$]\label{Mtildelemma}
Suppose that $a(k)$ has no zeros on $\RR$ only simple zeros in $\CC_+$. 
For each $(x,t) \in [0, \infty)\times [0, \infty)$, $\tilde{M}(x,t,k)$ satisfies the following RH problem:
\begin{enumerate}[$(a)$]
\item $\tilde{M}(x,t,\cdot): \mathbb{C} \setminus (\Sigma \cup \{k_j, \bar{k}_j\}_1^N) \to \mathbb{C}^{2 \times 2}$ is analytic.

\item The boundary values of $\tilde{M}(x,t,k)$ as $k$ approaches $\Sigma \setminus \{p\}$ from the left $(+)$ and right $(-)$ exist, are continuous on $\Sigma \setminus \{p\}$, and satisfy
\begin{align}\label{Mtildejump}
  \tilde{M}_+(x,t,k) = \tilde{M}_-(x, t, k) \tilde{J}(x, t, k), \qquad k \in \Sigma \setminus \{p\}.
\end{align}
where the jump matrix $\tilde{J}$ is defined by replacing $\Gamma(k;T)$ with $\tilde{\Gamma}(k)$ in the definition (\ref{Jdef}) of $J$.

\item $\tilde{M}(x,t,k) = I + O(k^{-1})$ as $k \to \infty$.

\item $\tilde{M}(x,t,k) = O(1)$ as $k \to p$.

\item Let $\{k_j\}_1^N$ denote the (necessarily finitely many) simple zeros of $a(k)$ in $\CC_+$. The first column of $\tilde{M}$ has at most a simple pole at each $k_j \in \CC_+$, the second column of $\tilde{M}$ has at most a simple pole each $\bar{k}_j \in \CC_-$, and the following residue conditions hold for $j = 1, \dots, N$:
\begin{subequations}\label{Mtilderesidues}
\begin{align}\label{Mtilderesiduesa}
& \underset{k = k_j}{\Res} [\tilde{M}(x,t,k)]_1 =  \frac{1}{\dot{a}(k_j)b(k_j)} e^{2i\theta(x,t,k_j)} [\tilde{M}(x,t,k_j)]_2,
	\\\label{Mtilderesiduesb}
& \underset{k=\bar{k}_j}{\Res} [\tilde{M}(x,t,k)]_2 =  \frac{\lambda}{\overline{\dot{a}(k_j)b(k_j)}} e^{-2i\theta(x,t,\bar{k}_j)}  [\tilde{M}(x,t,\bar{k}_j)]_1.	
\end{align}
\end{subequations}

\end{enumerate}
\end{lemma}
\begin{proof}
All properties except the normalization condition $\tilde{M} = I + O(k^{-1})$ follow immediately from Lemma \ref{Mlemma} and the definition (\ref{tildeMdef}) of $\tilde{M}$. To prove that $\tilde{M} = I + O(k^{-1})$ as $k \to \infty$, it is sufficient to show that 
\begin{align}\label{Gasymptotics}
  G_j(k;T) = I + O\big(k^{-1}\big)\quad \text{as $k \to \infty$, $k \in \bar{D}_j$, $j = 1,\dots, 4$}.
\end{align}
The large $k$ estimates (\ref{abABlargek}) of $a$ and $b$ together with the expression  (\ref{tildeGammadef}) for $\tilde{\Gamma}$ implies that $\tilde{\Gamma}(k) = O(k^{-1})$ as $k \to \infty$ in $\bar{D}_2$. Moreover, letting
$$X(k;T) := \frac{c^*(-k;T)}{a^*(-k)A^*(-k;T)}e^{-4ik^2 T},$$
and recalling that $A(k;T) = 1 + O(k^{-1})$ and $c(k;T) = O(k^{-1})$ as $k \to \infty$ in $\bar{D}_1$, we see that $X(k;T) = O(k^{-1})$ as $k \to \infty$ in $\bar{D}_2$.
It then follows from the expression (\ref{GammaaAc}) for $\Gamma(k;T)$ that
\begin{align*}
\Gamma(k;T) & = -\frac{\lambda \big(\frac{b^*(-k)}{a^*(-k)} - X\big)}{a(k)\big[a(k)\frac{2k-iq}{2k+iq} + \lambda b(k) \big(\frac{b^*(-k)}{a^*(-k)} - X\big)\big]}
= -\frac{\lambda \big(O(k^{-1}) - X\big)}{1 + O(k^{-1})}
	\\
& = \lambda X(k;T) + O(k^{-1}), \qquad k \to \infty, ~ k \in \bar{D}_2.
\end{align*}
Since $e^{2i\theta -4ik^2 T} = e^{2ikx + 4ik^2(t-T)}$ is bounded for $k \in \bar{D}_2$, we obtain the estimate
$$(\tilde{\Gamma} - \Gamma)e^{2i\theta} = O\big(k^{-1}\big) \quad \text{as $k \to \infty$, $k \in \bar{D}_2$},$$
from which (\ref{Gasymptotics}) follows. 
\end{proof}

\subsection{The solution $m$ and proof of Proposition \ref{mfromuprop}}\label{commonzeroissuesection}
Define $m(x,t,k)$ by
\begin{align}\label{mdef}
m(x,k,t)= \begin{cases}
    \tilde{M}(x,t,k), & k \in \mathcal{D}_2 \cup \mathcal{D}_3, \\
    \tilde{M}(x,t,k)\begin{pmatrix}
1 & 0\\ 
\tilde{\Gamma}(k)e^{2i\theta} & 1
\end{pmatrix}, & k \in \mathcal{D}_1, \\
    \tilde{M}(x,t,k)\begin{pmatrix}
1 & \lambda\overline{\tilde{\Gamma}(\bar{k})}e^{-2i\theta}\\ 
0 & 1
\end{pmatrix}, & k \in \mathcal{D}_4.
\end{cases}
\end{align}

We will show that $m(x,t,k)$ is the unique solution of the RH problem of Theorem \ref{linearizableth} and that $u(x,t)$ satisfies (\ref{recoveru}). 
It is clear that $m$ is analytic in $\mathcal{D}_2 \cup \mathcal{D}_3$, and also in $\mathcal{D}_1 \cup \mathcal{D}_4$ away from the possible zeros of $a$ and $\Delta$ and their complex conjugates. 
Let $\{k_j\}_1^N$ and $\{\xi_j\}_1^M$ be the sets of zeros of $a(k)$ and $\Delta(k)$ in $\CC_+$, respectively. By assumption, these sets are disjoint. 

If $a(k)$ has no zeros on $\RR$ and only simple zeros in $\CC_+$, then the analyticity of $m$ at the $k_j$ as well as the residue conditions (\ref{mresiduesa}) at the $\xi_j$ follow from the relations $[m]_1 = [\tilde{M}]_1 + \tilde{\Gamma} e^{2i\theta} [\tilde{M}]_2$ and $[m]_2 = [\tilde{M}]_2$ which are valid in $\mathcal{D}_1$ together with the relations
$$\underset{k=k_j}{\Res} \tilde{\Gamma}(x,t,k) = -\frac{1}{\dot{a}(k_j)b(k_j)}, \qquad
\underset{k=\xi_j}{\Res} \tilde{\Gamma}(x,t,k) = -\frac{\lambda \overline{b(-\bar{\xi}_j)}}{a(\xi_j)}\frac{2\xi_j+iq}{\dot{\Delta}(\xi_j)}.$$
Since $m$ obeys the symmetries
\begin{align}\label{msymm}
  m_{11}(x,t,k) = \overline{m_{22}(x,t,\bar{k})}, \qquad 
  m_{21}(x,t,k) = \lambda \overline{m_{12}(x,t,\bar{k})},
\end{align}
the analogous statements for $\bar{k}_j$ and $\bar{\xi}_j$ follow by symmetry.

To avoid having to assume that $a(k)$ is nonzero on $\RR$ and that all the zeros of $a(k)$ are simple, we note that the above facts can also be derived directly. Indeed, by (\ref{Mdef}), (\ref{tildeMdef}), and (\ref{mdef}), we have, for $k \in \mathcal{D}_1$,
$$[m]_1 = \frac{[\mu_2]_1}{a(k)} + \tilde{\Gamma} e^{2i\theta} [\mu_3]_2, \qquad [m]_2 = [\mu_3]_2,$$
implying that $[m]_2$ is analytic in $\mathcal{D}_1$ and that the residue conditions (\ref{mresiduesa}) hold at the $\xi_j$. Moreover, using the relation $[\mu_3]_2 = [\mu_2]_1 e^{-2i\theta} b + [\mu_2]_2 a$, we find after simplification that
\begin{equation}\label{m1equation}
[m]_1 = \frac{2k-iq}{\Delta(k)}\overline{a(-\bar{k})}[\mu_2]_1
- \frac{\lambda  (2k+ iq)}{\Delta(k)} \overline{b(-\bar{k})}e^{2i\theta}[\mu_2]_2,
\end{equation}
showing that $[m]_1$ is analytic away from the zeros of $\Delta$. The analogous conclusions in $\mathcal{D}_2$ follow similarly or from the symmetry (\ref{msymm}).

Using that $r = \gamma^* - \lambda\tilde{\Gamma}$, the jump relation (\ref{mjump}) follows from (\ref{Mtildejump}).
Since $e^{2i\theta}$ is bounded in $\bar{D}_1$ and $\tilde{\Gamma}(k) = O(k^{-1})$ as $k \to \infty$, $k \in \bar{D}_1$, by (\ref{ablargek}), we obtain the estimate
$$\tilde{\Gamma} e^{2i\theta} = O\big(k^{-1}\big) \quad \text{as $k \to \infty$, $k \in \bar{D}_1$},$$
which implies that $m = I + O(k^{-1})$ as $k \to \infty$.
Finally, equation (\ref{recoveru}) follows from (\ref{recoverufromM}) and (\ref{tildeMdef}). 
This completes the proof of Proposition \ref{mfromuprop}.

\section{Global Schwartz class solutions}\label{Schwartzsec}
The main result of this section is Proposition \ref{Schwartzprop}, which shows that the conclusion of Theorem \ref{linearizableth} holds whenever the initial data $u_0$ belongs to the Schwartz class. For the proof, we need the following lemma.

\begin{lemma}\label{rSchwartzlemma}
Let $u_0 \in \mathcal{S}(\RR_+)$ and $q \in \RR$. Suppose that $\Delta(k) \neq 0$ for all $k \in \RR$. Then $r \in \mathcal{S}(\RR)$.
\end{lemma}
\begin{proof}
Fix $T > 0$. Since $u_0 \in \mathcal{S}(\RR_+)$, standard arguments show that $r \in C^{\infty}(\RR)$. Moreover, it can be shown that there exist coefficients $\{a_j, b_j, A_j, B_j\}_1^\infty \subset \CC$ such that the spectral functions $a,b,A,B$ defined in \eqref{abABdef} admit the following asymptotic expansions to all orders as $k \to \infty$:
\begin{align*}
& a(k) \sim 1 + \sum_{j=1}^\infty \frac{a_j}{k^j}, \quad b(k) \sim \sum_{j=1}^\infty \frac{b_j}{k^j} \quad \text{ uniformly for $\arg k \in [0, \pi]$},
	\\
& A(k) \sim 1 + \sum_{j=1}^\infty \frac{A_j}{k^j}, \quad B(k) \sim \sum_{j=1}^\infty \frac{B_j}{k^j} \quad \text{ uniformly for $\arg k \in [0, \pi/2]\cup [\pi, 3\pi/2]$};
\end{align*}
furthermore, these expansions can be differentiated termwise any number of times. 
Considering the global relation (\ref{BAfromGR}) along a ray $\arg k \in (0, \pi/2)$, we see that the large $k$-expansions of $B/A$ and $b/a$ agree formally to all orders. Hence, recalling (\ref{ABsymm}), we find that, as formal power series in $k^{-1}$ for $k \in \RR$, 
$$\frac{b(k)}{a(k)} + \frac{(2k-iq)}{(2k+iq)}\frac{b(-k)}{a(-k)}
\sim \frac{b(k)}{a(k)} + \frac{(2k-iq)}{(2k+iq)}\frac{B(-k;T)}{A(-k;T)}
=  \frac{b(k)}{a(k)} - \frac{B(k;T)}{A(k;T)}
\sim 0$$
to all orders, and that this relation can be differentiated any number of times. Substituting this into the definition (\ref{rdef}) of $r(k)$, we conclude that $r(k)$ has rapid decay as $k \to \pm \infty$.
\end{proof}

\begin{proposition}\label{Schwartzprop}
Suppose $\lambda = 1$ or $\lambda = -1$. Let $u_0 \in \mathcal{S}(\RR_+)$ and suppose $u_0'(0) + q u_0(0) = 0$ for some $q \in \RR \setminus \{0\}$. If $\lambda = -1$, then suppose that Assumption \ref{zerosassumption} holds; if $\lambda = 1$ and $q > 0$, then suppose the RH problem of Theorem \ref{linearizableth} has a solution for each $(x,t) \in [0, \infty) \times [0, \infty)$.
Then there exists a unique global Schwartz class solution $u(x,t)$ of the Robin IBVP for NLS with parameter $q$ and initial data $u_0$. Moreover, for each $(x,t) \in [0, \infty) \times [0, \infty)$, the RH problem of Theorem \ref{linearizableth} has a unique solution $m(x,t,k)$, and $u(x,t)$ is given in terms of $m(x,t,k)$ by (\ref{recoveru}).
\end{proposition}
\begin{proof}
By Proposition \ref{rDeltaprop} (if $\lambda = 1$) and by Assumption \ref{zerosassumption} (if $\lambda = -1$), we have $\Delta(k) \neq 0$ for all $k \in \RR$, and hence, by Lemma \ref{rSchwartzlemma}, $r \in \mathcal{S}(\RR)$.
Moreover, for $k \in \RR$, we have
\begin{equation}\label{vplusvdagger}
v+v^\dagger = \begin{pmatrix}2(1-\lambda |r(k)|^2) & (\overline{r}-\lambda \overline{r})e^{-2i\theta} \\ (r-\lambda r)e^{2i\theta} & 2 \end{pmatrix},
\end{equation}
where $v^\dag$ denotes the complex conjugate transpose of the jump matrix $v$. When $\lambda=1$, the matrix in (\ref{vplusvdagger}) is diagonal with strictly positive entries (recall that $|r|<1$ on $\RR$ if $\lambda = 1$ by Proposition \ref{rDeltaprop}), and so it is positive definite. If $\lambda=-1$, then $v=v^\dag$ is positive definite by Sylvester's criterion. Hence, if $\Delta(k)$ has no zeros in $\CC_+$, then the existence of a solution of the RH problem of Theorem \ref{linearizableth} follows from the existence of a vanishing lemma, see \cite{Z1989}.

If $\Delta(k)$ has a finite number of simple zeros $\{\xi_j\}_1^M$ in $\CC_+$, then the RH problem for $m$ can be transformed into a regularized RH problem (whose solution we denote by $m_{\reg}$) together with a system of algebraic equations, see \cite[Proposition 2.4]{FI1996}. 
The solution $m_{\reg}$ exists by the above vanishing lemma argument. In the focusing case, there always exists a unique solution of the associated algebraic system, see \cite[Proposition 2.4]{FI1996}. In the defocusing case, by Proposition \ref{rDeltaprop}, $\Delta$ has a zero only if $q > 0$ and in this case the solution $m$ exists by assumption.
Thus, in either case, the solution $m$ of the RH problem of Theorem \ref{linearizableth} exists for each $(x,t) \in [0,\infty) \times [0, \infty)$. Uniqueness of $m$ follows because the jump matrix has unit determinant. Since $r \in \mathcal{S}(\RR)$, standard arguments based on ideas of the dressing method show that (\ref{recoveru}) defines a smooth solution $u(x,t)$ of the NLS equation, and a Deift--Zhou steepest descent analysis shows that $u(x,t)$ has rapid decay as $x \to \pm \infty$. 

We next verify that $u(x,t)$ obeys (i) the initial condition $u(x,0) = u_0(x)$, and (ii) the Robin boundary condition (\ref{Robincondition}). In the case of the focusing NLS, these properties were proved in \cite{IS2013}. In the case of (i), the analogous argument applies in the defocusing case. In the case of (ii), the defocusing case presents an additional difficulty. Indeed, the elegant verification of (ii) in \cite{IS2013} in the focusing case relies on the function
$$\delta(t) \coloneqq |m_{11}(0,t,-i\beta)|^2 - \lambda |m_{21}(0,t,-i\beta)|^2$$
being nonzero for all $t \geq 0$. In the focusing case of $\lambda = -1$, this condition is clearly always satisfied, and it is then shown in \cite{IS2013} that the following symmetry holds
\begin{align}\label{mRobinsymmetry}
\overline{m(-x,t,-\bar{k})} = \sigma_1 \bar{P}(t) \begin{pmatrix}\frac{1}{k - i\beta} & 0 \\ 0 & \frac{1}{k+i\beta} \end{pmatrix} \bar{P}(t)^{-1} m(x,t,k) \begin{pmatrix} k - i\beta & 0 \\ 0 & k+i\beta\end{pmatrix} D(k) \sigma_1,
\end{align}
where
\begin{align*}
& \bar{P}(t) \coloneqq 
\frac{1}{\delta(t)} \begin{pmatrix} m_{11}(0,t,-i\beta) & m_{12}(0,t,i\beta) \\ m_{21}(0,t,-i\beta) & m_{22}(0,t,i\beta) \end{pmatrix},
	\\
& D(k) \coloneqq \begin{cases} 
\begin{pmatrix} \Delta_e(k) & 0 \\ 0 & 1/\Delta_e(k) \end{pmatrix}, & k \in \CC_+, \\
\begin{pmatrix} 1/\Delta_e^*(k) & 0 \\ 0 & \Delta_e^*(k) \end{pmatrix}, & k \in \CC_-, 
\end{cases}
\end{align*}
and 
\begin{align}\label{Deltaedef}
\begin{cases}
\Delta_e(k) \coloneqq \frac{\Delta(k)}{2k-iq} \;\; \text{and} \;\;  \beta \coloneqq \frac{q}{2} & \text{if $q <0$, $a(-iq/2) \neq 0$ or $q > 0$, $b(iq/2) = 0$},
	\\
\Delta_e(k) \coloneqq \frac{\Delta(k)}{2k+iq}  \;\; \text{and} \;\;  \beta \coloneqq -\frac{q}{2} & \text{if $q >0$, $b(iq/2) \neq 0$ or $q < 0$, $a(-iq/2) = 0$}.
\end{cases}
\end{align}
As shown in \cite{IS2013}, the symmetry (\ref{mRobinsymmetry}) implies (ii). If $\delta(t) \neq 0$ for all $t \geq 0$, then analogous arguments show that the symmetry (\ref{mRobinsymmetry}) is satisfied and leads to (ii) also if $\lambda = 1$. But the difficulty of verifying that $\delta(t) \neq 0$ for all $t \geq 0$ remains. We claim that in fact this condition holds also if $\lambda =1$. To see this, note that
$$\delta(t) = \det X(t), \quad \text{where} \quad X(t) \coloneqq \begin{pmatrix} m_{11}(0,t,-i\beta) & m_{12}(0,t,i\beta) \\
 m_{21}(0,t,-i\beta) & m_{22}(0,t,i\beta) \end{pmatrix}.$$
Since $m(0,t,-i\beta)$ and $m(0,t,i\beta)$ both satisfy the $t$-part in (\ref{mulax}) (this follows from the dressing type arguments mentioned above), so does $X(t)$. Since $\mathsf{V}$ is trace-less, it follows that $\delta(t) \equiv \delta$ is independent of $t$.
Moreover, as $t \to \infty$, a steepest descent analysis of the RH problem shows that
$$m(0, t, \pm i \beta) \to \begin{cases} 
I, & \text{$\lambda = 1$ and $q < 0$}, \\
m_s(0,t,\pm i\beta), & \text{$\lambda = 1$ and $q > 0$}, 
\end{cases}$$
where $m_s$ is the solution corresponding to the stationary one-soliton.  
Since a computation shows that for $\lambda = 1$ and $q > 0$ (see Appendix \ref{RHsolitonsec}) 
$$|(m_{s}(0,t,-i\beta))_{11}|^2 - \lambda |(m_{s}(0,t,-i\beta))_{21}|^2 = -\frac{(\sqrt{\alpha^2 + \omega} + \sqrt{\omega})^2}{\alpha^2} \neq 0,$$ 
we conclude that $\delta(t) \equiv \delta \neq 0$ for all $t \geq 0$. This completes the verification of (ii) in the defocusing case.

Finally, the uniqueness of the Schwartz class solution $u(x,t)$ follows from Proposition \ref{mfromuprop}. Indeed, if $u_1(x,t)$ and $u_2(x,t)$ are two solutions, then Proposition \ref{mfromuprop} implies that both admit the representation (\ref{recoveru}) where $m$ is the unique solution of a RH problem whose formulation only involves the given data.
\end{proof}

\section{Proof of Theorem \ref{linearizableth}}\label{linearizableproofsec}
In Section \ref{Schwartzsec}, we showed that the conclusion of Theorem \ref{linearizableth} holds for initial data in the Schwartz class $\mathcal{S}(\RR_+)$. In order to prove Theorem \ref{linearizableth} in the general case of initial data in $H^{1,1}(\RR_+)$, we will use continuity arguments and the density of $\mathcal{S}(\RR_+)$ in $H^{1,1}(\RR_+)$.
We begin with a few lemmas.

\begin{lemma}\label{sequencelemma}
Let $q \in \RR$ and $u_0 \in H^{1,1}(\RR_+)$. Then there exists a sequence $\{u_0^{(n)}\}_{n=1}^\infty \subset \mathcal{S}(\RR_+)$ such that 
\begin{enumerate}[$(i)$]
\item $u_0^{(n)} \to u_0$ in $H^{1,1}(\RR_+)$ as $n \to \infty$, and 
\item $(u_{0}^{(n)})'(0) + qu_0^{(n)}(0) = 0$ for each $n \geq 1$.
\end{enumerate}
\end{lemma}
\begin{proof}
Let $\eta: \RR \to [0,1]$ be a smooth cut-off function such that $\eta$ is even, $\eta \equiv 1$ in a neighborhood of $0$, $\eta(x) = 0$ for $|x| \geq 1$, and $\int_\RR \eta(x) dx = 1$. 
Let $u_{0e} \in H^{1,1}(\RR)$ be an extension of $u_0$ to $\RR$ such that $u_{0e}(x) = u_0(x)$ for $x \geq 0$ and $u_{0e}(x) = 0$ for $x \leq -1$. 
Standard mollifier arguments show that $v_{\epsilon}(x) \coloneqq \int_{\RR} u_{0e}(y) \eta_{\epsilon}(x-y) dy$, where $\eta_\epsilon(x) \coloneqq \eta(x/\epsilon)/\epsilon$, are smooth compactly supported functions such that $v_\epsilon(x) \to u_{0e}$ in $H^{1,1}(\RR)$ as $\epsilon \downarrow 0$. If $u_0(0) \neq 0$, then $v_\epsilon(0) \neq 0$ for all sufficiently small $\epsilon$, and we can obtain the desired sequence by setting
$$u_0^{(n)}(x) = e^{(q + v_\epsilon'(0)/v_\epsilon(0))\int_x^\infty \eta(x'/\epsilon)dx'} v_\epsilon(x),$$
where $\epsilon = 1/(N+n)$ and $N$ is large enough. If $u_0(0) = 0$, then by modifying $u_{0e}$ so that $u_{0e}(-x) = -u_0(x)$ for all sufficiently small $x > 0$, we may assume that $v_\epsilon(0) = 0$ for all small enough $\epsilon$, and then straightforward estimates show that
$$u_0^{(n)}(x) = (1 - \eta(x/\sqrt{\epsilon}))v_\epsilon(x),$$
where $\epsilon = 1/(N+n)$, provides a sequence with the desired properties.
\end{proof}

\begin{lemma}\label{u0toabLinftylemma}
The maps $u_0 \mapsto a$ and $u_0 \mapsto b$ are continuous $H^{1,1}(\RR_+) \to L^\infty(\bar{\CC}_+)$.
\end{lemma}
\begin{proof}
Using the same notation as in the proof of Lemma \ref{continuitylemma1}, we have from (\ref{DeltapsiVolterra}) that $\Delta \psi = K_{\Delta u_0}\check{\psi}+K_{u_0}\Delta \psi$, where, by (\ref{Ebounded}), (\ref{Psisum}), and (\ref{Kdef}),
$$|(K_{\Delta u_0}\check{\psi})(x, k)|
\leq C\|\Delta u_0\|_{L^1(\RR_+)} e^{C \|u_0\|_{L^1(\RR_+)}}
\leq C\|\Delta u_0\|_{H^{1,1}(\RR_+)}, \qquad x \geq 0, \; \im k \geq 0.$$
Hence, using (\ref{Ebounded}) again, a standard Volterra series estimate as in (\ref{Psisum}) shows that, for all $x \geq 0$ and $\im k \geq 0$,
$$|\Delta\psi(x,k)| \leq C e^{C\|u_0\|_{L^1(\RR_+)}}\|\Delta u_0\|_{H^{1,1}(\RR_+)} \leq C \|\Delta u_0\|_{H^{1,1}(\RR_+)},$$
uniformly for $u_0, \check{u}_0$ in bounded subsets of $H^{1,1}(\RR_+)$. Setting $x = 0$ in this estimate, the lemma follows.
\end{proof}

\begin{lemma}\label{zeroscontinuitylemma}
Let $q \in \RR \setminus \{0\}$ and $\check{u}_0 \in H^{1,1}(\RR_+)$ and suppose the associated spectral functions $\check{\Delta}, \check{a}$ satisfy Assumption \ref{zerosassumption}. Then Assumption \ref{zerosassumption} holds also for any potential $u_0$ sufficiently close to $\check{u}_0$ in $H^{1,1}(\RR_+)$. Moreover, if $\{\xi_i\}_1^M$ are the simple zeros in $\CC_+$ of $\Delta(k)$, then the maps $u_0 \mapsto M \in \ZZ$ and $u_0 \mapsto (\xi_1, \dots, \xi_M) \in \CC^M$ are continuous at $\check{u}_0 \in H^{1,1}(\RR_+)$.
\end{lemma}

\begin{proof}
Suppose first that $q < 0$. The function $\Delta_a(k) \in 1 + H^1(\RR)$ defined in (\ref{Deltaadef}) depends continuously on $u_0 \in H^{1,1}(\RR_+)$ by Lemma \ref{continuitylemma1} (see (\ref{u0toDeltaacontinuous})). It follows that $\Delta_a$, and hence also $\Delta$, is nonzero on $\RR$ for any $u_0$ sufficiently close to $\check{u}_0$. Moreover, $\Delta_a$ has the same number of zeros and poles in $\CC_+$ as $\Delta$, and $\Delta_a \to 1$ as $k \to \infty$. Since $\Delta$ has no poles in $\CC_+$, the argument principle applied to a large semicircle enclosing the upper half-plane yields
\begin{align*}
Z_{\Delta} = \frac{\log \Delta_a(k)}{2\pi i} \bigg|_{k=-\infty}^{+\infty},
\end{align*}
where $Z_{\Delta}$ is the number of zeros of $\Delta$ in $\CC_+$ counted with multiplicity.
Using again that $\Delta_a(k) \in H^1(\RR)$ depends continuously on $u_0 \in H^{1,1}(\RR_+)$, we infer that the map $u_0 \mapsto Z_{\Delta} \in \ZZ$ is continuous at $\check{u}_0 \in H^{1,1}(\RR_+)$.

Since $q < 0$, the function $\Delta_a$ has the same zeros as $\Delta$ in the upper half-plane, and Cauchy's integral formula shows that if $\xi_j \in \CC_+$ is a simple zero of $\Delta_a$, then
$$\xi_j = \frac{1}{2\pi i} \oint_{\gamma_j} \frac{k \dot{\Delta}_a(k)}{\Delta_a(k)}dk,$$
where $\gamma_j$ is a small circle in $\CC_+$ around $\xi_j$ which contains no other zeros of $\Delta_a$. 
By Lemma \ref{u0toabLinftylemma}, the map $u_0 \mapsto \Delta_a$ is continuous $H^{1,1}(\RR_+) \to L^\infty(\bar{\CC}_+)$. Since $\Delta_a$ is analytic in $\CC_+$ it follows that $u_0 \mapsto \dot{\Delta}_a$ is continuous $H^{1,1}(\RR_+) \to L^\infty(K)$ where $K$ is any compact subset of $\CC_+$. 
Since the zeros $\{\check{\xi}_j\}_1^M$ of $\check{\Delta}(k)$ are simple by assumption, we conclude that the zeros $\{\xi_j\}_1^M$ of $\Delta(k)$ are simple and depend continuously on $u_0$ for $u_0$ in a $H^{1,1}(\RR_+)$-neighborhood of $\check{u}_0$.

Analogous arguments applied to the function
\begin{equation}\label{Deltabdef}
\Delta_b(k)\coloneqq \frac{\Delta(k)}{2k+iq}=\frac{2k-iq}{2k+iq}a(k)\overline{a(-\bar{k})}+\lambda b(k) \overline{b(-\bar{k})}, \qquad \im k \geq 0,
\end{equation}
instead of $\Delta_a(k)$ show that the above conclusions hold also if $q > 0$ (note that the identity (\ref{rDeltaarelation}) holds also with $\Delta_a$ replaced by $\Delta_b$). \end{proof}

If the set of zeros of $\Delta$ is non-empty, we replace the poles in the RH problem of Theorem \ref{linearizableth} by jumps along small circles in the standard way. To this end, let $D_j \subset \CC_+$, $j = 1, \dots, M$, be small disjoint open disks centered at the zeros $\xi_j$, $j = 1, \dots, M$, of $\Delta$. Let $D_j^*$ be the image of $D_j$ under complex conjugation. Define the contour $\Gamma$ by 
$$\Gamma = \RR \cup \big( \cup_{j=1}^M (\partial D_j \cup \partial D_j^*)\big),$$
where $\partial D_j$ is oriented clockwise and $\partial D_j^*$ is oriented counterclockwise. Let
$$\tilde{m} = \begin{cases} m(x,t,k), & k \in \CC \setminus (\RR \cup \big( \cup_{j=1}^M (\overline{D_j} \cup \overline{D_j^*})\big), 	
	\\
m(x,t,k) P_j(x,t,k)^{-1}, & k \in D_j, \;\; j = 1, \dots, M, 
	\\
m(x,t,k) Q_j(x,t,k), & k \in D_j^*, \;\; j = 1, \dots, M, 
\end{cases}$$
where
$$P_j(x,t,k) \coloneqq \begin{pmatrix} 1 & 0 \\ \frac{c_j e^{2i\theta(x,t,\xi_j)}}{k- \xi_j} & 1 \end{pmatrix}, \quad
Q_j(x,t,k) \coloneqq \begin{pmatrix} 1 & -\frac{\lambda \bar{c}_j e^{-2i\theta(x,t,\bar{\xi}_j)}}{k- \bar{\xi}_j}  \\ 0 & 1 \end{pmatrix}, \qquad j = 1, \dots, M,$$
and $\{c_j\}_1^M$ are the residue constants defined in (\ref{cjdef}).
Then $m$ satisfies the RH problem of Theorem \ref{linearizableth} if and only if $\tilde{m}$ satisfies the following RH problem. 

\begin{RHproblem}[RH problem for $\tilde{m}$]\label{RHmreg}
Find a $2 \times 2$-matrix valued function $\tilde{m}(x,t,k)$ with the following properties:
\begin{enumerate}[$(a)$]
\item $\tilde{m}(x,t,\cdot): \mathbb{C} \setminus \Gamma \to \mathbb{C}^{2 \times 2}$ is analytic.

\item The boundary values of $m(x,t,k)$ as $k$ approaches $\Gamma$ from the left $(+)$ and right $(-)$ exist, are continuous on $\Gamma$, and satisfy
\begin{align}\label{mregjump}
  \tilde{m}_+(x,t,k) = \tilde{m}_-(x, t, k) \tilde{v}(x, t, k), \qquad k \in \Gamma,
\end{align}
where $\tilde{v}$ is defined by
\begin{align}\label{vregdef}
\tilde{v}(x,t,k) \coloneqq \begin{cases}
v(x,t,k),& k \in \RR, 
	\\
P_j(x,t,k),& k \in \partial D_j, \; \; j = 1, \dots, M, 
	\\
Q_j(x,t,k),& k \in \partial D_j^*, \; \; j = 1, \dots, M.
\end{cases}
\end{align}

\item $\tilde{m}(x,t,k) \to I$ as $k \to \infty$.
\end{enumerate}
\end{RHproblem}

We next consider the set of all initial data for which the  assumptions of Theorem \ref{linearizableth} are fulfilled. As discussed in the introduction, we will have to assume existence of a solution of the RH problem in the defocusing case when there is a pole.

\begin{definition}\label{opensetUdef}
We let $\mathcal{U} \subset H^{1,1}(\RR_+)$ denote the set of all potentials $u_0 \in H^{1,1}(\RR_+)$ such that the corresponding spectral functions satisfy the assumptions of Theorem \ref{linearizableth}, i.e.,
\begin{itemize}
\item if $\lambda = -1$, then $\mathcal{U}$ consists of all $u_0 \in H^{1,1}(\RR_+)$ for which the associated spectral functions satisfy Assumption \ref{zerosassumption};

\item if $\lambda = 1$ and $q < 0$, then $\mathcal{U} = H^{1,1}(\RR_+)$;

\item if $\lambda = 1$ and $q > 0$, then $\mathcal{U}$ consists of all $u_0 \in H^{1,1}(\RR_+)$ for which the RH problem of Theorem \ref{linearizableth} has a solution $m(x,t,\cdot)$ for each $(x,t) \in [0,\infty) \times [0,\infty)$.

\end{itemize}

\end{definition}

For $h \in L^2(\Gamma)$, we define the Cauchy transform $\mathcal{C} h$ by
\begin{equation}
(\mathcal{C}h)(k)=\frac{1}{2\pi i}\int_{\Gamma}\frac{h(k')dk'}{k'-k}, \qquad k \in \CC \setminus \Gamma,
\end{equation}
and write $\mathcal{C}_+ h$ and $\mathcal{C}_-h$ for the left and right boundary values of  $\mathcal{C}f$ on $\Gamma$. Then $\mathcal{C}_+$ and $\mathcal{C}_-$ are bounded operators on $L^2(\Gamma)$ and given $w \in L^2(\Gamma) \cap L^\infty(\Gamma)$, we define $\mathcal{C}_{w}: L^2(\Gamma) + L^\infty(\Gamma) \to L^2(\Gamma)$ by $\mathcal{C}_{w}(f) = \mathcal{C}_-(f w)$.

\begin{lemma}\label{mexistencelemma}
Let $u_0 \in \mathcal{U}$ and $q \in \RR \setminus \{0\}$. Then RH problem \ref{RHmreg} has a unique solution $\tilde{m}(x,t,k)$ for each $(x,t) \in [0,\infty) \times [0,\infty)$ and this solution admits the representation
\begin{align}\label{mtilderepresentation}
\tilde{m}(x,t,k) = I + \frac{1}{2\pi i} \int_{\Gamma} \frac{\tilde{\mu}(x,t,k') \tilde{w}(x,t,k') dk}{k'-k}, \qquad k \in \CC \setminus \Gamma,
\end{align}
where $\tilde{w} = \tilde{v} - I$ and 
\begin{align}\label{mutildedef}
\tilde{\mu} = I + (I - \mathcal{C}_{\tilde{w}})^{-1}\mathcal{C}_{\tilde{w}}I \in I+L^2(\Gamma).
\end{align}
\end{lemma}
\begin{proof}
By Proposition \ref{rDeltaprop}, $r \in H^{1,1}(\RR)$. By Morrey's inequality, every function in $H^{1,1}(\RR_+) \subset H^1(\RR_+)$ is H\"older continuous with exponent $1/2$. Hence, if $\lambda = -1$, the unique existence of $\tilde{m}$ follows by the same vanishing lemma arguments already used to prove Proposition \ref{Schwartzprop}. The same argument applies if $\lambda = 1$ and $q < 0$, because in this case $M = 0$. If $\lambda = 1$ and $q > 0$, then $m$, and hence also $\tilde{m}$, exists by assumption. Standard theory for RH problems then yields the representation formula (\ref{mtilderepresentation}) for $\tilde{m}$.
\end{proof}

The next lemma shows that $\mathcal{U}$ is open and that the singular integral representation \eqref{mutildedef} for $\tilde{\mu}$ is continuous with respect to the initial data in $\mathcal{U}$. This lemma will enable us to approximate the RH solution corresponding to the initial data $u_0 \in \mathcal{U}$ by a sequence of RH solutions corresponding to Schwartz class initial data that approximate $u_0$.

\begin{lemma}\label{UmuBlemma}
The set $\mathcal{U}$ is open in $H^{1,1}(\RR_+)$. Moreover, the map 
\begin{align}\label{tildemumap}
(x,t,u_0) \mapsto \tilde{\mu}(x,t,\cdot): [0, \infty)^2 \times \mathcal{U} \to I +L^2(\Gamma),
\end{align}
where $\tilde{\mu}$ is defined by (\ref{mutildedef}), is continuous.
\end{lemma}
\begin{proof}
By Proposition \ref{rDeltaprop} (\ref{rDeltapropitem5}), the map $u_0 \mapsto r: \mathcal{U} \to H^{1,1}(\RR)$ is continuous. Recalling the definition (\ref{vdef}) of $v$, straightforward estimates then show that $(x,t,u_0) \mapsto v - I:[0, \infty)^2 \times \mathcal{U} \to (L^2 \cap L^\infty)(\RR)$ is continuous. On the other hand, since $\Delta_a \in L^\infty(\CC_+)$ depends continuously on $u_0 \in H^{1,1}(\RR_+)$ by Lemma \ref{u0toabLinftylemma}, and each $\xi_j$ depends continuously on $u_0 \in \mathcal{U}$ by Lemma \ref{zeroscontinuitylemma}, Cauchy's formula for the first derivative shows that $\dot{\Delta}(\xi_j) \in \CC$ depends continuously on $u_0 \in \mathcal{U}$. The continuous dependence of $a(\xi_j), b(-\bar{\xi}_j) \in \CC$ on $u_0 \in H^{1,1}(\RR_+)$ follows from Lemma \ref{u0toabLinftylemma}. Thus, $u_0\mapsto c_j e^{2i\theta(x,t,\xi_j)}$ is continuous $[0, \infty)^2 \times \mathcal{U} \to \CC$ for each $j$. We conclude that $(x,t,u_0) \mapsto \tilde{w}= \tilde{v} - I:[0, \infty)^2 \times \mathcal{U} \to (L^2 \cap L^\infty)(\Gamma)$ is continuous.

Let $\mathcal{B}(L^2(\Gamma))$ denote the space of bounded linear operators on $L^2(\Gamma)$. The set of invertible operators is open in $\mathcal{B}(L^2(\Gamma))$ and the linear map $\tilde{w} \mapsto \mathcal{C}_{\tilde{w}}I$ lies in $\mathcal{B}(L^2(\Gamma))$. Also, since $\|\mathcal{C}_{\tilde{w}}\|_{\mathcal{B}(L^2(\Gamma))} \leq C \|\tilde{w}\|_{L^\infty(\Gamma)}$, the map $\tilde{w} \mapsto I - \mathcal{C}_{\tilde{w}}$ is continuous $L^\infty(\Gamma) \to \mathcal{B}(L^2(\Gamma))$. Since $(x,t,u_0) \mapsto \tilde{\mu}(x,t,\cdot)-I$ with $\tilde{\mu}$ given by (\ref{mutildedef}) can be viewed as a combination of maps of the above forms together with the smooth inversion map $I - \mathcal{C}_{\tilde{w}} \mapsto (I - \mathcal{C}_{\tilde{w}})^{-1}$, this proves that the map in (\ref{tildemumap}) is continuous.

If $\lambda = -1$, the openness of $\mathcal{U}$ follows immediately from Lemma \ref{zeroscontinuitylemma}. If $\lambda = 1$ and $q < 0$, then $\mathcal{U} = H^{1,1}(\RR_+)$ is trivially open. Suppose therefore that $\lambda = 1$ and $q > 0$. Then $\Delta \neq 0$ on $\RR$, and $\Delta$ has exactly one zero $\xi_1$ in $\CC_+$ by Proposition  \ref{rDeltaprop} (\ref{rDeltapropitemc}). Moreover, $a(\xi_1)$ is nonzero by Proposition  \ref{rDeltaprop} (\ref{rDeltapropiteme}).  Thus, in this case, the arguments in the first part of the proof can be extended to all of $H^{1,1}(\RR_+)$, implying that
\begin{align}\label{tildewmaplambda1}
(x,t,u_0) \mapsto \tilde{w}(x,t,\cdot; u_0): [0, \infty)^2 \times H^{1,1}(\RR_+) \to (L^2 \cap L^\infty)(\RR) 
\end{align}
is continuous, where we have indicated the $u_0$-dependence explicitly for clarity. 
Fix $\check{u}_0 \in \mathcal{U}$. By assumption, the RH problem for $\tilde{m}$ corresponding to $\check{u}_0$ has a solution for each $(x,t) \in [0,\infty) \times [0,\infty)$. By standard theory for RH problem, this is equivalent to the map $I - \mathcal{C}_{\tilde{w}(x,t,\cdot; \check{u}_0)}: L^2(\Gamma) \to L^2(\Gamma)$ being invertible for each $(x,t) \in [0,\infty) \times [0,\infty)$. On the other hand, since $\tilde{w}\to I - \mathcal{C}_{\tilde{w}}:L^\infty(\RR) \to \mathcal{B}(L^2(\Gamma))$ is continuous, the continuity of the map (\ref{tildewmaplambda1}) implies that the map $F$ defined by
$$F:(x,t,u_0) \mapsto I - \mathcal{C}_{\tilde{w}(x,t,\cdot; u_0)}:[0, \infty)^2 \times H^{1,1}(\RR_+) \to \mathcal{B}(L^2(\Gamma))$$
is continuous. Let $V$ denote the open set of invertible operators in $\mathcal{B}(L^2(\Gamma))$. Then $F^{-1}(V)$ is open. Given any compact set $K \subset [0, \infty)^2$, $K \times \{\check{u}_0\}$ is a compact subset of $F^{-1}(V)$, and hence there is a neighborhood $U_1 \subset H^{1,1}(\RR_+)$ of $\check{u}_0$ such that $(x,t,U_1) \subset F^{-1}(V)$ for all $(x,t) \in K$, i.e., such that the solution $\tilde{m}(x,t,\cdot; u_0)$ exists whenever $u_0 \in U_1$ and $(x,t) \in K$.
On the other hand, by a Deift--Zhou steepest descent analysis of the RH problem for $\tilde{m}$, one obtains that there is a neighborhood $U_2 \subset H^{1,1}(\RR_+)$ of $\check{u}_0$ such that $\tilde{m}(x,t,\cdot; u_0)$ exists for all $u_0 \in U_2$ whenever $\sqrt{x^2 + t^2}$ is large enough. Indeed, for large $\sqrt{x^2 + t^2}$, the solution is well-approximated by the global parametrix, which in this case is the solution of the RH problem corresponding to the pure stationary one-soliton and always exists (see for example \cite{BJM2018}).
By choosing $K$ large enough, we find that $U_1 \cap U_2 \subset H^{1,1}(\RR_+)$ is an open neighborhood of $\check{u}_0$ such that $\tilde{m}(x,t,\cdot; u_0)$ exists for all  $u_0 \in U_1 \cap U_2$ and all $(x,t) \in [0, \infty)^2$, that is, $U_1 \cap U_2 \subset \mathcal{U}$. This shows that $\mathcal{U}$ is open. 
\end{proof}

\begin{proof}[Proof of Theorem \ref{linearizableth}]
Let $q\in \RR \setminus \{0\}$ and $u_0 \in \mathcal{U}$.
Let $\{u_0^{(n)}\}_{n=1}^\infty \subset \mathcal{S}(\RR_+)$ be a sequence such that $u_0^{(n)} \to u_0$ in $H^{1,1}(\RR_+)$ as $n \to \infty$ and such that $(u_{0}^{(n)})'(0) + qu_0^{(n)}(0) = 0$ for each $n \geq 1$. Such a sequence exists by Lemma \ref{sequencelemma}.
By Lemma \ref{UmuBlemma},  $\mathcal{U}$ is open in $H^{1,1}(\RR_+)$. Hence, passing to a subsequence if necessary, we may assume that $\{u_0^{(n)}\}_{n=1}^\infty \subset \mathcal{U}$, which means that the assumptions of Proposition \ref{Schwartzprop} are satisfied for each $u_0^{(n)}$. Applying Proposition \ref{Schwartzprop}, we conclude that there exists a unique global Schwartz class solution $u^{(n)}(x,t)$ of the Robin IBVP for NLS with parameter $q$ and initial data $u_0^{(n)}$ for each $n$.

By Lemma \ref{mexistencelemma}, the functions $\tilde{m}$ and $\tilde{\mu}$ corresponding to $u_0$, as well as the functions $\tilde{m}^{(n)}$ and $\tilde{\mu}^{(n)}$ corresponding to $u_0^{(n)}$ are all well-defined and satisfy (\ref{mtilderepresentation}) and (\ref{mutildedef}).
Proposition \ref{Schwartzprop} implies that $u^{(n)}$ satisfies, for each $n \geq 1$,
\begin{equation}\label{recoverun}
u^{(n)}(x,t) = 2i\lim_{k \to \infty} k(\tilde{m}^{(n)}(x,t,k))_{12}, \qquad x \geq 0, ~ t \geq 0.
\end{equation}

From the global well-posedness result Proposition \ref{H11existenceprop}, there is a unique global solution $u(x,t)$ in $H^{1,1}(\RR_+)$ of the Robin IBVP for NLS with parameter $q$ and initial data $u_0$.
The Schwartz class solution $u^{(n)}(x,t)$ is clearly also a solution in $H^{1,1}(\RR_+)$ and by the continuity of the data-to-solution mapping established in Proposition \ref{H11existenceprop}, it follows that $u^{(n)}(x,t) \to u(x,t)$ as $n \to \infty$ for each $(x,t)$. Hence, letting $n \to \infty$ in the $(12)$-entry of (\ref{recoverun}),
\begin{equation}\label{ulimlimmn}
u(x,t) = \lim_{n\to \infty}  2i \lim_{k \to \infty} k(\tilde{m}^{(n)}(x,t,k))_{12}, \qquad x \geq 0, ~ t \geq 0.
\end{equation}

Next note that, by (\ref{mtilderepresentation}),
\begin{align}\nonumber
2i\lim_{k \to \infty}k(\tilde{m}^{(n)}(x,t,k)-I) 
& = -\frac{1}{\pi}\int_\RR(\tilde{\mu}^{(n)} \tilde{w}^{(n)})(x,t,k')dk'
	\\ \label{limkmnreg}
& =-\frac{1}{\pi}\left(X_n(x,t) + Y_n(x,t) + Z(x,t)\right)
\end{align}
and
\begin{align}\label{limkmreg}
2i\lim_{k \to \infty}k(\tilde{m}(x,t,k)-I) =&-\frac{1}{\pi}\int_\RR(\tilde{\mu} \tilde{w})(x,t,k')dk'
=-\frac{Z(x,t)}{\pi},
\end{align}
where the limits are taken nontangentially with respect to $\RR$, and
\begin{align*}
X_n \coloneqq \int_\RR (\tilde{\mu}^{(n)}-\tilde{\mu}) \tilde{w}^{(n)} dk',\quad 
Y_n \coloneqq \int_\RR \tilde{\mu} (\tilde{w}^{(n)} - \tilde{w})dk',\quad 
Z \coloneqq \int_{\RR}\tilde{\mu} \tilde{w} dk'.
\end{align*}
Fix $(x,t) \in [0, \infty) \times [0, \infty)$. By Lemma \ref{UmuBlemma}, we have $\tilde{\mu}^{(n)} \to \tilde{\mu}$ in $I + L^2(\Gamma)$  as $n \to \infty$. By Lemma \ref{rcontinuitylemma} and (\ref{vregdef}), we have $\tilde{w}^{(n)} \to \tilde{w}$ in $H^{1,1}(\RR_+)$ as $n \to \infty$.
Hence, $X_n \to 0$ and $Y_n \to 0$ as $n \to \infty$, so letting $n \to \infty$ in the $(12)$-entry of (\ref{limkmnreg}) and comparing with the $(12)$-entry of (\ref{limkmreg}), we find
\begin{align}\label{limlimmregnlimmreg}
\lim_{n\to \infty} 2i\lim_{k \to \infty} k(\tilde{m}^{(n)}(x,t,k))_{12} 
= 2i\lim_{k \to \infty}k(\tilde{m}(x,t,k))_{12}
= 2i\lim_{k \to \infty}k(m(x,t,k))_{12}.
\end{align}
Substituting this into (\ref{ulimlimmn}), we obtain the desired reconstruction formula \eqref{recoveru} for $u$. This completes the proof the theorem.
\end{proof}

\section{Long-time asymptotics}\label{asymproofsection}

\subsection{Proof of Theorem \ref{onesolitonasymptoticsth}}\label{onesolitonasymptoticsthproof}
Let $\lambda = 1$ and $q > 0$. Let $u_0 \in H^{1,1}(\RR_+)$ and suppose the RH problem of Theorem \ref{linearizableth} has a solution $m(x,t,k)$ for each $(x,t) \in [0,\infty) \times [0,\infty)$. By Proposition \ref{rDeltaprop} (\ref{rDeltapropitemc}), $\Delta(k)$ has one simple zero $\xi_1 \in i\RR_{>0}$, which means that $m(x,t,k)$ has simple poles at $\xi_1$ and $\bar{\xi}_1$. We will use a Darboux transformation to remove these poles and consider a solution $m_{\reg}$ of an associated regular RH problem without poles.
Define 
\begin{equation}\label{rregdef}
r_{\reg}(k):=r(k)\frac{k-\xi_1}{k-\bar{\xi}_1}, \qquad k \in \RR.
\end{equation}
By Proposition \ref{rDeltaprop}, $r,r_\reg \in H^{1,1}(\RR)$ and $|r(k)| = |r_\reg(k)| < 1$ for $k \in \RR$. Let $m_\reg(x,t,k)$ be the unique solution of the RH problem of Theorem \ref{linearizableth} with $r(k)$ replaced by $r_\reg(k)$ and with no poles in $\CC \setminus \RR$. Such a solution exists by a standard vanishing lemma argument as described in the proof of Lemma \ref{Schwartzprop} and satisfies $\det m_\reg = 1$. 
Define $B_1$ by
\begin{align}\label{B1definition}
B_1 =-kI+m(x,t,k)\begin{pmatrix}k-\xi_1 & 0 \\ 0 & k -\bar{\xi}_1\end{pmatrix}m_{\reg}(x,t,k)^{-1}.
\end{align}
The properties of $m$ and $m_\reg$ imply that $B_1$ is an entire function of $k$ which is $O(1)$ as $k \to \infty$; hence $B_1 \equiv B_1(x,t)$ is independent of $k$. 
Evaluating (\ref{B1definition}) as $k \to \xi_1$ and as $k \to \bar{\xi}_1$, and using the residue conditions (\ref{mresidues}) satisfied by $m$, we conclude that $B_1$ solves the algebraic system
\begin{equation}\label{algsys}
\begin{cases}
(\xi_1 I+B_1(x,t))m_{\reg}(x,t,\xi_1)\begin{psmallmatrix}1 \\ -d_1(x,t) \end{psmallmatrix}=0,
	\\
(\bar{\xi}_1 I+B_1(x,t))m_{\reg}(x,t,\bar{\xi}_1)\begin{psmallmatrix}-\lambda\overline{d_1(x,t)} \\ 1 \end{psmallmatrix}= 0,
\end{cases}
\end{equation}
with
\begin{align}\label{Darbouxconst}
& d_1(x,t) :=  \frac{c_1 e^{2i\theta(x,t,\xi_1)}}{\xi_1 - \bar{\xi}_1}.
\end{align}
Equation (\ref{B1definition}) can be rewritten as
\begin{equation}\label{Darboux}
m(x,t,k)=(kI+B_1)m_{\reg}(x,t,k)\begin{pmatrix}\frac{1}{k-\xi_1} & 0 \\ 0 & \frac{1}{k-\bar{\xi}_1}\end{pmatrix}.
\end{equation}
As a consequence of Theorem \ref{linearizableth}, $u(x,t)$ is given in terms of $m$ by \eqref{recoveru}. According to \eqref{recoveru} and \eqref{Darboux}, we have
\begin{equation}\label{uasymformula}
u(x,t) = 2i\lim_{k \to \infty} k(m(x,t,k))_{12}= u_\reg(x,t) + 2i (B_1(x,t))_{12}.
\end{equation}
where
\begin{align}\label{uregdef}
u_{\reg}(x,t) \coloneqq 2i\lim_{k \to \infty} k(m_\reg(x,t,k))_{12}.
\end{align}
To determine the asymptotics of $u$, it is therefore sufficient to find the asymptotics of $u_\reg$ and $(B_1)_{12}$. 

\begin{lemma}\label{regasymplemma}

\begin{enumerate}[$(a)$]
\item \label{regasymplemmaitem1}
As $t \to \infty$, the function $u_{\reg}$ defined in (\ref{uregdef}) obeys the asymptotics (\ref{udefocusingqnegative}) with $r$ replaced by $r_\reg$, i.e., 
\begin{align}\label{uregasymptotics}
u_{\reg}(x,t)=\frac{u_{\mathrm{rad}}^{(1)}(x,t)}{\sqrt{t}} + O\big(t^{-3/4}\big)
\end{align}
as $t \to \infty$ uniformly for $x \in [0, \infty)$, where $u_{\mathrm{rad}}^{(1)}$ is given by (\ref{urad1defoc}).

\item \label{regasymplemmaitem2}
If $\xi_1 \in \CC_+$, then
\begin{equation}
m_{\reg}(x, t, \xi_1) = \delta(\zeta,\xi_1)^{\sigma_3} + \frac{Y(\zeta, t) m_1^X(\zeta) Y(\zeta, t)^{-1}}{(k_0 - \xi_1)\sqrt{8t}}\delta(\zeta,\xi_1)^{\sigma_3} 
+ O\big(t^{-3/4}\big)
\end{equation}
as $t\to \infty$ uniformly for $x \in [0, \infty)$, where
\begin{align}\nonumber
\delta(\zeta,k):=&\; \exp\left[\frac{1}{2\pi i}\int_{-\infty}^{k_0}\frac{\log(1-\lambda |r(s)|^2)}{s-k}ds\right], \qquad \delta_0(\zeta,t):= (8t)^{\frac{i\nu}{2}}e^{\chi(\zeta,k_0)},
	\\ \label{deltadelta0Ym1Xdef}
Y(\zeta,t) := &\; e^{-\frac{t\Phi(\zeta, k_0)\sigma_3}{2}} \delta_0(\zeta, t)^{\sigma_3},
\qquad
m_1^X(\zeta) := i\begin{pmatrix} 0  & - \beta(r(k_0)) \\ \lambda \overline{\beta(r(k_0))} & 0 \end{pmatrix}.
\end{align}
\end{enumerate}
\end{lemma}
\begin{proof}
The RH problem for $m_\reg$ has the same form as the RH problem associated to the defocusing NLS equation on the line with reflection coefficient given by $r_\reg(k)$. The lemma therefore follows from well-known results (see \cite{DMM2019}) in the same way as Theorem \ref{nosolitonsasymptoticsth}.
\end{proof}

We can write the algebraic system \eqref{algsys} as
\begin{equation}\label{B1def}
B_1=-W_1 \begin{pmatrix} \xi_1 & 0 \\ 0 & \bar{\xi}_1 \end{pmatrix} W_1^{-1},
\end{equation}
where
\begin{equation}\label{W}
W_1 \equiv W_1(x,t;u_0)=\left(m_{\reg}(x,t,\xi_1) \begin{pmatrix} 1 \\  -d_1(x,t) \end{pmatrix}, m_{\reg}(x,t,\bar{\xi}_1)\begin{pmatrix} -\lambda \overline{d_1}(x,t) \\ 1 \end{pmatrix}\right).
\end{equation}
From \eqref{B1def}, we have
\begin{equation}\label{B112}
2i(B_1(x,t))_{12} = \frac{4i\xi_1 (W_1)_{11}(W_1)_{12}}{|(W_1)_{11}|^2 - \lambda |(W_1)_{12}|^2}
\end{equation}
where
\begin{align*}
& (W_1(x,t))_{11} = (m_\reg(x,t,\xi_1))_{11} - d_1(x,t) (m_\reg(x,t,\xi_1))_{12},
	\\
& (W_1(x,t))_{12} = (m_\reg(x,t,\bar{\xi}_1))_{12} - \lambda \overline{d_1(x,t)} (m_\reg(x,t,\bar{\xi}_1))_{11}.
\end{align*}
By Lemma \ref{regasymplemma} (\ref{regasymplemmaitem2}), as $t \to \infty$, we have
\begin{align*}
(W_1)_{11} =&\; \delta(\zeta,\xi_1) + d_1(x,t) \frac{\lambda ie^{-t\Phi(\zeta,k_0)}\beta^X(r_\reg(k_0))\delta_0(\zeta,t)^2 \delta(\zeta,\xi_1)^{-1}}{(k_0-\xi_1)\sqrt{8t}} + O\big(t^{-3/4}\big),
	\\
(W_1)_{12}=&-\lambda \overline{(m_{\reg}(x,t,\xi_1))_{22}} \overline{d_1(x,t)}+\lambda \overline{(m_{\reg}(x,t,\xi_1))_{21}}
	\nonumber\\
=&-\lambda \overline{\delta(\zeta,\xi_1)}^{-1} \overline{d_1(x,t)} - \frac{i e^{t\overline{\Phi(\zeta,k_0)}}\beta^X(r_\reg(k_0)) \delta_0(\zeta,t)^2 \overline{\delta(\zeta,\xi_1)}}{(k_0+\xi_1)\sqrt{8t}} + O\big(t^{-3/4}\big).
\end{align*}
Substituting these expansions into \eqref{B112}, we find
\begin{equation}\label{B112asym}
2i(B_1(x,t))_{12}=u_{\mathrm{sol}}(x,t)+\frac{u_{\mathrm{rad}}^{(2)}(x,t)}{\sqrt{t}}+O\big(t^{-3/4}\big),\qquad t \to \infty,
\end{equation}
uniformly for $x \in [0,\infty)$, where $u_{\mathrm{sol}}$ and $u_{\mathrm{rad}}^{(2)}$ are given by (\ref{usoldefoc}) and (\ref{urad2defoc}), respectively.  
Substitution of \eqref{uregasymptotics} and \eqref{B112asym} into \eqref{uasymformula} gives the asymptotic formula \eqref{uqpositiveasymptotics} and this completes the proof of the Theorem \ref{onesolitonasymptoticsth}.

\subsection{Proof of Theorem \ref{asymptoticsthdefocusing}}\label{asymptoticsthdefocusingproof}
Suppose $\lambda = 1$ and let $q = \sqrt{\alpha^2+\omega} > 0$ be the value of $q$ associated with the stationary one-soliton. 
Theorem \ref{onesolitonasymptoticsth} provides the asymptotics for the solution $u(x,t)$ of the defocusing NLS equation whenever the initial data $u_0$ lies in the subset $\mathcal{U}$ of $H^{1,1}(\RR_+)$ defined in Definition \ref{opensetUdef}. The stationary one-soliton initial data $u_{s0}$ clearly lies in $\mathcal{U}$. Moreover, by Lemma \ref{UmuBlemma}, $\mathcal{U}$ is open in $H^{1,1}(\RR_+)$. Hence there exists a neighborhood $U$ of $u_{s0}$ such that $U \subset \mathcal{U}$. It follows that if $u_0 \in U$, then the global weak solution $u(x,t)$ with parameter $q$ and initial data $u_0$ satisfies \eqref{uqpositiveasymptotics} as $t \to \infty$. 
On the other hand, from Proposition \ref{rDeltaprop}, the spectral function $\Delta(k)$ associated to $u_0$ has exactly one simple zero in $\CC_+$. By Lemma \ref{zeroscontinuitylemma}, if we let $\xi_1$ and $\xi_{s1}$ be the zeros of $u_0$ and $u_{s0}$, respectively, then $|\xi_1-\xi_{s1}| \to 0$ as $|u_0-u_{s0}|_{H^{1,1}(\RR_+)} \to 0$. This completes the proof of the Theorem \ref{asymptoticsthdefocusing}.

\subsection{Proof of Theorem \ref{asymptoticsthfocusing}}\label{asymptoticsthfocusingproof}
Suppose $\lambda = -1$ and let $q = \sqrt{\omega} \tanh{\phi} \in \RR \setminus \{0\}$ be the value of $q$ associated with the stationary one-soliton. 

Suppose first that $q > 0$. In this case, the spectral functions $\Delta_s$ and $a_s$ corresponding to the stationary one-soliton initial data $u_{s0}$ fulfill Assumption \ref{zerosassumption}, because $\Delta_s$ has only one simple zero in $\bar{\CC}_+$ at $\xi_{s1} := i\sqrt{\omega}/2$, and $a_s$ has no zeros in $\bar{\CC}_+$ (see (\ref{bsas}) and (\ref{Deltasformula})). This means that $u_{s0} \in \mathcal{U}$ where $\mathcal{U}$ is the subset of $H^{1,1}(\RR_+)$ defined in Definition \ref{opensetUdef}. By Lemma \ref{UmuBlemma}, $\mathcal{U}$ is open, and hence the desired conclusions follow immediately  from Proposition \ref{rDeltaprop}, Lemma \ref{zeroscontinuitylemma}, and Theorem \ref{focusingasymptoticsth}.

Suppose now that $q < 0$. In this case, $\Delta_s$ has two simple zeros in $\CC_+$ located at $\xi_{s1}:= i\sqrt{\omega}/2$ and $\xi_{s2} := -iq/2$. Since $a_s(k)$ also has a simple zero at $\xi_{s2}$, Assumption \ref{zerosassumption} is not fulfilled. This means that we cannot immediately apply Theorem \ref{focusingasymptoticsth}. However, as we now describe, it is easy to generalize the statement of Theorem \ref{focusingasymptoticsth} to include also the present situation.

Let $u_0(x)$ be a small perturbation of $u_{s0}$ in $H^{1,1}(\RR_+)$. Let $a(k)$ and $\Delta(k)$ be the spectral functions corresponding to $u_0$. By Lemma \ref{zeroscontinuitylemma}, $\Delta(k)$ has two simple zeros $\xi_1$ and $\xi_2$ and these zeros are such that $\xi_1 \to \xi_{s1}$ and $\xi_2 \to \xi_{s2}$ as $u_0 \to u_{s0}$.
In particular, $\xi_1 \neq \xi_2$ and $\xi_1 \neq \xi_{s2}$ for $u_0$ sufficiently close to $u_{s0}$. The symmetry (\ref{Deltasymm}) implies that both $\xi_1$ and $\xi_2$ lie in $i\RR_+$. 
Furthermore, by Lemma \ref{u0toabLinftylemma} the map $u_0 \mapsto a$ is continuous $H^{1,1}(\RR_+) \to L^\infty(\bar{\CC}_+)$, so the same arguments used to prove Lemma \ref{zeroscontinuitylemma} show that, whenever $u_0$ is close to $u_{s0}$, $a(k)$ has exactly one simple zero $k_1$ in $\CC_+$ and this zero satisfies $k_1 \to \xi_{s2}$ as $u_0 \to u_{s0}$.
We distinguish two cases depending on whether $\xi_2$ coincides with $\xi_{s2}$ or not.

Case 1 ($\xi_2 = \xi_{s2}$). Suppose $\xi_2$ coincides with $\xi_{s2}$. Equation (\ref{Deltadef}) evaluated at $k = \xi_{s2} = -iq/2$ reads
\begin{align}\label{Deltaatxis2}
\Delta(\xi_{s2}) = -2iq|a(\xi_{s2})|^2.
\end{align}
Since $\Delta(k)$ vanishes at $\xi_{s2}$, we deduce that $a(k)$ also vanishes at $\xi_{s2}$. This means that both $\Delta$ and $a$ have simple zeros at $\xi_{s2}$.
In particular, Assumption \ref{zerosassumption} is not fulfilled. However, we claim that the conclusion of  Theorem \ref{linearizableth} still holds provided that the zero $\xi_2$ of $\Delta$ is excluded from the list of poles of $m$. Indeed, consider the derivation of the residue conditions (\ref{mresidues}) in Section \ref{commonzeroissuesection}. As in Section \ref{commonzeroissuesection}, $\xi_1$ is in the present situation a simple zero of $\Delta(k)$ and $a(\xi_1) \neq 0$. This means that $m$ satisfies the residue conditions (\ref{mresidues}) for $j = 1$. However, $\xi_2 = k_1 = \xi_{s2} = -iq/2$ is now a simple zero of both $\Delta$ and $a$. It therefore follows from (\ref{m1equation}) that $[m]_1$ is analytic at $\xi_2$. By symmetry, $[m]_2$ is then analytic at $\bar{\xi}_2$. The upshot is that $m$ satisfies the RH problem of Theorem \ref{focusingasymptoticsth} with the slight modification that the residue conditions (\ref{mresidues}) should only be imposed for $j = 1$ and not for $j = 2$.
In other words, the RH problem involves the modified discrete scattering data $\hat{\sigma}_d=\{(\xi_1,\hat{c}_1)\}$, where $\hat{c}_1$ is defined in \eqref{modifieddataresidue}. We conclude that the long-time asymptotics formula \eqref{ufocusingasymptotics} holds as $t \to \infty$ with the stationary one-soliton solution $u_{\mathrm{sol}}(x,t;\hat{\sigma}_d)$ defined by \eqref{recoverusol} and $u_{\mathrm{rad}}(x,t)$ defined by equation \eqref{uradformula} with the modified scattering data $\hat{\sigma}_d$ and the reflection coefficient $r(k)$ given by \eqref{rdef}.
This completes the proof in the case when $\xi_2 = \xi_{s2}$.

Case 2 ($\xi_2 \neq \xi_{s2}$). Suppose $\xi_2$ does not coincide with $\xi_{s2}$. In this case, we have $a(\xi_2) \neq 0$. Indeed, evaluating (\ref{Deltadef}) at $\xi_2$, we obtain
\begin{align}\label{Deltadefatxi2}
0 = (2\xi_2-iq)|a(\xi_2)|^2 + \lambda(2\xi_2+iq)|b(\xi_2)|^2.
\end{align}
From the explicit expression (\ref{bsas}) for the spectral function $b_s(k)$ corresponding to $u_{s0}$, we see that $b_s$ is nonzero everywhere in the upper half-plane. Furthermore, by Lemma \ref{u0toabLinftylemma}, the map $u_0 \mapsto b$ is continuous $H^{1,1}(\RR_+) \to L^\infty(\bar{\CC}_+)$. Hence, $b(\xi_2) \neq 0$ for all $u_0$ close to $u_{s0}$.
Since $\xi_2 \neq -iq/2$ by assumption, we have $2\xi_2 \pm iq \neq 0$, and hence (\ref{Deltadefatxi2}) implies that $a(\xi_2) \neq 0$. We conclude that, whenever $u_0$ is sufficiently close to $u_{s0}$, Assumption \ref{zerosassumption} is fulfilled and the RH solution $m$ of Theorem \ref{linearizableth} has two simple poles $\xi_1, \xi_2$ in $\CC_+$. This means that the RH problem involves the modified discrete scattering data $\hat{\sigma}_d=\{(\xi_1,\hat{c}_1), (\xi_2,\hat{c}_2)\}$, and that we have the long-time asymptotics formula \eqref{ufocusingasymptotics} as $t \to \infty$ with the stationary two-soliton solution $u_{\mathrm{sol}}(x,t;\hat{\sigma}_d)$ defined by \eqref{recoverusol} and $u_{\mathrm{rad}}(x,t)$ defined by \eqref{uradformula}. This completes the proof of the theorem.


\appendix


\section{Global well-posedness of the NLS equation}\label{globalwpn}
In this appendix, we prove Proposition \ref{H11existenceprop}, i.e., we establish global well-posedness of the Robin IBVP for the NLS equation in the space $H^{1,1}(\RR_+)$. The proof is inspired by \cite{DP2011} and we refer to Theorem 3 in \cite{DP2011} for more details of the proofs of existence and uniqueness in the focusing case. 

Let $H_q^+$ be the operator introduced in Definition \ref{H11globalsoln}.
Let $U_t \equiv e^{-i H_q^+ t/2}$ be the one-parameter unitary operator generated by $H_q^+$ in $L^2(\RR_+)$. Let $X=C([0,T],H^{1,1}(\RR_+))$, $0<T<\infty$, denote the set of continuous maps from $[0,T]$ to $H^{1,1}(\RR_+)$. Equip $X$ with a norm $\|v\|_{X}\equiv \sup_{0 \le t \le T}\|v(t)\|_{H^{1,1}(\RR_+)}$. For $v=v(t) \in X$, define $T_qv$ by
$$(T_q v)(t)=U_t v(0)-i \lambda \int_0^t f(s,v(s)) ds,\quad 0 \le t \le T,$$
where $f(s,v(s))=U_{t-s}|v(s)|^2 v(s)$. Note that $T_q$ is a contraction on $X$, since $f$ is Lipschitz continuous on $X$: for $v, w \in X$ we have the estimate
$$\|f(s,v(s))-f(s,w(s))\|_{H^{1,1}(\RR_+)} \le K \|v(s)-w(s)\|_{H^{1,1}(\RR_+)},$$
where $K=C(T)(\|v\|_{H^{1,0}}^2+\|w\|_{H^{1,0}}^2)$. Hence, there is a unique fixed point $v \in X$ such that $T_q v=v$ and $v(t)$ is a unique weak solution for the Robin IBVP of the NLS equation for $t \in [0,T]$ with small $T$. Furthermore, the following two quantities are constant in time:
\begin{align*}
\mathcal{M}=&\int_0^\infty |u|^2 dx',\\
\mathcal{H}=&\int_0^\infty (|u_x|^2 + \lambda |u|^4) dx - q |u(0,t)|^2.
\end{align*}
These conserved quantities provide an a priori bound on the $H^{1,0}$-norm of $u(t)$, which implies that the Lipchitz constant is bounded. Hence, the contraction mapping can be extended indefinitely by iteration and so the solution of the NLS equation exists globally.

Finally, we show that the data-to-solution mapping is continuous. Suppose that $v, w$ are fixed points of $T_q$ with initial data $v_0, w_0 \in H^{1,1}(\RR_+)$, respectively. Then
\begin{align*}
\|v(t)-w(t)\|_{H^{1,1}(\RR_+)}=&\|T_q v(t)-T_q w(t)\|_{H^{1,1}(\RR_+)}\\
\le &\|U_t (v_0-w_0\|_{H^{1,1}(\RR_+)}+\int_0^t \|f(s,v(s))-f(s,w(s))\|_{H^{1,1}(\RR_+)} ds\\
\le &C \|v_0-w_0\|_{H^{1,1}(\RR_+)}+K\int_0^t \|v(s)-w(s)\|_{H^{1,1}(\RR_+)}ds
\end{align*}
Hence, by the Gr\"{o}nwall's lemma, we have
$$\|v(t)-w(t)\|_{H^{1,1}(\RR_+)} \le C\|v_0-w_0\|_{H^{1,1}(\RR_+)}e^{K t}.$$
Thus, the initial data-to-solution mapping is continuous.

\section{Soliton solutions}\label{RHsolitonsec}

Theorem \ref{linearizableth} provides a representation of the solution $u(x,t)$ of the Robin IBVP for the NLS equation in terms of the RH problem \eqref{mjump} for $m(x,t,k)$. For soliton solutions,  the reflection coefficient $r(k)$ vanishes identically and the jump matrix of the RH problem is trivial. This implies that we can explicitly determine the solution $m_{\mathrm{sol}}(x,t,k)$ of the RH problem \eqref{mjump} corresponding to the $M$-soliton solution $u_{\mathrm{sol}}(x,t)$.

\begin{RHproblem}[RH problem for solitons]\label{RHsoliton}
Let $\lambda = 1$ or $\lambda = -1$. Given discrete data $\sigma_d=\{(\xi_j,c_j)\}_1^M \subset \CC_+ \times (\CC \setminus \{0\})$, find a $2 \times 2$-matrix valued function $m_{\mathrm{sol}}(x,t,k)$ with the following properties:
\begin{enumerate}[$(a)$]
\item $m_{\mathrm{sol}}(x,t,\cdot): \mathbb{C} \setminus \{\xi_j, \bar{\xi}_j\}_1^M \to \mathbb{C}^{2 \times 2}$ is analytic.
\item $m_{\mathrm{sol}}(x,t,k) \to I$ as $k \to \infty$.
\item For $j = 1, \dots, M$, the first column of $m_{\mathrm{sol}}$ has a simple pole at $\xi_j$, the second column of $m_{\mathrm{sol}}$ has a simple pole at $\bar{\xi}_j$, and the following residue conditions hold:
\begin{subequations}\label{msresidues}
\begin{align}\label{msresiduesa}
& \underset{k = \xi_j}{\Res} [m_{\mathrm{sol}}(x,t,k)]_1 =  c_j e^{2i\theta(x,t,\xi_j)} [m_{\mathrm{sol}}(x,t,\xi_j)]_2,
	\\\label{msresiduesb}
& \underset{k=\bar{\xi}_j}{\Res} [m_{\mathrm{sol}}(x,t,k)]_2 = \lambda\bar{c}_j e^{-2i\theta(x,t,\bar{\xi}_j)}  [m_{\mathrm{sol}}(x,t,\bar{\xi}_j)]_1.
\end{align}
\end{subequations}
\end{enumerate}
\end{RHproblem}

If the RH problem \ref{RHsoliton} has a solution, then
\begin{align}\label{recoverusol}
u_{\mathrm{sol}}(x,t;\sigma_d) = 2i\lim_{k \to \infty} k(m_{\mathrm{sol}}(x,t,k))_{12}, \qquad x \geq 0, ~ t \geq 0,
\end{align}
is the NLS $M$-soliton corresponding to the discrete scattering data $\sigma_d$. Note that the data $\sigma_d=\{(\xi_j,c_j)\}_1^M \subset \CC_+ \times (\CC \setminus \{0\})$ determines $m_{\mathrm{sol}}$ and $u_{\mathrm{sol}}$ completely. 
If $\lambda = -1$ and all the $\xi_j$ are distinct, then there exists a unique solution of RH problem \ref{RHsoliton} for each $(x,t) \in \RR^2$; see \cite[Proposition B.1.]{BJM2018} for a proof. If $\lambda = 1$, this is not the case (this is related to the fact that the defocusing NLS does not support bright solitons on the line). However, in the case of the stationary one-soliton (\ref{usdef}) which corresponds to the case of $M=1$ and $\xi_1$ pure imaginary, the RH problem \ref{RHsoliton} does have a unique solution $m_\mathrm{sol}$ whenever $x,t\geq 0$. 
In what follows, we construct this solution explicitly and compute the associated spectral functions $a_s(k)$, $b_s(k)$, and $\Delta_s(k)$ for both $\lambda = -1$ and $\lambda = 1$.

\subsection{The stationary one-soliton}
Let $m_s$ be the solution of RH problem \ref{RHsoliton} with the discrete scattering data $\sigma_d$ given by $\sigma_d = \{(\xi_1,c_1)\}$ for some $\xi_1 \in i\RR_{>0}$. 
Since $m_s \to I$ as $k \to \infty$, (\ref{msresiduesa}) implies that
\begin{equation}\label{expansion1}
[m_s(x,t,k)]_{1}=\begin{pmatrix}
1\\
0
\end{pmatrix}+\frac{c_1 e^{2i\theta(x,t,\xi_1)}}{k-\xi_1}[m_s(x,t,\xi_1)]_{2}.
\end{equation}
From uniqueness of solution of the RH problem and the symmetry condition \eqref{msymm}, we infer that $m_s$ satisfies the symmetries $m_{s,11} = m_{s,22}^*$ and $m_{s,21} = \lambda m_{s,12}^*$. Evaluating (\ref{expansion1}) at $k=\bar{\xi}_1$ and using these symmetries, we find the algebraic system
$$\begin{cases}
\overline{m_{s,22}(\xi_1)}=1+\frac{c_1 e^{2i\theta(\xi_1)}}{\bar{\xi}_1-\xi_1}m_{s,12}(\xi_1), \\
\lambda \overline{m_{s,12}(\xi_1)}=\frac{c_1 e^{2i\theta(\xi_1)}}{\bar{\xi}_1-\xi_1}m_{s,22}(\xi_1).
\end{cases}
$$
Solving for $m_{s,12}$ and $m_{s,22}$, we deduce that
\begin{align*}
m_{s,22}(\xi_1)=\frac{|\bar{\xi}_1-\xi_1|^2}{|\bar{\xi}_1-\xi_1|^2-\lambda |c_1|^2|e^{2i\theta(\xi_1)}|^2},\;\;\;\;m_{s,12}(\xi_1)=\frac{\lambda \overline{c_1} e^{-2i\theta(\bar{\xi}_1)}(\bar{\xi}_1-\xi_1)}{|\bar{\xi}_1-\xi_1|^2-\lambda |c_1|^2|e^{2i\theta(\xi_1)}|^2},
\end{align*}
and substituting these expressions back into \eqref{expansion1}, we obtain
\begin{align*}
m_{s,11}(k)=&\; 1+\frac{\lambda |c_1|^2 |e^{2i\theta(\xi_1)}|^2(\bar{\xi}_1-\xi_1)}{(k-\xi_1)(|\bar{\xi}_1-\xi_1|^2-\lambda |c_1|^2|e^{2i\theta(\xi_1)}|^2)},
	\\
m_{s,21}(k)=&\; \frac{c_1 e^{2i\theta(\xi_1)}}{k-\xi_1}\frac{|\bar{\xi}_1-\xi_1|^2}{|\bar{\xi}_1-\xi_1|^2-\lambda |c_1|^2|e^{2i\theta(\xi_1)}|^2}.
\end{align*}
Using the symmetry and writing $\xi_1=i \rho_1$, we arrive at
\begin{align}\nonumber
m_{s,22}(k)=&\; 1+\frac{2 i \lambda \rho_1  |c_1|^2 e^{-4\rho_1 x}}{(k+i \rho_1)(4 \rho_1 ^2-\lambda |c_1|^2 e^{-4 \rho_1 x})},
	\\ \label{ms22ms12}
m_{s,12}(k)=&\; \frac{4 \lambda \rho_1^2 \overline{c_1} e^{-2 \rho_1 x}e^{4i\rho_1^2 t}}{(k+i\rho_1)(4\rho_1^2-\lambda |c_1|^2e^{-4\rho_1 x})}.
\end{align}
It follows that
\begin{equation}\label{solitonsol}
2i\lim_{k \to \infty} km_{s,12}(x,t,k)=\frac{8i\lambda\rho_1^2 \overline{c_1} e^{2 \rho_1 x+i4\rho_1^2 t}}{4\rho_1^2 e^{4 \rho_1 x}-\lambda |c_1|^2}.
\end{equation}
By choosing $\rho_1= \sqrt{\omega}/2$ for $\lambda=\pm 1$, and
$$c_1= \begin{cases}
- i \sqrt{\omega} e^{-\phi}, & \lambda=-1,
	\\
\frac{i\alpha \sqrt{\omega}}{\sqrt{\alpha^2+\omega}+\sqrt{\omega}}, & \lambda=1,
\end{cases}
$$
the right-hand side of \eqref{solitonsol} coincides with the expression for the stationary one-soliton given in (\ref{usdef}). The denominators in (\ref{ms22ms12}) involve the factor
$$4 \rho_1 ^2-\lambda |c_1|^2 e^{-4 \rho_1 x} = \begin{cases} 
\omega(1 + e^{-2(\phi + x\sqrt{\omega})}), & \lambda = -1, 
\\
\omega  \Big(1 - \frac{\alpha^2 e^{-2 x \sqrt{\omega }}}{(\sqrt{\alpha ^2+\omega}+\sqrt{\omega })^2}\Big), & \lambda = 1.
\end{cases}$$
We conclude that $m_s$ is well-defined for all $x, t \in \RR$ if $\lambda = -1$, but is singular at the negative value of $x$ given in (\ref{singularxvalue}) if $\lambda = 1$. 
Since
$$\begin{pmatrix} b_s(k) \\ a_s(k) \end{pmatrix} = [\mu_3(0,0,k)]_2 = [m_s(0,0,k)]_2,$$
we find that the spectral functions $a_s$ and $b_s$ are given by 
\begin{equation}\label{bsas}
\begin{cases}
b_s(k)= -\frac{i \sqrt{\omega -q^2}}{2k+i\sqrt{\omega}},\\
a_s(k)= \frac{2k+i q}{2k+i\sqrt{\omega}},
\end{cases}\; \text{if $\lambda=-1$}, \quad \text{and} \quad
\begin{cases}
b_s(k)=-\frac{i\sqrt{q^2- \omega}}{2k+i \sqrt{\omega}},\\
a_s(k)=\frac{2k+iq}{2k+i \sqrt{\omega}},
\end{cases}\; \text{if $\lambda=1$}.
\end{equation}
As expected, the reflection coefficient $r_s(k) \equiv 0$ in both cases. It also follows that the function $\Delta_s(k)$ defined in (\ref{Deltadef}) is given for $\lambda = \pm 1$ by
\begin{equation}\label{Deltasformula}
\Delta_s(k)=\frac{(k-i\rho_1)(2k+iq)}{k+i\rho_1}.
\end{equation}
We observe that $\Delta_s$ has a simple zero at $\xi_1 = i\rho_1$ and that $a_s(\xi_1) \ne 0$.

\subsection{Renormalization of the reflectionless RH problem}
Following \cite{BJM2018}, we introduce a renormalization of the solution $m_{\mathrm{sol}}$ of the RH problem \ref{RHsoliton} that satisfies a modified discrete RH problem. Given discrete scattering data $\sigma_d$ and a subset $\Box \subseteq \{1,2,\ldots,M\}$, define
\begin{equation}
m^\Box(x,t,k) \coloneqq m_{\mathrm{sol}}(x,t,k)a_\Box (k)^{\sigma_3}, \qquad a_\Box (k) \coloneqq \prod_{j \in \Box}\frac{k-\xi_j}{k-\bar{\xi}_j}.
\end{equation}
Then $m^\Box$ satisfies the following RH problem with modified scattering data.
\begin{RHproblem}[renormalized reflectionless RH problem]\label{RHsolitonrenorm}
Given discrete scattering data $\sigma_d=\{(\xi_j,c_j)\}_1^M \subset \CC_+ \times (\CC \setminus \{0\})$ and $\Box \subseteq \{1,\ldots,M\}$ find a $2 \times 2$-matrix valued function $m^\Box(x,t,k)$ with the following properties:
\begin{enumerate}[$(a)$]
\item $m^\Box(x,t,\cdot): \mathbb{C} \setminus \{\xi_j, \bar{\xi}_j\}_1^M \to \mathbb{C}^{2 \times 2}$ is analytic.
\item $m^\Box(x,t,k) \to I$ as $k \to \infty$.

\item Each point of $\{\xi_j, \bar{\xi}_j\}_1^M$ is a simple pole of $m^\Box(x,t,k)$ and the following residue conditions hold:
\begin{align*}
& \underset{k = \xi_j}{\Res} [m^\Box(x,t,k)]_1 =  \lim_{k \to \xi_j} m^\Box(x,t,k) n^\Box_j, 
\quad \underset{k=\bar{\xi}_j}{\Res} [m^\Box(x,t,k)]_2 = \lim_{k \to \bar{\xi}_j}m^\Box(x,t,k)\sigma_2 (\overline{n^\Box_j})\sigma_2,
\end{align*}
where
\begin{equation}
n^\Box_j=\begin{cases}
\begin{pmatrix}0 & 0 \\ c_j e^{2i\theta(x,t,\xi_j)}a_\Box(\xi_j)^2 & 0 \end{pmatrix} \quad & j \in \{1,2,\ldots,M\}\setminus \Box, \\
 \begin{pmatrix}0 & \frac{1}{c_j e^{2i\theta(x,t,\xi_j)}} a'_\Box(\xi_j)^{-2}\\ 0 & 0 \end{pmatrix} \quad & j \in \Box.
\end{cases}
\end{equation}
\end{enumerate}
\end{RHproblem}

\bibliographystyle{plain}

\end{document}